\documentclass[11pt]{amsart}

\usepackage{tikz}
\usetikzlibrary{positioning}
\usepackage{amsmath}
\usepackage{amsthm}
\usepackage{amsxtra}
\usepackage{graphicx}
\usepackage{amsfonts,amssymb}
\usepackage{latexsym}
\usepackage{bbm}

\theoremstyle{definition}

\newtheorem{theorem}{Theorem}[section]
\newtheorem{proposition}[theorem]{Proposition}
\newtheorem{lemma}[theorem]{Lemma}
\newtheorem{corollary}[theorem]{Corollary}

\theoremstyle{definition}
\newtheorem{definition}[theorem]{Definition}
\theoremstyle{remark}
\newtheorem{remark}[theorem]{Remark}
\newtheorem{example}[theorem]{Example}

\newcommand{\supp}{\mathrm{supp}}

\newcommand{\tr}{\mathrm{tr}}

\newcommand{\fix}{\mathrm{Fix}}
\newcommand{\sym}{\mathrm{Sym}}
\newcommand{\alt}{\mathrm{Alt}}

\newcommand{\id}{\mathrm{Id}}

\newcommand{\ie}{\text{i.e.\;\,}}

\numberwithin{equation}{section}

\title[Characters of Full Groups]{Characters and $\textrm{II}_1$-Factor Representations of Full Groups of Cantor Minimal Systems}

\author{Artem Dudko}
\address{Institute of Mathematics of Polish Academy of Sciences, Warsaw, Poland}
\email{adudko@impan.pl}

\author{Constantine Medynets}
\address{United States Naval Academy, Annapolis, MD, USA}
\email{medynets@usna.edu}



\date{}

\begin{document}

\begin{abstract} 
Let $(X,T)$ be a Cantor minimal system, and let $\Gamma$ denote either its associated topological full group or the full group of a Bratteli diagram associated with $(X,T)$. In this paper we describe the structure of indecomposable (extreme) characters and the associated $\textrm{II}_1$-factor representations for the group $\Gamma$ and its commutator subgroup $\Gamma'$. In particular, we prove that:  
(1) for every nontrivial indecomposable character $\chi$ of $\Gamma'$, there exists a finite collection (with repetitions allowed) $\{\mu_i\}_{i\in I}$ of $T$-invariant ergodic measures on $X$ such that  
\[
\chi(\gamma) = \prod_{i\in I} \mu_i(\fix(\gamma)),
\quad \text{for every } \gamma \in \Gamma', 
\]
where $\fix(\gamma) = \{x\in X : \gamma x = x\}$; and  
(2) each indecomposable character of $\Gamma$ is the product of an indecomposable character of the form  
$\prod_{i\in I} \mu_i(\fix(\gamma))$ and a homomorphism from $\Gamma$ into the unit circle.  

As a consequence, we show that any finite-type unitary representation of $\Gamma'$ that does not contain a regular subrepresentation is automatically continuous with respect to the uniform topology on $\Gamma'$.  

We also establish a general result on automatic continuity of finite-type unitary representations of infinite groups, which we use in our proofs.
\end{abstract}

\maketitle

\tableofcontents 

\section{Introduction}

The classification of characters of a group has been a central theme in representation theory for over a century. For infinite discrete groups, a \emph{character} on a group $G$ is defined as a conjugation-invariant function $\chi : G \to \mathbb{C}$ of positive type, normalized such that $\chi(e)=1$.

The systematic study of characters on infinite discrete groups dates back to the seminal work of Elmar Thoma in 1964 \cite{Thoma-Characters-64, Thoma:MathZ-84, Thoma:MathAnn-153}. In these works, he established general topological and functional-analytical properties of the set of characters, clarified the connection with unitary representations, and obtained a complete description of the characters of the infinite symmetric group $\mathrm{S}(\infty)$.

If $W$ is a finite von Neumann algebra with a faithful normal trace $\tau$, and $\pi : G \to U(W)$ is a unitary representation, then the function $g \mapsto \tau(\pi(g))$ defines a character on $G$. Moreover, the Gelfand-Naimark-Segal (GNS) construction shows that every character arises in this way. The set of group characters forms a Choquet simplex \cite{Thoma-Characters-64}. Extreme (indecomposable) characters correspond to finite factor representations of $G$. The classification of extreme characters of $G$ is equivalent to the classification of finite factor representations of $G$ up to quasi-equivalence; see Section~\ref{Section:PreliminariesGNS} for details.

In \cite{Thoma:MathAnn-153}, Thoma showed that a group $G$ is of type~I (meaning every unitary representation of $G$ generates a von Neumann algebra of type~I) if and only if $G$ is virtually Abelian. By \cite{Glimm1961TypeIC}, the classification of the unitary dual for non-type~I groups is an intractable problem in the sense of Borel reducibility theory (see also the discussion in \cite{tonti2019shortproofthomastheorem}). Thus, the study of the set of characters, or, equivalently, finite factor representations, serves as a natural substitute for the full unitary dual for general groups. The interested reader may refer to the book \cite{BekkadeLaHarpeBook:2020} for a detailed discussion of various unitary duals for discrete groups. 

The set of group characters captures the lattice of normal subgroups of the group in question. If $N$ is a normal subgroup of $G$, then the characteristic function of $N$ is a character on $G$. More generally, if $\nu$ is a conjugation-invariant probability measure on the space of subgroups of $G$, also called an invariant random subgroup (IRS), then the functions
\begin{equation}\label{eqnCharactersIRS}
\chi'(g) = \nu(\{H < G : g\in H\}) \quad \text{and} \quad \chi''(g) = \nu(\{H < G : gHg^{-1} = H\})
\end{equation}
define (possibly distinct) characters on $G$. We refer the reader to \cite{dudko2023charactersirssbranchgroups} for a discussion of the relationship between IRSs and group characters.  

In \cite{AbertGlasnerVirag:2014}, the authors showed that for every IRS $\nu$ of $G$ one can find a probability measure-preserving action of $G$ on $(X, \mu)$ such that $\nu$ can be realized as the pushforward of the measure $\mu$ under the stabilizer map $x \mapsto \mathrm{St}_G(x)$, where $\mathrm{St}_G(x) = \{g\in G : g(x) = x\}$. This implies that  $\chi'(g) = \mu(\fix(g))$ for every $g\in G$, where $\fix(g)= \{x\in X : g(x) =x\}$. Analogously, the character $\chi''(g)$ is given by the measure of the fixed-point set under the conjugation action on the space of subgroups.

The connection between characters and measure-preserving actions goes back to the work of Vershik and Kerov in the early 1980s \cite{VershikKerov:1981, VershikKerov:1982}. They introduced what they called the \emph{ergodic method} for studying characters of locally finite groups, in which characters are approximated by those of finite subgroups. Analyzing the convergence of these approximations led to a new proof of Thoma’s classification of characters of $\mathrm{S}(\infty)$. Vershik and Kerov further showed that a large class of characters of $\mathrm{S}(\infty)$ can be realized via measure-preserving actions of $\mathrm{S}(\infty)$ on probability spaces $(X,\mu)$ through the formula
\[
\chi_\mu(g) = \mu(\fix(g)).
\]
We refer the reader to \cite{Vershik:NonFreeActions2012} and \cite{Thomas:2022} for a precise description of the relationship between characters of $\mathrm{S}(\infty)$ and its invariant random subgroups. In \cite{Vershik-Nonfree-10}, Vershik suggested that a similarly strong connection between characters and measure-preserving actions should hold for a broader class of locally finite groups. From this perspective, the results of the present paper may be viewed as establishing  Vershik's conjecture for the class of topological full groups of Cantor minimal systems.

We note several papers where characters are completely classified for specific groups: linear groups \cite{Bekka2007, Kirillov1965, Ovchinnikov1971, Skudlarek1976, LeinenPuglisi:2007, PetersonThom2016}, algebraic groups \cite{BekkaFrancini:2022}, inductive limits of symmetric groups \cite{Thoma-Characters-64, DudkoMedynets-CharactersSymmetric-13, LeinenPuglisi:2004, VershikKerov:1982, GoryachkoPetrov-Rearrangements-10,NessonovNgo:2025}, permutational wreath products \cite{DudkoNessonov-CharsWreath-07, DudkoNessonov-CharsProjective-08}, and Higman-Thompson groups \cite{DudkoMedynets2014, GardellaTanner:2024}. 

We conclude by mentioning that group characters also play an important role in the study of two-sided ideals of group algebras \cite{Zalesskii:1995}, local Hilbert-Schmidt stability \cite{Fournier_Gerasimova_Spaas:2025}, and invariant subalgebras of von Neumann algebras \cite{Jiang_Zhou:2024,Dudko_Jiang:2024,Amrutam_Dudko_Skalski:2025}.

%
%

\subsection{Statements of the Main Results}

The main result of the present paper is a complete description of characters on full groups and AF full groups associated with Cantor minimal systems.

By a \emph{Cantor minimal system}, we mean a pair $(X,T)$ where $X$ is the Cantor set and $T$ is a homeomorphism of $X$ such that every orbit of $T$ is dense; equivalently, $T$ admits no nontrivial closed invariant subsets. With such a pair, we associate the group $\mathcal{F}(\langle T \rangle)$ of homeomorphisms of $X$ that locally coincide with powers of $T$. The group $\mathcal{F}(\langle T \rangle)$ is called the (topological) \emph{full group} of $(X,T)$.

Let $(X,T)$ be a Cantor minimal system and $x_0 \in X$. We consider the subgroup $\mathcal{F}(\langle T \rangle)_{x_0}$, dubbed \emph{approximately finite (AF) full group}, of $\mathcal F(\langle T \rangle)$ consisting of homeomorphisms preserving the forward $T$-orbit of $x_0$. This group is locally finite and independent of the choice of $x_0$ (see Section~\ref{Section:FullGroups}). We note that AF full groups can be represented diagrammatically using infinite graded graphs called Bratteli diagrams (see Section~\ref{Section:FullGroups}) and include a large class of inductive limits of symmetric groups such as those discussed in \cite{LeinenPuglisi:2004,GoryachkoPetrov-Rearrangements-10}.

Full groups first appeared in \cite{Putnam:1989} in the context of crossed product $C^*$-algebras $C(X) \rtimes \mathbb{Z}$. The study of full groups has since been a central topic in topological dynamics and operator algebras \cite{GlasnerWeiss-Equivalence-95, GiordanoPutnamSkau:1999, Matui:2006, Medynets:2011, GrigorchukMedynets-AlgebraicFull-14, GrigorchukMedynets-PresentationFull-18}; see \cite{katzlinger2019topological} for a survey of the  development of the theory. Notably, full groups of Cantor minimal systems provided the first examples of simple, finitely generated, infinite amenable groups \cite{Juschenko-Monod:2013}.

In \cite{Matui:2012, Matui-FullGroups-2015}, Matui introduced a broader perspective on full groups using the concept of essentially principal \'etale groupoids. This
framework extends beyond full groups arising from transformation groupoids
(group actions) and, for example, includes full groups associated with purely-infinite groupoids (non-invertible systems), such as the class of Higman-
Thompson groups. A detailed exposition of the groupoid perspective on full
groups is provided in the monograph  \cite{Nekrashevych:Book2022}. 

In view of the reconstruction theorems in \cite{GiordanoPutnamSkau:1999}, full groups and AF full groups completely determine the underlying system up to flip conjugacy and strong orbit equivalence, respectively. Thus, it is natural to expect that the algebraic structure of these groups encodes information about the dynamical properties of the system and vice versa. The following theorem is the main result of the paper; it demonstrates that characters are indeed parametrized by invariant measures. The proof is presented in Theorem \ref{theoremClassificationCharacters}.

\begin{theorem}\label{TheoremMainIntro} 
Let $(X, T)$ be a Cantor minimal system and $\Gamma$ be either the full group $\mathcal{F}(\langle T \rangle)$, the AF full group $\mathcal{F}(\langle T \rangle)_{x_0}$, or the commutator subgroup of either. If $\chi$ is an indecomposable character of $\Gamma$, then one of the following holds:
\begin{enumerate}
    \item[(i)] $\chi = 1$,
    \item[(ii)] $\chi$ is the regular character,
    \item[(iii)] $\chi$ is of the form 
    $$\chi(\gamma) = \rho(\gamma) \prod_{i=1}^k \mu_i(\fix(\gamma)) \quad \text{for } \gamma \in \Gamma,$$ 
    where $\mu_1, \dots, \mu_k$ are $T$-invariant ergodic measures (repetitions permitted) and $\rho: \Gamma \to S^1$ is a group homomorphism.
\end{enumerate}
If $\Gamma$ is the commutator subgroup of $\mathcal{F}(\langle T \rangle)$ or $\mathcal{F}(\langle T \rangle)_{x_0}$, then $\rho(\gamma) = 1$ for every $\gamma \in \Gamma$.

Additionally, any function $\chi$ as in (i), (ii), or (iii) is an indecomposable character on $\Gamma$.
\end{theorem}

In \cite{zheng2020rigid}, the author completely classified ergodic IRSs of the commutator subgroups of the topological full groups and AF full groups of Cantor minimal systems. Combining the results of \cite{zheng2020rigid} with Theorem~\ref{TheoremMainIntro}, we obtain that indecomposable characters of the commutator subgroups of $\mathcal{F}(\langle T \rangle)$ and $\mathcal{F}(\langle T \rangle)_{x_0}$ are in one-to-one correspondence with ergodic IRSs.

Given a Cantor minimal system $(X,T)$, we define the metric 
$$D(g,h) = \sup_{\mu \in \mathcal{M}(X,T)} \mu(\{x \in X : g(x) \neq h(x)\})$$
on the full group $\mathcal{F}(\langle T \rangle)$, where $\mathcal{M}(X,T)$ denotes the set of all $T$-invariant probability measures. We note that the topology generated by this metric is the analogue of the {\it uniform topology} studied in ergodic theory; see \cite{BezuglyiTopologiesOnCantorSet} for details.

As a corollary of our main theorem, we obtain the following result establishing the automatic continuity of finite-type unitary representations of commutator subgroups of full groups. This result is proven in Corollaries~\ref{corollaryCommutatorAutomaticCont} and \ref{corollaryLargeCommutatorAutomaticCont}. 

\begin{corollary}
Let $(X, T)$ be a Cantor minimal system, let $\Gamma$ be the commutator subgroup of either $\mathcal{F}(\langle T \rangle)$ or $\mathcal{F}(\langle T \rangle)_{x_0}$, and let $\pi$ be a finite-type unitary representation of $\Gamma$. Suppose that the central decomposition of $\pi$ contains no regular subrepresentation. Then for any Cauchy sequence $\{g_n\}$ in $\Gamma$ with respect to the metric $D$, the sequence $\{\pi(g_n)\}$ converges in the strong operator topology.
\end{corollary}

In \cite{DudkoMedynets-CharactersSymmetric-13}, the characters of $\mathcal{F}(\langle T \rangle)_{x_0}$ were classified under the assumptions that the system admits only finitely many ergodic measures and that the group is simple. In the present paper, we overcome these limitations and complete the classification of characters for both AF full group and the full group of all Cantor minimal systems. The details of our approach are outlined below.

First, we show that finite-type factor representations for a broad class of groups possess certain automatic continuity properties (Theorem~\ref{ThAutomaticContinuity}). We then apply this result in the proof of Theorem~\ref{TheoremCharFromTrace} to demonstrate that for any non-regular character $\chi$ on the commutator subgroup of the AF full group $\mathcal{F}(\langle T \rangle)_{x_0}$, one has:
\begin{equation}\label{eqnAuxCharTraceProj}
\chi(g) = \tr(P^{\mathrm{supp}(g)}),
\end{equation}
where $\tr$ is the trace in the corresponding GNS representation and $P^{\mathrm{supp}(g)}$ is the projection onto the subspace invariant under all elements whose support is contained in $\mathrm{supp}(g) = \{x\in X : g(x) \neq x\}$.

We note that minimal $\mathbb{Z}$-systems satisfy the comparison property \cite{GlasnerWeiss-Equivalence-95, Kerr:2022}, which implies that clopen sets having the same values for all invariant measures are ``almost'' equivalent under the action of the full group. Thus, by (\ref{eqnAuxCharTraceProj}),  character values depend, in a sense, only on the values of invariants measures of $\mathrm{supp}(g)$.

Let $\mathcal{K}$ be the closure of the $T$-invariant ergodic measures in the weak* topology. Each clopen set $A$ defines an evaluation function $m_A(\mu):=\mu(A)$ on $\mathcal{K}$. In our proof, we define $\phi(m_A) = \tr(P^{X\setminus A})$ and show that the function $\log \circ \phi\circ \exp$ extends to a continuous linear functional on $C(\mathcal K)$.  Applying the Riesz--Markov--Kakutani Representation Theorem, we show that:
$$\chi(g) = \tr(P^{X\setminus \fix(g)}) = \phi(m_{\fix(g)})  =  \exp\left(\int_{\mathcal{K}} \log(\mu(\fix(g))) \, d\nu(\mu)\right).$$
The final step consists of proving that $\nu$ is a purely atomic measure supported on finitely many points (which are necessarily ergodic measures) with integer weights, which yields the classification of characters for the commutator subgroups of AF full groups.

We then establish an automatic continuity result for finite-type representations of the commutator subgroups of full groups, which allows us to extend both the character classification and the automatic continuity properties to the commutator subgroup of the full group. The paper concludes with the classification of characters for the full groups themselves (Theorem~\ref{theoremClassificationCharacters}).

While several parts of our proof rely on structural results specific to Cantor minimal $\mathbb{Z}$-systems, other arguments are more general and apply, for instance, to any minimal system satisfying the comparison property. We conjecture that the main results of this paper should hold for full groups of minimal almost finite systems; see \cite{Suzuki:AlmostFiniteness, Kerr:2022} for a detailed study of their properties.

\subsection{Structure of the Paper}
The paper is organized as follows. In Section \ref{Section:PreliminariesGNS}, we present the necessary definitions regarding group characters, finite von Neumann algebras, and the GNS construction. We also introduce the Feldman-Moore groupoid construction, which provides a model for realizing permutation characters. In Section \ref{Section:FullGroups}, we survey the algebraic properties of topological full groups and AF full groups associated with Bratteli diagrams.

Section \ref{Section:PermutationCharacter} is devoted to the properties of permutation characters $\mu(\text{Fix}(g))$. We prove that products of these characters corresponding to ergodic measures are indecomposable and show that only integer powers of these characters result in positive-definite functions.

In Section \ref{Section:AutomaticContinuity}, we establish a general result on the automatic continuity of finite-type unitary representations of groups acting on the Cantor set.

Section \ref{SectionFullGroupsBratteliDiagrams} provides technical tools regarding the approximating properties of AF full groups, including a version of the Rokhlin Lemma and results on the asymptotic centrality of involutions.

Section \ref{Section:Projections} contains the construction of projections $\{P^A\}$ onto subspaces invariant under local subgroups. We use these to prove Theorem \ref{TheoremCharFromTrace}, establishing Formula (\ref{eqnAuxCharTraceProj}).

In Section \ref{Section:SpaceOfMeasures}, we study the space of invariant measures and prove that evaluation functions are uniformly dense in the space of continuous functions on the closure of ergodic measures, which allows for the extension of trace functionals.

In Section \ref{SectionCharactersCommutarorSub} we apply these results to classify indecomposable characters of commutator subgroups of AF full groups (Theorem \ref{thmClassificationCharactersCommutatorSubgroup}) and establish the automatic continuity of their finite-type representations (Corollary \ref{corollaryCommutatorAutomaticCont}). 

Finally, in Section \ref{Section:CharactersFullGroup}, we extend this classification to the commutator subgroups of topological full groups (Theorem \ref{thmClassificationCharactersLargeCommutatorSubgroup}) and establish the  automatic continuity of their finite-type representations (Corollary \ref{corollaryLargeCommutatorAutomaticCont}).  The main result of the paper is established in Theorem \ref{theoremClassificationCharacters}. 

\medskip

\noindent {\bf Acknowledgement.} The first-named author acknowledges support from the long-term program for Ukrainian research teams at the Polish Academy of Sciences, carried out in collaboration with the U.S. National Academy of Sciences and with financial support from external partners. A.D. was also partially supported by the National Science Centre, Poland, via Grant OPUS 21 (2021/41/B/ST1/00461), ``Holomorphic dynamics, fractals, thermodynamic formalism'', as well as by Simons Foundation grant (award no. SFI-MPS-T-Institutes-00010825). Additional support was provided by State Treasury funds through a task commissioned by the Minister of Science and Higher Education under the project ``Organization of the Simons Semesters at the Banach Center - New Energies in 2026-2028'' (agreement no. MNiSW/2025/DAP/491).  The second-named author thanks the Institute of Mathematics of the Polish Academy of Sciences for its hospitality and the USNA Kinnear Fellowship for its support. 

Both authors would like to thank David Kerr and Spyros
Petrakos for the discussions related to almost finite actions.





%
%
%

\section{Preliminaries}

 \subsection{Group Characters and Finite von Neumann Algebras}\label{Section:PreliminariesGNS}
 
 In this section, we define the notion of a group character and describe its connection with von Neumann algebras, which will serve as one of the main tools in this paper.

Classically, the character of a representation of a finite group is the function that assigns to each group element the trace of the corresponding matrix. The following definition extends this concept to a broader setting that includes infinite groups.

\begin{definition}\label{definitionCharacter}
A \emph{character} of a group $G$ is a function $\chi: G \to \mathbb{C}$ satisfying the following properties:
\begin{itemize}
\item[(1)] $\chi(g_1g_2) = \chi(g_2g_1)$ for all $g_1, g_2 \in G$;
\item[(2)] the matrix $\left\{\chi(g_i g_j^{-1})\right\}_{i,j=1}^n$ is positive semi-definite for any $n \geq 1$ and any choice of elements $g_1, \dots, g_n \in G$;
\item[(3)] $\chi(e) = 1$, where $e$ is the identity element of $G$.
\end{itemize}
\end{definition}

From this definition and basic properties of Hermitian matrices, it follows that $|\chi(g)| \leq 1$ and $\chi(g^{-1}) = \overline{\chi(g)}$ for all $g \in G$; see \cite[Appendix C]{BekkaHarpeValette-Kazdan-08}. A detailed treatment of group characters can be found in \cite{BekkadeLaHarpeBook:2020} and \cite[Appendix C]{BekkaHarpeValette-Kazdan-08}, which serve as our main references. We now explain the connection between group characters and von Neumann algebras.

Let $\mathcal{H}$ be a separable Hilbert space with inner product $(\cdot,\cdot)$, and let $B(\mathcal{H})$ denote the algebra of bounded linear operators on $\mathcal{H}$, equipped with the weak operator topology. For a subset $S \subset B(\mathcal{H})$, denote its \emph{commutant} by $S'$, i.e., the set of operators in $B(\mathcal{H})$ that commute with every element of $S$. A unital $*$-subalgebra $M \subset B(\mathcal{H})$ is called a \emph{von Neumann algebra} if it is closed in the weak operator topology or, equivalently, by the von Neumann Bicommutant Theorem, if $M = M''$.

A \emph{unitary representation} of a group $G$ is a homomorphism $\pi: G \to U(\mathcal{H})$, where $U(\mathcal{H})$ denotes the group of unitary operators on $\mathcal{H}$. For a representation $\pi$, let $\mathcal{M}_\pi$ denote the von Neumann algebra generated by $\pi(G)$ or, equivalently, $\mathcal{M}_\pi = \{\pi(G)\}''$ by the von Neumann Bicommutant Theorem. All groups and von Neumann algebras considered in this paper are assumed to be countable and separable, respectively.

A \emph{trace} on a von Neumann algebra $\mathcal{M}$ is a positive linear functional $\tau: \mathcal{M} \to \mathbb{C}$ such that $\tau(ab) = \tau(ba)$ for all $a,b \in \mathcal{M}$ and $\tau(\id) = 1$. The trace is called \emph{faithful} if $\tau(aa^*) = 0$ implies $a = 0$, and \emph{normal} if it is continuous on the unit ball of $\mathcal{M}$ in the weak operator topology.

A von Neumann algebra $\mathcal{M}$ is called \emph{finite} if every isometry in $\mathcal{M}$ is unitary; that is, if $v \in \mathcal{M}$ satisfies $v^*v = \id$, then also $vv^* = \id$. For separable von Neumann algebras, finiteness is equivalent to the existence of a faithful normal trace \cite[Remark 4.7]{Ioana:IntroVonNeumannAlgebras}. Von Neumann algebras admitting a faithful normal trace are also referred to as \emph{tracial}.

A von Neumann algebra $\mathcal{M}$ is called a \emph{factor} if $\mathcal{M} \cap \mathcal{M}' = \mathbb{C}\id$. We note that if $\mathcal M$ is a finite factor, then it admits a {\it unique} faithful normal trace \cite[Theorem 8.2.8]{KadisonRingrose:Vol2}. 

\begin{definition}
(1) A unitary representation $\pi$ of a group $G$ is said to be of \emph{finite type} if the von Neumann algebra $\mathcal{M}_\pi = \{\pi(G)\}''$ is finite.

(2)  A unitary representation $\pi$ of a group $G$ is called a \emph{factor representation or factorial} if $\mathcal{M}_\pi = \{\pi(G)\}''$ is a factor.  
\end{definition}  

If $\pi$ is a finite-type representation of $G$ and $\tau$ is a faithful normal trace on $\mathcal{M}_\pi$, then the function $\chi: G \to \mathbb{C}$ defined by $\chi(g) = \tau(\pi(g))$ is a character of $G$; see \cite[Example 11.B.3]{BekkadeLaHarpeBook:2020}.

Conversely, if $\chi: G \to \mathbb{C}$ is a character, then the Gelfand–Naimark–Segal (GNS) construction produces a triple $(\mathcal{H}, \xi, \pi)$,  where $\mathcal H$ is a Hilbert space, $\pi: G\rightarrow U(\mathcal H)$ is a unitary representation, and $\xi\in \mathcal H$ is a unit vector such that
\begin{enumerate} 
\item the vector $\xi$ is {\it cyclic}, in the sense that $\mathcal H$ is the smallest closed subset containing the span of $\{\pi(G)\xi\}$;

\item the vector $\xi$ is  {\it separating} for the von Neumann Algebra $\mathcal M_\pi$ generated by $\pi(G)$, that is, if $a\xi = 0$, for some $a\in \mathcal M_\pi$, then $a = 0$;

\item $\chi(g) = (\pi(g)\xi,\xi)$ for every $g\in G$,  where $(\cdot,\cdot)$ denotes the scalar product in $\mathcal H$;

\item the function $\tr(m) = (mx,x)$,  $m\in \mathcal M_\pi$, is a faithful normal trace on $\mathcal M_\pi$. 
\end{enumerate}

We refer the reader to \cite[Sections 1.B and 11.B]{BekkadeLaHarpeBook:2020} and \cite[Section 2.2]{DudkoMedynets-CharactersSymmetric-13} for further details on the GNS-construction for group representations. The utility of the GNS construction lies in the fact that it allows one to realize a character as a state on a von Neumann algebra, thereby enabling the use of analytic and topological methods in the study of discrete groups.

\begin{definition} (1) Two unitary representations $(\mathcal H_1, \pi_1)$ and $(\mathcal H_2, \pi_2)$  of a group $G$ are called {\it equivalent} if there exists a unitary operator $V$ from $\mathcal H_1$ onto $\mathcal H_2$ such that $\pi_2(g) =V\pi_1(g)V^*$ for every $g\in G$. 

(2) Two unitary representations $\pi_1$ and $\pi_2$  of a group $G$ are called {\it quasi-equivalent} if there is a $*$-isomorphism $\varphi $ between the von Neumann algebras  $\mathcal M_{\pi_1}$ and $\mathcal M_{\pi_1}$ such that $\varphi(\pi_1(g)) = \pi_2(g)$ for every $g\in G$. 
\end{definition}

We note that any (algebraic) $*$-isomorphism between von Neumann algebras is automatically continuous on (norm-) bounded sets with respect to the weak and strong operator topologies \cite[Proposition III.2.2.14]{Blackadar:BookOperatorAlgebras}.

Various reformulations  of quasi-equivalence can be found in Section 6.A and Proposition 6.B.5 in \cite{BekkadeLaHarpeBook:2020}.    We will need the following result from  \cite[Theorem 6.B.15]{BekkadeLaHarpeBook:2020}.  

\begin{proposition}\label{PropPropertiesQuasiEquivalence} Let $G$ be a countable group and $\pi$ be a factor representation of $G$.
(1) Any  representation of $G$ that is quasi-equivalent to $\pi$ is factorial. 
(2) Every non-trivial  sub-representation of $\pi$  is quasi-equivalent to $\pi$. 
\end{proposition}

The following result appears in \cite[Theorem A]{HiraiHirai:2005}, see also \cite[Propostion 11.B.7]{BekkadeLaHarpeBook:2020}. 

\begin{proposition}\label{PropGNSQuasiEquivalence} Let $G$ be a countable group and $\pi: G\rightarrow U(\mathcal H)$ be a factor representation of finite type. Let $\tr$ be a faithful normal trace on $\mathcal M_{\pi}$, and $
\chi = \tr\circ \pi$. Then the representation $\pi$ is quasi-equivalent to the GNS representation constructed by the character $\chi$. 
\end{proposition}

If $\chi_1$ and $\chi_2$ are characters on a group $G$, then so are their pointwise product and any convex combination. 

\begin{definition}
A character $\chi$ is called \emph{indecomposable} or {\it extremal} if it cannot be written as a nontrivial convex combination $\chi = \alpha \chi_1 + (1-\alpha)\chi_2$ for $0 < \alpha < 1$ and distinct characters $\chi_1, \chi_2$.
\end{definition}

The set of  characters forms a compact convex subset of the unit ball in $\ell^\infty(G)$ with respect to the weak*-topology. In particular, this set is a Choquet simplex \cite{Thoma-Characters-64}, and thus, by Choquet theory, every character can be uniquely represented as a barycenter (i.e., an integral) of extreme characters \cite[Section 10]{Phelps:LecturesChoquetTheory}.

The following classical result characterizes indecomposable characters in terms of their GNS representations; see \cite[Proposition 11.C.3]{BekkadeLaHarpeBook:2020}.


\begin{proposition}\label{PropEquivalenceExtremeCharactersFactors} Let $G$ be a countable group, $\chi$  a character on $G$, and $\pi$ the corresponding GNS representation. Then the  von Neumann algebra $\{\pi(G)\}''$ is a factor if and only if the character $\chi$ is indecomposable. 
\end{proposition}

\begin{remark}\label{remarkCharacterTimesHomomorsphism}  Let \(G\) be a group, let \(\chi\) be an indecomposable character of \(G\), 
and let \(\lambda : G \to \mathbb{T}\) be a group homomorphism into the unit circle. 
Then the function
$
\chi_{\lambda}(g) = \lambda(g)\,\chi(g)
$
is also an indecomposable character. 
Indeed, the von Neumann algebra associated with \(\chi_{\lambda}\) coincides with 
the von Neumann algebra associated with \(\chi\), which is a factor by 
Proposition~\ref{PropEquivalenceExtremeCharactersFactors}.
\end{remark}

Every finite-type unitary representation of a group admits a decomposition as a direct integral of finite-type factor representations, known as the \emph{central decomposition}. The following result describes this structure. A proof of Part (I), along with further references, can be found in \cite[Section 6.C.c; Theorems 6.C.7 and 6.C.8]{BekkadeLaHarpeBook:2020}. Part (II) is a consequence of Part (I). For a general treatment of disintegration theory in the context of von Neumann algebras, see \cite[Chapter 14]{KadisonRingrose:Vol2}.


\begin{theorem}\label{TheoremCentralDecomposition}
Let $G$ be a countable group, and let $\pi$ be a finite-type unitary representation of $G$ on a separable Hilbert space $\mathcal H$. Denote by $\mathcal Z$ the center of the von Neumann algebra generated by $\pi(G)$.

\smallskip

\noindent
\textup{(I)} There exist a standard Borel space $X$, a probability measure $\mu$ on $X$, a measurable field $x \mapsto (\mathcal H_x, \pi_x)$ of Hilbert spaces and representations, a Hilbert space isomorphism
$$U : \mathcal H \longrightarrow \int_X^\oplus \mathcal H_x\, d\mu(x),$$
and a measurable set $Y \subseteq X$ of full measure such that:
\begin{enumerate}
\item $U \pi(g) U^{-1} = \int_X^\oplus \pi_x(g) d\mu(x)$ for all $g \in G$;
\item $U \mathcal Z U^{-1}$ coincides with the algebra of diagonal operators on $\int_X^\oplus \mathcal H_x d\mu(x)$;
\item each $\pi_x$ is factorial for all $x \in Y$;
\item for distinct $x,y \in Y$, the representations $\pi_x$ and $\pi_y$ are disjoint (i.e., not quasi-equivalent).
\end{enumerate}

\smallskip

\noindent
\textup{(II)} If $\chi$ is a character of $G$ with associated GNS triple $(\mathcal H, \xi, \pi)$, then the central decomposition of $\pi$ from part~(I) satisfies
$$\chi(g) = ( \pi(g)\xi, \xi )
= \int_X^\oplus ( \pi_x(g)\xi_x, \xi_x )_x\, d\mu(x),$$
where $U\xi = \int_X^\oplus \xi_x, d\mu(x)$, and for $\mu$-almost every $x \in X$, the map
$m \longmapsto ( m \xi_x, \xi_x)_x$
defines a faithful normal trace on the von Neumann algebra generated by $\pi_x(G)$.
\end{theorem}

We note that every group admits at least two characters: the identity character and the regular character.

\begin{example}[Regular Character]\label{ExampleRegularRepresentation}Let $G$ be a group, and let $L(G)$ denote the (left) group von Neumann algebra generated by the left-regular representation of $G$ on $\ell^2(G)$. Let $\xi$ be the characteristic function of the identity element $e \in G$. Then the function $\chi_{\mathrm{reg}}(g) = (g\xi, \xi) = \delta_{e,g}$ defines a character on $G$, called the \emph{regular character}. By Proposition~\ref{PropEquivalenceExtremeCharactersFactors}, $\chi_{\mathrm{reg}}$ is indecomposable if and only if $L(G)$ is a factor, which holds precisely when $G$ is an ICC group.

We note that in view of Proposition~\ref{PropGNSQuasiEquivalence}, any finite-type factor representation that generates the regular character is quasi-equivalent to the regular representation.
\end{example}

The following example illustrates how measure-preserving group actions give rise to characters. This construction, explicitly presented in \cite{Vershik-Nonfree-10}, uses the notion of groupoid von Neumann algebras developed by Feldman and Moore in \cite{FeldmanMoore-ErgodicRelations-77Part2}.

%

\begin{example}[The Feldman-Moore Construction] \label{ExampleFeldmanMooreConstruction}
Let $X$ be a standard Borel space, $G$ a countable group acting on $X$ by Borel transformations, and $\mu$ a $G$-invariant probability Borel measure on $X$. Denote by $\mathcal{R} = \{(x, gx) : x \in X,\, g \in G\}$ the orbit equivalence relation of $(X, G)$. The group $G$ admits a left action on the space $\mathcal{R}$,  
defined by  
$g(x, y) = (gx, y)$, $(x, y) \in \mathcal{R}$, $g \in G$.   

 For a Borel subset $A \subset \mathcal{R}$, define  
\begin{equation*}\widetilde{\mu}(A) = \int_{X} \operatorname{card}(A_x) \, d\mu(x),  \end{equation*}
where $A_x = \{(x, y) \in A\}$. See \cite[Theorem 2]{FeldmanMoore-ErgodicRelations-77} for details. The measure $\widetilde{\mu}$ is Borel and invariant with respect to the left action of $G$.  Furthermore, the restriction of $\widetilde{\mu}$ to the diagonal  
\[ \Delta = \{(x, x) : x \in X\} \]  
coincides with the measure $\mu$, after identifying $\Delta$ with $X$.  

Consider the Hilbert space $\mathcal{H} = L^2(\mathcal{R}, \widetilde{\mu})$. Define the left  
representation of the group $G$ on $\mathcal{H}$ by  
\[
(\pi(g)f)(x,y) = f(g^{-1}x, y), \quad f \in \mathcal{H}.
\]  
The representation $\pi$ is called a {\it groupoid representation}.  The von Neumann algebra generated by the operators $\{\pi(G)\}$ will be denoted by $\mathcal{M}_\pi$.  

For each function $m \in L^\infty(X, \mu)$, define the operator $W_m: \mathcal{H} \to \mathcal{H}$ by  
\[
(W_m f)(x, y) = m(x) f(x, y).
\]

Denote by $\mathcal{M}_{\mathcal{R}}$ the von Neumann algebra generated by $\mathcal{M}_\pi$ and the operators $W_m$, where $m \in L^\infty(X, \mu)$. While technically the algebra $\mathcal{M}_{\mathcal{R}}$ is larger than $\mathcal{M}_\pi$, Theorem \ref{ThmCharacterDudkoGrigorchuk} will show that for a large class of systems, these two algebras coincide. We note that for ergodic actions, the algebra $\mathcal{M}_{\mathcal{R}}$ is a finite-type factor \cite[Proposition 2.9]{FeldmanMoore-ErgodicRelations-77Part2}. 

Let $\xi = 1_\Delta \in \mathcal H$ denote the characteristic function of the diagonal $\Delta \subset X \times X$. We observe that the vector $\xi$ is cyclic and separating for the von Neumann algebra $\mathcal{M}_{\mathcal{R}}$  \cite[Proposition 2.5]{FeldmanMoore-ErgodicRelations-77Part2}. A straightforward computation shows that
$$
(\pi(g)\xi, \xi) = \mu(\fix(g)) \quad \text{for all } g \in G,
$$
which implies that the function $\chi(g) = \mu(\fix(g))$ defines a character on $G$.
\end{example}

%
%
\begin{example}[The Product Action]\label{exampleProductAction}  Let $X$ be a standard Borel space, $G$ a countable group acting on $X$ by Borel transformations, and $\mu_1,\mu_2,\ldots,\mu_n$ be a collection (repetitions are permitted) of $G$-invariant  probability Borel measures on $X$. Consider the product acton of $G$ on $X^n$. Then the product measure $\mu = \mu_1\times \cdots \times \mu_n$ is $G$-invariant and $\chi(g) = \mu(\fix(g)_n)$, $g\in G$, is a character on $G$, where $$\fix(g)_n = \{(x_1,\ldots, x_n)\in X^n : g(x_i)=x_i\} .$$   Note that    $\chi(g) = \prod_{i=1}^n \mu_i(\textrm{Fix}(g))$ for every $g\in G$.

  Denote by $\mathcal R$ the $G$-orbit equivalence relation for the product action of $G$ on $X^n$,  by $\widetilde{\mu} $ the counting measure on $\mathbb R$, and by $\pi$ the unitary representation of $G$ in $L^2(\mathcal R, \widetilde{\mu})$ as in Example \ref{ExampleFeldmanMooreConstruction}. 
 Define $\mathcal H_\xi$ as the closure of $\textrm{Lin}(\pi( G) \xi)$, where  $\xi$ is  the indicator function of the diagonal in $X^n\times X^n$.   Denote by $\mathcal M_\pi$ the von Neumann algebra generated by $\pi(G)$ on $\mathcal H_\xi$.   

 Note that the functional $\tr(m) = (m\xi, \xi)$ defines a faithful normal trace on $\mathcal{M}_\pi$ and satisfies $\tr(\pi(g)) = \chi(g)$ for every $g \in G$. If $\mathcal{M}_\chi$ is the von Neumann algebra obtained via the GNS construction associated with the character $\chi$, then $\mathcal{M}_\pi$ and $\mathcal{M}_\chi$ are unitarily equivalent; see, for example, the argument in \cite[Proposition 1.B.8]{BekkadeLaHarpeBook:2020}.
  \end{example}

The indecomposability of permutation characters defined in Examples  \ref{ExampleFeldmanMooreConstruction} and \ref{exampleProductAction} will be discussed in Section \ref{Section:PermutationCharacter}.

The following result shows that groupoid representations are continuous with respect to the uniform topology induced by the underlying group action. For a detailed discussion of the weak and uniform topologies on the group of measure-preserving transformations, we refer the reader to the classical monograph~\cite{Halmos:LecturesErgodicTheory}.

\begin{proposition}\label{PropGroupRepresentationContinuity}
Let $G$ act on a standard probability space $(X, \mu)$ by measure-preserving transformations. Denote by $(\pi, L^2(\mathcal{R}, \widetilde\mu), \xi)$ the corresponding Feldman-Moore groupoid representation (see Example \ref{ExampleFeldmanMooreConstruction}), and let $\mathcal{H} = [\pi(G)\xi]$ be the cyclic subspace generated by $\xi$. 

Assume that a sequence $\{g_n\}$ of elements of $G$ is Cauchy  with respect to the psuedometric 
$$
d_\mu(g,h) = \mu(\{x \in X : gx \neq hx\}).
$$
Then the sequence $\{\pi(g_n)|_{\mathcal{H}}\}$ converges in the strong operator topology. In particular, if $\{g_n\}$ converges to $g\in G$ in the pseudometric $d_\mu$, then  $\{\pi(g_n)|_{\mathcal{H}}\}$ converges to $\pi(g)|_{\mathcal{H}}$  in the strong operator topology. 
\end{proposition}
\begin{proof}
Let $\varepsilon > 0$ and $\eta \in \mathcal{H}$. There exists a finite linear combination of the form
$$
\eta' = \sum_{i=1}^k c_i \pi(h_i)\xi
$$
such that
$
\bigl\| \eta - \eta' \bigr\| < \frac{\varepsilon}{4}.
$
Set $R = \max_{1 \le i \le k} |c_i|$. Choose $N$ such that
$$
d_\mu(g_n, g_m) < \frac{\varepsilon^2}{8R^2 k^2} \quad \text{for all } n,m \ge N.
$$
A direct computation shows that for each $1 \le i \le k$ and all $n,m \ge N$,
\begin{equation*}\begin{split}
\|\left(\pi(g_n) - \pi(g_m)\right)\pi(h_i)\xi \|^2  
& = 2 - 2 \operatorname{Re} (\pi(h_i^{-1} g_m^{-1} g_n h_i)\xi,\xi)  \\
&= 2 - 2\mu(\mathrm{Fix}(g_m^{-1} g_n))  \\ & =  2d_\mu(g_n, g_m) 
< \frac{\varepsilon^2}{4R^2 k^2}.
\end{split}
\end{equation*} 
Thus, for every $n,m \ge N$, we have that 
$$\|\pi(g_n)\eta' - \pi(g_m)\eta'\| = \| \sum_{i=1}^k \left(\pi(g_n) - \pi(g_m)\right) c_i \pi(h_i)\xi \| <\frac{\varepsilon}{2}.$$

It follows that 
$$
\|\pi(g_n)\eta - \pi(g_m)\eta\| < \varepsilon, \qquad \text{for all } n,m \ge N.
$$
Hence $\{\pi(g_n)\eta\}$ is a Cauchy sequence in $\mathcal{H}$, and therefore convergent. Since this holds for every $\eta \in \mathcal{H}$, the sequence $\{\pi(g_n)|_{\mathcal{H}}\}$ converges in the strong operator topology. 

If the sequence $\{g_n\}$ converges to a group element $g$, then, by applying the  argument above to the Cauchy sequence $\{g_1,g, g_2, g, \ldots\}$, we obtain that  $\{\pi(g_n)|_{\mathcal{H}}\}$ converges to $\pi(g)|_{\mathcal{H}}$  in the strong operator topology. 
\end{proof}

Note that in the previous result we considered the restriction of $\pi$ to the cyclic subspace $\mathcal{H}$, rather than to the entire Hilbert space $L^2(\mathcal{R}, \nu)$. In the case of a perfectly non-free ergodic action, however, this distinction disappears, since then $\mathcal{H} = L^2(\mathcal{R}, \widetilde\mu)$; see Example~\ref{ExampleFeldmanMooreConstruction} and Theorem~\ref{ThmCharacterDudkoGrigorchuk}. 

%
%
%
 %

 \subsection{Full Groups  and  Their Algebraic Properties}\label{Section:FullGroups}
 
In this section, we define full groups of dynamical systems and those arising from Bratteli diagrams, and discuss some of their algebraic properties. We cite the most general results when possible, while also referencing earlier work on full groups of Bratteli diagrams and minimal $\mathbb{Z}$-systems.

\begin{definition}\label{Definition_Bratteli_Diagram} A {\it Bratteli diagram} is an
infinite graph $\mathrm B=(V,E)$ such that the vertex set
$V=\bigcup_{i\geq 0}V_i$ and the edge set $E=\bigcup_{i\geq 1}E_i$
are partitioned into disjoint subsets $V_i$ and $E_i$ such that

(i) $V_0=\{v_0\}$ is a single point;

(ii) $V_i$ and $E_i$ are finite sets;

(iii) there exist a range map $r$ and a source map $s$ from $E$ to
$V$ such that $r(E_i)= V_i$, $s(E_i)= V_{i-1}$, and $s^{-1}(v)\neq\emptyset$, $r^{-1}(v')\neq\emptyset$ for all $v\in V$ and $v'\in V\setminus V_0$.
\end{definition}

 The pair  $(V_i,E_i)$ is called the $i$-th level of the diagram $\mathrm B$.
A finite or infinite sequence of edges $(e_i : e_i\in E_i)$ such
that $r(e_{i})=s(e_{i+1})$ is called a {\it finite} or {\it infinite path},
respectively. A graphical example of a Bratteli diagram is given in Figure (\ref{fig:BratteliDiagram1}).

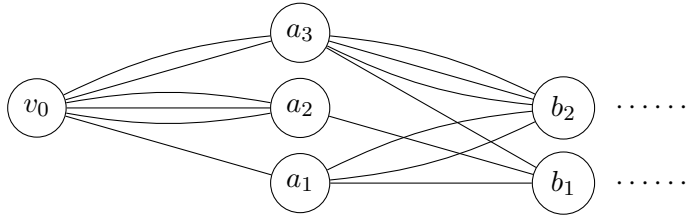
\begin{figure}[ht]
    \centering 
\begin{tikzpicture}[every node/.style={circle, draw, minimum size=0.01cm}, 
    sibling distance=1cm, level distance=3.5cm, 
    edge from parent/.style={draw}, 
    grow=east]
    
\node (root) {$v_0$}
    child {node (a1) {$a_1$}
        child {node (b1) {$b_1$}}         
    }
    child {node (a2) {$a_2$}
        child {node (b2) {$b_2$} edge from parent[draw=none]}
    }
    child {node (a3) {$a_3$}
    };

\draw[bend left=10] (root) to (a3);
\draw[bend left=10] (root) to (a2);
\draw[bend right=10] (root) to (a2);
\draw (a2) to (b1);
\draw[bend right=10] (a1) to (b2);
\draw[bend left=10] (a1) to (b2);
\draw (a3) to (b1);
\draw (a3) to (b2);
\draw[bend right=10] (a3) to (b2);
\draw[bend left=10] (a3) to (b2);

\node[right=0.1cm of b2, draw=none, fill=none] (dots1) {$\cdots\cdots$};
\node[right=0.1cm of b1, draw=none, fill=none] (dots2) {$\cdots\cdots$};
\end{tikzpicture}
\caption{A Bratteli diagram with a root node $v_0$, $V_1=\{a_1,  a_2, a_3\}$, and $V_2  = \{ b_1 , b_2 \}$.}
    \label{fig:BratteliDiagram1}
\end{figure}

 For a Bratteli diagram
$\mathrm B$, we denote by $X_\mathrm{B}$ the set of infinite paths starting at the root vertex $v_0$. We will   omit the index $\mathrm B$ in the notation $X_\mathrm{B}$ as the diagram will be  clear from the context. We endow the set $X$ with the topology generated by  cylinder sets $U(e_1,\ldots,e_n)=\{x\in X : x_i=e_i,\;i=1,\ldots,n\}$, where
$(e_1,\ldots,e_n)$ is a finite path from $\mathrm B$. Then $X$ is a 0-dimensional compact metric space with
respect to this topology. Observe that the space $X$ might have isolated points, but in the context discussed in this paper, $X$ will be perfect, i.e. $X$ will be a Cantor set.

 Denote by $h_v^{(n)}$, $v\in V_n$, the number of finite paths connecting the root vertex $v_0$ and the vertex $v$. Given a vertex $v\in V_n$, enumerate the cylinder sets $U(e_1,\ldots,e_n)$ with $r(e_n)=v$ by $\{C_{v,0}^{(n)},\ldots,C_{v,h_v^{(n)-1}}^{(n)}\}$. Then

 \begin{equation}\label{eqnXin} \Xi_n = \{C_{v,i}^{(n)} : v\in V_n,\;i=0,\ldots,h_v^{(n)}-1\}
 \end{equation}
 is a clopen partition of $X$, dubbed a {\it Kakutani-Rokhlin partition of $X$}. The collection of sets $\{C_{v,0}^{(n)},\ldots,C_{v,h_v^{(n)-1}}^{(n)}\}$ will be referred to as a {\it tower} associated with the level $n$.   Notice that  if $A$ is a clopen subset of $X$, then we can  always find a level $n\geq 1$ such that the set $A$ is a disjoint union of cylinder sets $U(e_1,\ldots,e_n)$ of depth $n$.  We will say then that the set $A$ is {\it compatible with the partition $\Xi_n$.}

Fix a Bratteli diagram $\mathrm B = (V,E)$ and a level $n\geq 1$.
Denote by $G_n$ the group of all homeomorphisms of $X: = X_B$
that changes only the first $n$-edges of each infinite path.  Namely, a homeomorphism
$g$ belongs to $G_n$ if and only if for any finite path
$(e_1,\ldots,e_n)$ there exists a finite path $(e_1',\ldots,e_n')$
with $r(e_n')=r(e_n)$ such that
$$ g(e_1,\ldots,e_n,e_{n+1},e_{n+2},\ldots)=(e_1',\ldots,e_n',e_{n+1},e_{n+2},\ldots)$$
for each infinite extension $(e_1,\ldots,e_n,e_{n+1},e_{n+2},\ldots)$ of the path $(e_1,\ldots,e_n)$.
 Observe that $G_n \cong \prod_{v\in V_n}\textrm{Sym}(h_v^{(n)})$, where $\textrm{Sym}(h_v^{(n)})$ is the symmetric group acting on the finite paths connecting the root vertex $v_0$ to the vertex $v\in V_n$.  Each group $G_n$ naturally embeds into  $G_{n+1}$ via a block-diagonal embedding,  where each factor $\textrm{Sym}(h_v^{(n)})$ embeds into $\textrm{Sym}(h_w^{(n+1)})$, $w\in V_{n+1}$, with multiplicity given by the number of edges connecting $v$ to $w$.  Denote by $G = \lim\limits_{\to}G_n$ the direct limit of groups $\{G_n\}$.  The group $G$ is countable and locally finite. 
 
 \begin{definition} (i) The group $G$ defined above is called the {\it full group of the Bratteli diagram $\mathrm B$}. 
 (ii)  The commutator (derived) subgroup of $G$ will be denoted by $G'$. 
  \end{definition}

\begin{example} (1) The finite groups $\{G_n\}$ associated with the Bratteli diagram depicted in Figure (\ref{fig:BratteliDiagram1}) are given by
 $$G_0=\textrm{Sym}(1),\; G_1 = \textrm{Sym}(1)\times \textrm{Sym}(3)\times \textrm{Sym}(2),\;G_2 = \textrm{Sym}(6)\times \textrm{Sym}(8),\ldots.$$

(2)  The finite groups associated with the Bratteli diagram in Figure (\ref{fig:BratteliDiagram:Rational}) are isomorphic to $ \textrm{Sym}(n!)$. The inductive limit $G = \lim\limits_{\to}G_n$ is isomorphic to the group of rational permutations of the interval $[0,1)$, see details in  \cite{goryachko:CharactersRational}. 

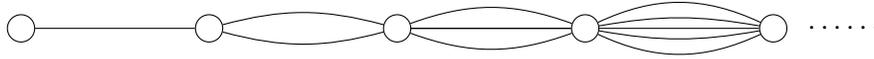
\begin{figure}[ht]
    \centering
    \begin{tikzpicture}[every node/.style={circle, draw, minimum size=0.01cm}, 
        sibling distance=0.2cm, level distance=2.5cm, 
        edge from parent/.style={draw}, 
        grow=east]
        
    \node (root) {}
        child {node (l1) {}
            child {node (l2) {} edge from parent[draw=none]
                child {node (l3) {}
                    child {node (l4) {}}
                }
            }
        };

    \draw[bend left=15] (l1) to (l2);
    \draw[bend right=15] (l1) to (l2);
    
    \draw[bend left=20] (l2) to (l3);
    \draw[bend right=20] (l2) to (l3);
    \draw (l2) to (l3);

    \draw[bend left=25] (l3) to (l4);
    \draw[bend right=25] (l3) to (l4);
    \draw[bend left=10] (l3) to (l4);
    \draw[bend right=10] (l3) to (l4);
   
     \node[right=0.1cm of l4, draw=none, fill=none] (dots) {$\cdots\cdots$};

    \end{tikzpicture}
    \caption{A Bratteli diagram with one vertex per level and the exactly $n$ edges between levels $V_{n-1}$ and $V_n$.}
    \label{fig:BratteliDiagram:Rational}
\end{figure} 
\end{example}

Given a Bratteli diagram $\mathrm{B}=(V,E)$, it is straightforward to check that the action of its full group $G$ on the path-space is minimal if and only if the diagram $\mathrm B$ is {\it simple} in the sense that   for any level $V_n$ there is a level $V_m$, $m>n$, such that every pair of vertices $(v,w)\in V_n\times V_m$ is connected by a finite path.

 \begin{definition}\label{DefFullGroups}  

Let $(X, G)$ be a Cantor dynamical system.  Throughout the paper,  the group $G$ is assumed to be countable. 

\begin{itemize}
    \item[(i)] Define $\mathcal{F}(G)$ as the set of homeomorphisms of $X$ that act locally as elements of $G$. More precisely, a homeomorphism $g \in \mathcal{F}(G)$ if there exists a clopen partition $\{A_1, \dots, A_n\}$ of $X$ and elements $\{g_1, \dots, g_n\} \subset G$ such that $g|_{A_i} = g_i|_{A_i}$ for each $i$. The group $\mathcal{F}(G)$ is called the \textit{full group} of the system $(X, G)$.  
    
    \item[(ii)] Denote by  $\mathcal D(G)$ the {\it commutator (derived) subgroup} of $\mathcal F(G)$, that is, the subgroup generated by the set of commutators $$<[g,h] : g,h\in \mathcal F(G)>.$$ 

    \item[(iii)] Define $\mathcal{A}(G)$, which, following \cite{Nekrashevych:2019} (see also \cite[Section 5.1.3]{Nekrashevych:Book2022}), we refer to as the \textit{alternating full group}, as the subgroup of $\mathcal{F}(G)$ generated by cycles of length 3. Here, an element $g \in \mathcal{F}(G)$ is called a \textit{cycle of length} $n$ if there exists a clopen set $U$ such that the sets $\{U, g(U), \dots, g^{n-1}(U)\}$ are disjoint, $g^n|_U = \mathrm{id}$, and $g$ acts trivially on the complement of $\bigcup_{i=0}^{n-1} g^i(U)$.  
    
    \item[(iv)] Given a clopen set $U$, define $\mathcal F(G|_U)$ as the set of elements $\mathcal F(G)$ with $\supp(\gamma)\subset U$, where $\supp(g) = \{x\in X : \gamma(x) \neq x\}$.  The groups $\mathcal D(G|_U)$ and $\mathcal A(G|_U)$ are defined analogously. The groups $\mathcal F(G|_U)$, $\mathcal D(G|_U)$, and $\mathcal A(G|_U)$ are called {\it local subgroups}.
\end{itemize}
\end{definition} 

In the literature, full groups are often denoted by $[[G]]$, but for consistency with the notations $\mathcal{A}(G)$ and $\mathcal{D}(G)$, we will use $\mathcal{F}(G)$. 

In the language of groupoids, the full groups associated with Bratteli diagrams are precisely the full groups of approximately finite (AF) groupoids; see, for example, \cite{Matui:2012} and \cite[Section 5.2]{Nekrashevych:Book2022}.

In \cite{grigorchuk2024maximalsubgroupsamplegroups}, the authors refer to $\mathcal{F}(G)$ as the \emph{full amplification} of $G$, and define a group $G$ to be \emph{ample} if $\mathcal{F}(G) = G$. Note that $\mathcal{F}(\mathcal{F}(G)) = \mathcal{F}(G)$.

\begin{remark} If $G$ is the full group of a Bratteli diagram, then it is ample in the sense that $
\mathcal{F}(G) = G$,
whereas its derived subgroup $G'$ may not be ample. We also note that $
\mathcal{A}(G) = G'$. 

Proposition 2.11 in \cite{DudkoMedynets-CharactersSymmetric-13} establishes that the full group $G$ of a Bratteli diagram is simple, in which case  $G = G’$, if and only if there exists an increasing sequence of levels ${n_k}$ such that, for every $k \geq 0$, the number of paths between any vertex in $V_{n_k}$ and any vertex in $V_{n_{k+1}}$ is a nonzero even number.
\end{remark}

If $(X,G)$ is a Cantor minimal system, then $
\mathcal{A}(G) \subset \mathcal{D}(G)$.  It remains an open question whether $
\mathcal{A}(G) = \mathcal{D}(G)$ for all minimal systems; cf.\ Proposition \ref{PropositionDerivedSubgroupFullGroup}.

The following result establishes the simplicity of the alternating full groups for Cantor minimal systems. Statements (i) and (ii) are proved in full generality in \cite[Theorem 4.1]{Nekrashevych:2019} (see also \cite[Theorem 5.1.10]{Nekrashevych:Book2022}). Earlier proofs for full groups of Bratteli diagrams appear in \cite[Lemma 3.4]{Matui:2006} and \cite[Theorem 5.1]{LavrenyukNekrashevych:2007}. Analogous results for full groups of Cantor minimal $\mathbb{Z}$-systems are given in \cite[Theorem 4.9]{Matui:2006} and \cite[Theorem 3.4]{BezuglyiMedynets:2008}.

Statement (iii) follows from \cite[Lemma 3.5]{Nekrashevych:2019}, with a formal proof provided in \cite[Lemma 4.8]{zheng2020rigid}. For the full groups of minimal $\mathbb{Z}$-actions, Statement (iii) was established in \cite[Lemma 3.19]{GiordanoPutnamSkau:1999}, and for full groups of simple Bratteli diagrams in the proof of \cite[Lemma 3.4]{DudkoMedynets-IRS-17}.

\begin{proposition}\label{PropositionSimplicityCommutatorSubgroup} Let  $(X,G)$ be a Cantor  minimal system. Then  
\begin{itemize}
\item[(i)] The group $\mathcal A(G)$ is simple and is contained in every non-trivial normal subgroup of the full group $\mathcal F(G)$;
\item[(ii)]  For every clopen set $U$, the local subgroup $\mathcal A(G|_U)$ is simple;
\item[(iii)]  If $U$ and $V$ are clopen sets such that $U\cap V \neq \emptyset$, then $$\left <\mathcal A(G|_U), \mathcal A(G|_V) \right> = \mathcal A(G|_{U\cup V}).$$
\end{itemize}
\end{proposition}

In \cite{Matui:2012}, Matui introduced the notion of almost finiteness for the class of \'etale groupoids with compact zero-dimensional unit spaces, which was later translated into the language of group actions in \cite[Lemma 5.2]{Suzuki:AlmostFiniteness}.

\begin{definition}[Almost Finiteness]\label{DefAlmostFiniteness} Let $(X,G)$ be a Cantor dynamical system.  

\begin{itemize} 

\item[(i)] A {\it tower} is a pair $(S,B)$, where $B$ is a clopen subset of $X$ and $S\subset G$ such that the sets $
\{s(B) : s\in S\}$ are pairwise disjoint. The set $S$ is called the {\it shape} of the tower $(S,B)$. A {\it castle} is a family of towers with disjoint supports. 

\item[(ii)] We say that a finite set $F\subset G$ is {\it $(K,\delta)$-invariant} for some finite subset $K\subset G$ and $\delta>0$ if $|KF  \triangle F | <\delta |F|$.  

\item[(iii)] The dynamical system $(X,G)$ or, equivalently, the action of $G$ on $X$ is said to be {\it almost finite} if for any finite set $K\subset G$, and $\delta>0$ there is a clopen castle whose shapes are $(K,\delta)$-invariant and whose levels partition $X$.  

\end{itemize}

\end{definition}

If $G$ is the full group of a Bratteli diagram with path-space $X$, then the system $(X,G)$ is almost finite. The open-standing conjecture in the field is that {\it free} actions of amenable groups on the Cantor set are almost finite.  In \cite{KerrNaryshkin:2021}, the authors showed that elementary amenable groups satisfy the almost finiteness conjecture. We refer the  reader to \cite{naryshkin2025finitenessgroupsdynamicalorigin} for the current status of this conjecture. 

 The following result, in its generality, is established in \cite[Theorem 4.7]{Matui-FullGroups-2015}.  Earlier proofs for the full groups of minimal $\mathbb Z$-actions appear in \cite[Theorem 4.9]{Matui:2006} and in \cite[Theorem 3.4]{BezuglyiMedynets:2008}.

\begin{proposition}\label{PropositionDerivedSubgroupFullGroup}  Let $(X,G)$ be an almost finite minimal dynamical system.  Then $\mathcal A(G)$ coincides with the commutator subgroup of $\mathcal F(G)$.   
\end{proposition}

 Let $T$  be a minimal homeomorphism of a Cantor set $X$.  Consider the full group $\mathcal F(\langle T\rangle)$ generated by the action of $T$ and for $x_0\in X$ define $\mathcal F(\langle T \rangle)_{x_0}$ the subgroup of $\mathcal F(\langle T \rangle)$ consisting of homeomorphisms preserving the forward $T$-orbit of $x_0$.  We note that  $\mathcal F(\langle T\rangle)_{x_0}$ is a maximal subgroup in $\mathcal F(\langle T\rangle)$, see  \cite[Proposition 5.2]{GrigorchukMedynets-AlgebraicFull-14}.

 In \cite{HermanPutnamSkau:1992}, the authors showed that for a given Cantor minimal system $(X,T)$ and a point $x_0 \in X$, one can construct a Bratteli diagram $B$ and a homeomorphism $\varphi_B$, known as the {\it Vershik map}, on the path space $X_B$ of $B$, such that the systems $(X,T)$ and $(X_B, \varphi_B)$ are topologically conjugate. Moreover, under this construction, the conjugacy map between the systems induces an isomorphism between the group $\mathcal F(\langle T \rangle)_{x_0}$ and the full group of the Bratteli diagram $B$. Consequently, the groups $\mathcal F(\langle T \rangle)_{x_0}$, $x_0 \in X$, are locally finite. Furthermore,  it follows from \cite[Corollary 3.6]{Krieger:1980} that the groups $\mathcal F(\langle T \rangle)_{x_0}$, $x_0 \in X$, are isomorphic to one another.

In \cite[Corollary 4.11]{GiordanoPutnamSkau:1999} the authors established that two Cantor minimal systems $(X,T)$ and $(Y,S)$ are strong orbit equivalent if and only if the groups $\mathcal F(\langle T \rangle)_{x_0}$, $x_0 \in X$, and $\mathcal F( \langle S \rangle)_{x_0}$, $y_0 \in Y$, are isomorphic as abstract groups. In the proof, they established that every abstract isomorphism between these groups is spatially generated as these groups encode a large amount of information about the underlying space. Similar reconstruction results have been established for larger classes of full groups  \cite[Corollary 4.4]{GiordanoPutnamSkau:1999}, \cite{Medynets:2011},  \cite[Theorem 3.5]{Matui-FullGroups-2015}, and \cite[Theorem 5.1.20]{Nekrashevych:Book2022}).

We will need the following result from \cite[Theorem 5.4]{GrigorchukMedynets-AlgebraicFull-14}. 

\begin{proposition}\label{PropStructureDerivedFullGroup}
Let $T$ be a minimal homeomorphism of the Cantor set $X$. For each $z \in X$, let $\Gamma_z$ denote the subgroup of $\mathcal A(\langle T \rangle)$ consisting of homeomorphisms that preserve the forward $T$-orbit of $z$. If $x,y \in X$ belong to distinct $T$-orbits, then
\[
\mathcal A(\langle T \rangle) = \Gamma_x \cdot \Gamma_y .
\]

Moreover, for every $g \in \mathcal A(\langle T \rangle)$, there exist sequences $\{p_n\} \subseteq \Gamma_x$ and $\{r_n\} \subseteq \Gamma_y$ such that
$
g = p_n r_n \quad \text{for all } n,
$
and the supports of the elements $r_n$ converge to a finite set in the following sense: there exists a finite set $Z \subset X$ such that for every clopen neighborhood $U$ of $Z$,
$
\supp(r_n) \subset U$ for all sufficiently large $n$.
\end{proposition}

\begin{proposition}\label{PropCosetFullGroupSmallSupport}
Let $T$ be a minimal homeomorphism of the Cantor set $X$, let $\Gamma$ be either $\mathcal F(\langle T \rangle)$ or the full group of the associated Bratteli diagram, and let $\Gamma'$ be the commutator subgroup of $\Gamma$. Then for any $\gamma \in \Gamma$ and any nonempty clopen set $C$ there exist elements $\gamma_1 \in \Gamma'$ and $\gamma_2 \in \Gamma$ such that $\gamma = \gamma_1 \gamma_2$ and $\supp(\gamma_2) \subset C$.
\end{proposition}

\begin{proof}
Fix a nonempty clopen set $C$. Denote by $T_C$ the induced transformation on $C$, that is $T_C(x) = T^{n_C(x)}(x)$ if $x\in C$ and $T_C(x) =  x$ elsewhere, where $n_C(x) = \min\{n\geq 1 : T(x)\in C\}$.

First consider the case when $\Gamma$ is the full group of the Bratteli diagram associated with $(X,T)$. Given $\gamma \in \Gamma$, multiply $\gamma$ by an odd permutation supported in $C$ if necessary, we can obtain an element of $\Gamma'$ whose difference from $\gamma$ is supported inside $C$, establishing the claim in this case.

Now suppose $\Gamma = \mathcal F(\langle T \rangle)$. Fix a $T$-invariant measure $\mu$ and define the function $I(g) = \int_X n_g(x)\, d\mu(x)$, where the function $n_g$ is given by $g(x) = T^{n_g(x)} x$. The function $I$, called the index map, is a homomorphism from $\Gamma$ onto the group of integers $\mathbb Z$; see \cite[Section~5]{GiordanoPutnamSkau:1999} for details. Furthermore, the image of $T$ or any induced transformation is $1$; see, for example, the proof of Lemma~5.3 in \cite{GrigorchukMedynets-AlgebraicFull-14}. In particular, $I(T_C) = 1$.

Setting $n = I(\gamma)$, $\gamma_1 = T_C^n$, and $\gamma_0 = \gamma \gamma_1^{-1}$, we obtain $\gamma = \gamma_0 \gamma_1$, where $I(\gamma_0) = 0$ and $\gamma_1$ is supported in $C$.

Fix points $x$ and $y$ lying in distinct $T$-orbits. Applying Lemma~4.1 from \cite{Matui:2006} to $\gamma_0$, we can find elements $\gamma_x$ and $\gamma_y$ belonging to the full groups $\Gamma_x$ and $\Gamma_y$ of the Bratteli diagrams constructed with bases $\{x\}$ and $\{y\}$, respectively, such that $\gamma_0 = \gamma_x \gamma_y$.

Applying the argument above to $\gamma_x$ and $\gamma_y$, we can find elements $\gamma_x'$, $\gamma_x''$, $\gamma_y'$, and $\gamma_y''$ such that $\gamma_x = \gamma_x' \gamma_x''$, $\gamma_y = \gamma_y' \gamma_y''$, $\gamma_x'' \in \Gamma_x' \subset \Gamma'$, $\gamma_y'' \in \Gamma_y' \subset \Gamma'$, and $\gamma_x''$ and $\gamma_y''$ are supported in $C$.

It follows that
\[
\gamma = \gamma_0 \gamma_1 
= \gamma_x' \gamma_x'' \gamma_y' \gamma_y'' \gamma_1 
= \left( \gamma_x' \gamma_x'' \gamma_y' (\gamma_x'')^{-1} \right)
  \left( \gamma_x'' \gamma_y'' \gamma_1 \right).
\]

The first factor lies in $\Gamma'$, and the second factor is supported in $C$. This gives the desired decomposition and completes the proof.
\end{proof}

%
%
%

\section{The Permutation Character on Full Groups}\label{Section:PermutationCharacter}

In this section, we investigate properties of the character $\mu(\mathrm{Fix}(g))$ arising from measure-preserving actions of full groups. The main results are that the products of these permutation characters corresponding to ergodic measures are indecomposable, and that only integer powers of these characters are positive-definite (and hence remain characters).

The following result is a slight modification of \cite[Theorem 2.4]{ThomasTuckerDrob:2016}. We include the proof for the reader's convenience.  Recall that a measure preserving action is called {\it weakly mixing} if its Koopman representation has no nontrivial finite-dimensional subrepresentations, see, for example,  Definition 2.1 in \cite{Schmidt:1984}. 
\begin{proposition}\label{PropositionErgodicWeakMixing} Let $(X,G)$ be a Cantor minimal system. Then every ergodic action of $\mathcal A(G)$ is weakly mixing. 
\end{proposition}
\begin{proof} Consider an ergodic action of $\mathcal A(G)$  on a standard  measure space $(Z,\mu)$. Assume towards contradiction  that the Koopman representation on $L^2(Z,\mu)$ of $\mathcal A(G)$ admits a non-trivial finite-dimensional unitary subrepresentation. By Proposition \ref{PropositionSimplicityCommutatorSubgroup},  the group $\mathcal A(G)$ is simple. Therefore, the group $\mathcal A(G)$ must embed into  the general linear group $\textrm{GL}(n,\mathbb C)$ for some $n$. Since the group $\mathcal A(G)$ contains all symmetric groups $\textrm{Sym}(k)$, $k\geq 1$,  as subgroups, we obtain a contradiction.  The fact that $\textrm{GL}(n,\mathbb C)$ does not contain $\textrm{Sym}(k)$ for large values of $k$ can be derived, for example, from the Jordan-Schur Theorem \cite[Theorem 36.14]{CurtisReiner:1964}. 
 \end{proof}

The notion of \emph{perfectly non-free} actions was introduced in \cite{DudkoGrigorchuk-DiagonalActions-18}.

\begin{definition}
A measure-preserving action of a group $G$ on a standard measure space $(X, \mu)$ is called \emph{perfectly non-free} if there exists a countable collection $\mathcal A$ of measurable subsets of $X$ such that $\mathcal A$, together with the null sets, generates the $\sigma$-algebra, and for each $A \in \mathcal A$ and $\mu$-almost every $x \in A$, the set
\[
\{g x : g \in G,\; \mu(\supp(g) \setminus A) = 0\}
\]
is infinite.
\end{definition}

The full group actions considered in this paper satisfy a stronger non-freeness condition and are, in particular, perfectly non-free. The following result, established in \cite[Theorem 13]{DudkoGrigorchuk-DiagonalActions-18}, shows that ergodicity together with perfect non-freeness is sufficient for the indecomposability of the character $\chi(g) = \mu(\fix(g))$.

\begin{theorem}\label{ThmCharacterDudkoGrigorchuk} Let a countable group $G$ act on a standard measure space $(X,\mu)$ by measure-preserving transformations. Assume that the action is ergodic and perfectly non-free.   Let $\mathcal R$ be the associated orbit equivalence relation, $\pi$ the Feldman–Moore groupoid representation, and $\mathcal M_\mathcal R$ the groupoid von Neumann algebra as in Example~\ref{ExampleFeldmanMooreConstruction}. Then $\mathcal M_\pi = \mathcal M_{\mathcal R}$, $\mathcal M_\pi$ is a factor,  and the character $\chi(g) = \mu(\fix(g))$ is indecomposable.
\end{theorem}

 If $(X,G)$ is a Cantor minimal system with $G$ countable, then the action of the
alternating full group $\mathcal A(G)$ is perfectly non-free with respect to any
$G$-invariant measure $\mu$ (when such a measure exists). Consequently, the
character $g \mapsto \mu(\mathrm{Fix}(g))$ arising from an ergodic measure is
indecomposable. We also note that $G$ and $\mathcal A(G)$ share the same simplex
of invariant measures on $X$.

Product actions of $\mathcal A(G)$, however, are not necessarily perfectly
non-free. Nevertheless, in the following result we show that the characters of
$\mathcal A(G)$ defined using product measures with ergodic components remain
indecomposable. The proof follows the ideas of
\cite[Theorem~14]{DudkoGrigorchuk-DiagonalActions-18}. Furthermore, we show that
the characters defined in \eqref{eqnDefProductChar} are uniquely determined by
the ergodic measures used in their construction.

\begin{proposition}\label{propProductMeasuresIsCharacter}
Let $(X,G)$ be a Cantor minimal system that admits at least one invariant measure.  
Let $\mu_1, \mu_2, \ldots, \mu_n$ be a collection (repetitions allowed) of $G$-invariant ergodic measures on $X$. Then:

(i) the product measure $\mu = \mu_1 \times \mu_2 \times \cdots \times \mu_n$ is ergodic under the product action of $\mathcal A(G)$ on $X^n$;

(ii) the function
\begin{equation}\label{eqnDefProductChar}
\chi(g) = \mu(\mathrm{Fix}_{n}(g)) = \prod_{i=1}^{n} \mu_i(\mathrm{Fix}(g)), 
\qquad g \in \mathcal A(G),
\end{equation}
where $\mathrm{Fix}_{n}(g)$ denotes the set of fixed points of $g$ under the product action of $\mathcal A(G)$, is an indecomposable character of $\mathcal A(G)$.

Additionally, 

(iii) characters constructed by \eqref{eqnDefProductChar} from distinct families of ergodic measures are themselves distinct;

(iv) if any of the measures $\mu_i$ is not ergodic, then the character $\chi$ in \eqref{eqnDefProductChar} is not indecomposable.
\end{proposition}
\begin{proof} (i) Statement (i) follows from Proposition \ref{PropositionErgodicWeakMixing} and is an equivalent characterization of weak mixing, see  \cite[Proposition 2.2]{Schmidt:1984}. 

(ii) It follows from the Feldman-Moore construction, see Example \ref{ExampleFeldmanMooreConstruction}, that the function $\chi$ is a character of the group $\mathcal A(G)$. In what follows, we show that the character $\chi$ is indecomposable. 
 
The symmetric group $\sym(n)$ acts naturally on $X^n$ by permuting the coordinates of $(x_1, \ldots, x_n) \in X^n$. The measure 
\begin{equation}\label{eqnSumProductMeasures}\bar \mu  = \frac{1}{n!}\sum_{q\in \sym(n)}\mu_{q(1)}\times \cdots \times \mu_{q(n)}\end{equation} is both $\sym(n)$-invariant and $\mathcal A(G)$-invariant.  Notice that the actions of $ \sym(n) $ and $\mathcal A(G) $ commute. Thus, $ \mathcal A(G) $ preserves the orbits of $ \sym(n) $, and the product action naturally descends to the quotient space $ \widetilde{X} = X^n / \sym(n) $. Note that  $\widetilde{X}$ is a standard Borel space. Denote by $\varphi: X\rightarrow \widetilde{X}$  the factor map and by $\widetilde{\mu} $ the push-forward measure $ \bar{\mu} $. 

If $A\subset \widetilde{X}$ is a $\mathcal A(G)$-invariant measurable set, then $\varphi^{-1}(A)$ is both $\mathcal A(G)$- and $\sym(n)$-invariant. Hence, for $q\in \sym(n)$, we get that 
\begin{equation*}\begin{split}\mu_{q(1)}\times \cdots \times \mu_{q(n)}(\varphi^{-1}(A)) = \mu_{1}\times \cdots \times \mu_{n}(q^{-1}(\varphi^{-1}(A))) \\ = \mu_{1}\times \cdots \times \mu_{n}(\varphi^{-1}(A)).\end{split}\end{equation*}

Using (\ref{eqnSumProductMeasures}), we obtain that $\widetilde \mu(A) = \bar \mu(\varphi^{-1}(A)) = \mu_{1}\times \cdots \times \mu_{n}(\varphi^{-1}(A))$. By $\mathcal A(G)$-ergodicity of $\mu_{1}\times \cdots \times \mu_{n}$, we obtain that $\widetilde \mu(A)$ is either 1 or 0, which implies that the measure $\widetilde \mu$ is ergodic.  

Denote by $X_0^n$ the subset of $X^n$ consisting of points $(x_1, \ldots,x_n)$ such that $x_i\neq x_j$, $i\neq j$. Notice that $X_0^n$ is both $\mathcal A(G)$- and $\sym(n)$-invariant and $\bar \mu(X_0^n)=1$.  Set $\widetilde X_0 = \varphi(X_0^n)$. Our goal is to show the dynamical system $(\widetilde X_0,\widetilde \mu, \mathcal A(G))$ is perfectly non-free.

Observe that clopen subsets of $X$ separate $n$-element subsets of $X$. It follows that sets of the form $A^n$, where $A$ is clopen, separate points of $X_0^n$. Since sets of the form $A^n$ are $\sym(n)$-invariant, we obtain that the sets $\widetilde A_0 = \varphi(A^n) \cap \widetilde X_0$, where $A \subset X$ is clopen, separate points of $\widetilde X_0$.

Note that $(\widetilde X_0, \mathcal S)$, where $\mathcal S$ is the $\sigma$-algebra generated by the sets ${\widetilde A_0}$, is a standard Borel space \cite[Proposition 12.1]{Kechris:ClassicalDescriptiveSetTheory}. In fact, applying \cite[Theorem 14.12]{Kechris:ClassicalDescriptiveSetTheory}, one can show that $\mathcal S$ is the $\sigma$-algebra of all Borel subsets.

For every clopen set $A\subset X$ and for every $\bar x \in A^n$, the orbit $\{g (\bar x)\}$, where $g\in \mathcal A(G)$, $\supp(g)\subset A$, is infinite. Hence, the  dynamical system $(\widetilde X_0, \mathcal S, \widetilde \mu, \mathcal A(G))$ is perfectly non-free. Applying Theorem \ref{ThmCharacterDudkoGrigorchuk}, we obtain that the character defined by $\widetilde \mu(\fix(g))$, $g \in \mathcal A(G)$, is indecomposable. Since $\widetilde \mu(\fix(g)) =  \mu(\fix_n(g))$, we obtain that the character $\chi$ is indecomposable.

(iii) Fix two distinct families of ergodic measures $M$ and $N$.  
Grouping the measures in $M$ and $N$ according to their multiplicities, we may write
$$
M = \{\mu_{1}^{\times m_{1}},\ldots,\mu_{k}^{\times m_{k}}\}
\qquad\text{and}\qquad
N = \{\nu_{1}^{\times n_{1}},\ldots,\nu_{l}^{\times n_{l}}\},
$$
where $\mu_{i} \neq \mu_{j}$ and $\nu_{i} \neq \nu_{j}$ for $i \neq j$.  
Since the families $M$ and $N$ are distinct, we may assume without loss of generality that  
$$
\mu_{1} = \nu_{1} \quad\text{and}\quad m_{1} > n_{1},
$$
allowing $n_{1}=0$ if the family $N$ does not contain the measure $\mu_{1}$.

Using these families, construct the characters $\chi_{M}$ and $\chi_{N}$ via \eqref{eqnDefProductChar}.  
Because distinct ergodic measures are mutually singular, for any $\varepsilon > 0$ we can approximate the complement of the support of $\mu_{1}$ by a clopen set $C_{\varepsilon}$ of the form $C_{\varepsilon} = \mathrm{Fix}(g_{\varepsilon})$ for some $g_{\varepsilon} \in \mathcal A(G)$, such that
$$
\mu_{1}(C_{\varepsilon}) < \varepsilon
\qquad\text{and}\qquad
\mu_{i}(C_{\varepsilon}) > 1 - \varepsilon,\; \nu_{i}(C_{\varepsilon}) > 1 - \varepsilon
\quad\text{for all } i>1.
$$

It follows that
$$
\frac{\chi_{M}(g_{\varepsilon})}{\chi_{N}(g_{\varepsilon})}
    = \mu_{1}(C_{\varepsilon})^{\,m_{1}-n_{1}}
      \cdot
      \frac{
        \mu_{2}(C_{\varepsilon})^{m_{2}} \cdots \mu_{k}(C_{\varepsilon})^{m_{k}}
      }{
        \nu_{2}(C_{\varepsilon})^{n_{2}} \cdots \nu_{l}(C_{\varepsilon})^{n_{l}}
      }
    < 1
$$ for sufficiently small $\varepsilon$, which shows that $\chi_{M} \neq \chi_{N}$.

(iv) Finally, without loss of generality, suppose that 
$
\mu_1 = \alpha \mu_1' + (1-\alpha) \mu_1'',
$
where $\mu_1'$ and $\mu_1''$ are $G$-invariant probability measures. Then
$$
\chi(g) 
= \alpha \, \mu_1'(\mathrm{Fix}(g)) \prod_{i=2}^{n} \mu_i(\mathrm{Fix}(g))
  + (1-\alpha) \, \mu_1''(\mathrm{Fix}(g)) \prod_{i=2}^{n} \mu_i(\mathrm{Fix}(g)),
$$
which shows that the character $\chi$ is not indecomposable.
\end{proof}

 In \cite[Lemma A.3]{zheng2020rigid} Zheng showed that every ergodic measure invariant under the product action of the full group of a simple Bratteli diagram is necessarily a product measure.  It would be interesting to explore if a similar result can be obtained for more general full groups.

 
 \begin{example}\label{exampleS2Infty} Denote by $S_{2^\infty}$ the full group associated with the diagram depicted in Figure \ref{fig:BratteliDiagram:2Odometer}. The group $S_{2^\infty}$  is simple. The Vershik map on the diagram is topologically conjugate to the 2-odometer and, thus, is uniquely ergodic. Denote by $\nu$ the unique $S_{2^\infty}$-invariant  measure on the path-space of the Bratteli diagram associated with $S_{2^\infty}$. It was shown in \cite{Dudko-CharsErgodic-11} that every non-regular indecomposable character $\chi$ on $S_{2^\infty}$ is of the form $\chi(g)  = \nu(\fix(g))^n$, $g\in S_{2^\infty}$, where $n$ is a non-negative integer.

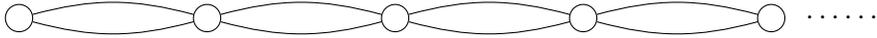
\begin{figure}[ht]
    \centering
    \begin{tikzpicture}[every node/.style={circle, draw, minimum size=0.01cm}, 
        sibling distance=0.2cm, level distance=2.5cm, 
        edge from parent/.style={draw}, 
        grow=east]
        
    \node (root) {}
        child {node (l1) {} edge from parent[draw=none]
            child {node (l2) {} edge from parent[draw=none]
                child {node (l3) {} edge from parent[draw=none]
                    child {node (l4) {} edge from parent[draw=none] }
                }
            }
        };

    \draw[bend left=15] (root) to (l1);
    \draw[bend right=15] (root) to (l1);
    
    \draw[bend left=15] (l1) to (l2);
    \draw[bend right=15] (l1) to (l2);
    
    \draw[bend left=15] (l2) to (l3);
    \draw[bend right=15] (l2) to (l3);

    \draw[bend left=15] (l3) to (l4);
    \draw[bend right=15] (l3) to (l4);
   
     \node[right=0.1cm of l4, draw=none, fill=none] (dots) {$\cdots\cdots$};

    \end{tikzpicture}
    \caption{The Bratteli diagram corresponding to the 2-odometer and the infinite symmetric group $S_{2^\infty}$.}
    \label{fig:BratteliDiagram:2Odometer}
\end{figure}
 
 \end{example}
 
 One of the key arguments in \cite{Dudko-CharsErgodic-11} is Proposition 12, 
which establishes that the class function 
$
g \mapsto \nu(\mathrm{Fix}(g))^\alpha
$, $\alpha \geq 0$, 
is positive semi-definite only if $\alpha$ is an integer. In Proposition \ref{PropAlphaAnInteger}, we generalize this fact to the setting of alternating full groups of almost-finite minimal systems. 

Our proof relies on the alternating representation of $\mathrm{Sym}(m)$ and ideas introduced in \cite[p.~3567]{Okunkov-CharactersSymmetric-97} as well as the fact that $\mathrm{Sym}(m)$ can be (almost densely) embedded into $\mathcal A(G)$.  The construction of the embedding of $\mathrm{Sym}(m)$ into $\mathcal A(G)$ used in the proof was suggested to us by Spyros Petrakos. 

It is interesting to compare Proposition \ref{PropAlphaAnInteger} with the discussion of Hadamard matrix functions in \cite[Theorem 6.3.7]{Horn_Johnson:1991}, where only integer power functions are known to preserve the cone of positive-definite matrices. A similar phenomenon occurs here, with the power function applied to a discrete set of matrices arising from characters of the group $\mathcal A(G)$.

We note that every almost-finite system admits at least one invariant measure \cite[Lemma 6.5]{Matui:2012}.

 %

  \begin{proposition}\label{PropAlphaAnInteger}  Let $(X,G)$ be an almost finite Cantor minimal system and let $\mu$ be a $G$-invariant probability measure on $X$. Then the function $\chi:  \mathcal A(G) \rightarrow \mathbb R $ given by $\chi(g) = \mu(\fix(g))^\alpha$, $\alpha\geq 0$, is a character on $\mathcal  A(G)$ if only if $\alpha$ is an integer. 
  \end{proposition}
 \begin{proof} The ``if'' part of the proof follows from the fact the product of characters is a character. 
 
Now, we will show that if $\alpha$ is not an integer, then the function $\chi$ is not positive-semidefinite.     Fix $k,n \in \mathbb N$ and $\delta>0$. Set $m = k^n$. Since the action of $G$ on $X$ is almost finite and the group $G$ is infinite, we can partition $X$ into a finite number of clopen towers (as in Definition \ref{DefAlmostFiniteness}) $\{S_i(B_i)\}$, $i\in I$, where $B_i$ is the base of the $i$-th tower, so that for each $i\in I$ the size of the shape $S_i$ is at least $2m/\delta $.  Note that 
\begin{equation}\label{eqnAux1AlphaAnInteger}\mu\left(\bigcup_{i\in I} B_i\right)\leq \frac{\delta}{2m}. \end{equation}

For each $i\in I$, we can remove $|S_i| \mod (2m)$ levels from the $i$-th tower $S_i(B_i)$. Thus, we  are left with  towers $\{S_i'(B_i)\}$ of height divisible by $2m$.  Thus, for each $i\in I$, we can combine levels of the tower $S_i'(B_i)$ to form a single tower   $Q_i(C_i)$ of height $2m$, where $C_i$ is a union of levels of $S_i'(B_i)$ and  the shape $Q_i$ is a subset of the topological full group $\mathcal F (G)$. 

It follows from (\ref{eqnAux1AlphaAnInteger}) that

$$\mu\left(\bigcup_{i\in I} Q_i(C_i)\right) = \mu\left(\bigcup_{i\in I} S_i'(B_i)\right)\geq 1-\delta.$$  

For each $i \in I$, fix an action of $ \sym(2m)$ on the levels $Q_i(C_i)$. By combining these actions, we obtain an embedding of $ \sym(2m) $ into $\mathcal{F}(G)$ with the property that there exists a clopen set $C = \bigcup_{i \in I} C_i$ such that the orbit \( \sym(2m)(C) \) consists of pairwise disjoint clopen sets and covers all of $X$ up to a set of $\mu$-measure less than $\delta$. Splitting the orbit \( \sym(2m)(C) \) into two suborbits of size $m$, we obtain an embedding of \( \sym(m) \) into \( \sym(2m) \) that doubles the support of each element of \( \sym(m) \). Note that $$\sym(m) \subset \alt(2m) \subset \mathcal{A}(G) .$$

Thus, we have constructed an  embedding of $\sym(k^n)$ into $\mathcal A(G)$ with the property that there is a nonempty clopen set $A$ such that  the $\sym(k^n)$-orbit of $A$ consists of pairwise disjoint clopen sets, the group   $\sym(k^n)$ acts trivially outside of its orbit  $\sym(k^n)(A)$, and $\mu(\sym(k^n)(A)) >1-\delta$.   Enumerate the disjoint clopen sets in $\sym(k^n)(A)$  by $\{A_{i,j}\}$, $0\leq i\leq n-1$, $0\leq j\leq k-1$.  This allows us to naturally identify the   $\sym(k^n)$-orbit of $A$ with the set $X_{k,n} = \{0,1,\ldots, k-1\}^n$. 
 
 Denote by $H_{n}\cong \sym(n) $ the subgroup of  $\sym(k^n)$ that acts on $X_{k,n}$ by permuting its coordinates, that is, for every $g\in H_{n}$ we have:  $g(x_i) = x_{g(i)}$ for every $0\leq i\leq n-1$.   Set $r = \mu(X\setminus \sym(k^n)(A))$.     Consider $s\in H_n$ and denote by $\{c_1,\ldots, c_q\}$ the independent cycles of $s$. Each fixed point in $s$ is accounted for as a cycle of length one.  Then 
\begin{equation}\label{eqnAuxCyclesAlphaAnIneteger}\mu(\fix(s)) = \prod_{i=1}^q k\left (\frac{1}{k}\right)^{|c_i|}\mu(A_{0,0}) + r =   \frac{k^{l(s)}}{k^n} \mu(A_{0,0}) + r, \end{equation} where $l(s)$ is the number of cycles in the permutation $s$. 
 
 Using the Feldman-Moore representation from Example \ref{ExampleFeldmanMooreConstruction}, we can find a unitary representation  $\pi$ of $\mathcal A (G)$ on a Hilbert space $\mathcal H$ and a vector $\xi$ such that $\chi(g) = (\pi(g)\xi,\xi) = \mu(\fix(g))^\alpha$ for every $g\in \mathcal A(G)$. Define the orthogonal projection $P = \frac{1}{n!}\sum_{s\in H_n} \textrm{sign}(s) \pi(s).$ Using (\ref{eqnAuxCyclesAlphaAnIneteger}), we obtain that 
 
\begin{align*}
 0 & \leq (P\xi,\xi) = \frac{1}{n!}\sum_{s\in H_n}  \textrm{sign}(s) \mu(\fix(s))^\alpha \\ 
 & = \frac{1}{n!}\sum_{s\in \sym(n)}  \textrm{sign}(s)  \left[\frac{k^{l(s)}}{k^n} \mu(A_{0,0}) + r\right]^\alpha.
 \end{align*}

  Since $r< \delta$ and $\delta$ can be chosen arbitrarily small, taking the limit as $\delta \to 0$, we obtain that $
  \mu(A_{0,0}) \to 1/k^n$ and 
  
  \begin{equation}\label{eqnAuxLastIneqAlphaAnInteger}0\leq \sum_{s\in \sym(n)}  \textrm{sign}(s) k^{\alpha l(s)}. \end{equation}
 
Using the equality $\textrm{sign}(s) = (-1)^{n+l(s)}$ for $s\in \sym(n)$ and the identity $\sum_{s\in \sym(n)} x^{l(s)} = x(x+1)\cdots (x+n-1)$, see, for example, \cite[Prop. 1.3.4]{StanleyEnumerativeCombinatoricsVol1}, in (\ref{eqnAuxLastIneqAlphaAnInteger}), we obtain that 
 \begin{align*} 0\leq \sum_{s\in \sym(n)}  (-1)^n  (-k^\alpha)^{l(s)} = k^\alpha (k^\alpha-1)\cdots (k^\alpha- n+1).
 \end{align*}
Since this inequality must be nonnegative for all $n\geq 1$, it must terminate for some $n$, which is possible only when $k^\alpha$ is an integer.  

Since $k^\alpha\in \mathbb N$ for every $k\geq 0$, we conclude that $\alpha$ must be an integer. This combinatorial fact appeared as Problem A-6 on the 1971 Putnam exam; a solution can be found, for example, in \cite{AMSMonthlyVol80}. 

 \end{proof}
 

%


  \section{Automatic Continuity of Finite-Type Unitary Representations}\label{Section:AutomaticContinuity}
  
In this section, we establish that finite-type unitary representations of groups acting on the Cantor set X, whose properties closely resemble those of full groups, automatically satisfy certain continuity properties. Specifically, Theorem \ref{ThAutomaticContinuity} shows that under appropriate technical assumptions, if a sequence of homeomorphisms converges to the identity, then any convergent subsequence of the corresponding unitary operators converges to a scalar multiple of the identity operator.

We note that the assumptions of Theorem \ref{ThAutomaticContinuity} are automatically satisfied by the full group $\mathcal F(\Gamma)$ of any free Cantor minimal system $(X, \Gamma)$ when the set $Y$ is a singleton. Moreover, if $(X, \Gamma)$ is free and almost finite, and $Y$ intersects each $\Gamma$-orbit at most once, then, by applying the comparison property \cite[Theorem 9.2]{Kerr:2022}, the theorem also holds for $\mathcal{F}(\Gamma)$.

We will need the following classical result \cite[Theorem 9.6-3]{Kreyszig:Book}.

  \begin{lemma}\label{LemmaLimitProjections} (a) If $\{P_n\}_{n=1}^\infty$ is a monotone decreasing sequence of orthogonal projections in some Hilbert space $\mathcal H$, then $\{P_n\}_{n=1}^\infty$ strongly  converges to the orthogonal projection onto $\bigcap_{n=1}^\infty P_n(\mathcal H)$.
  
  (b) If $\{P_n\}_{n=1}^\infty$ is a monotone increasing sequence of orthogonal projections in some Hilbert space $\mathcal H$, then $\{P_n\}_{n=1}^\infty$ strongly  converges to the orthogonal projection onto $\overline{\bigcup_{n=1}^\infty P_n(\mathcal H)}$.
    \end{lemma}
  
  %
  %

\begin{theorem}\label{ThAutomaticContinuity} 

Let $X$ be a Cantor space, and let $G$ be a  countable group acting on $X$ by homeomorphisms. Assume that  the fixed-point set $\operatorname{Fix}(g)$ is clopen for every $g \in G$.  Let $Y$ be a non-empty closed subset of $X$  that intersects every $G$-orbit at most once and satisfies the following conditions:

\begin{itemize}      
    \item[(a)] The family of sets  $
    \mathcal{Y} = \{q(Y) : q \in G\}
    $
 is {\it everywhere dense in $X$} in the sense that, for every clopen set $U\subset X$, there exists an element $q\in G$ such that $q(Y)\subset U$;
 
     \item[(b)] For every $n \geq 1$, the action of $G$ on  disjoint subsets of $\mathcal Y$
    is $n$-transitive  in the sense that, if $L\subset \mathcal Y$ is a family consisting of pairwise disjoint subsets and  $\{Y_1, \ldots, Y_n\}, \{Y_1', \ldots, Y_n'\} \subset L$, then there exists $g \in G$ such that $g(Y_i) = Y_i'$ for every $i = 1, \ldots, n$.
    
\end{itemize}

 Let $\{h_n\}$ be a sequence of elements in $G$ satisfying the following conditions:  
\begin{itemize}
    \item[(c)] $h_n$ acts trivially on $Y$, i.e., $h_n|_Y = \text{id}$ for all $n \geq 1$; 
    \item[(d)] for every clopen set $U$ containing $Y$, there exists $N$ such that for all $n \geq N$, $\supp(h_n) \subset U$.
\end{itemize}

Let $\pi$ be a  non-regular finite-type factor representation of $G$.  Then all limit points of the sequence $\{\pi(h_n)\}_{n=1}^\infty$ in the weak operator topology are scalar operators.
\end{theorem}
%
\noindent {\it Proof.}  The proof will be presented as a series of lemmas. Throughout, we assume that $Y$ and ${h_n}$ are as given in the statement of the theorem. 

Let $A$ be a limit point of the sequence ${\pi(h_n)}$. To simplify notation, we assume that ${\pi(h_n)}$ converges to $A$ in the weak operator topology. Suppose, for the sake of contradiction, that $A$ is not a scalar operator.

Denote by $\tr$ and $\chi$ the trace and character corresponding to the representation $\pi$. We will show that $\chi(g) = 0$ for every $g \in G \setminus {e}$, which implies that $\pi$ is quasi-equivalent to the regular representation (see Example \ref{ExampleRegularRepresentation}), leading to a contradiction.

Notice that the von Neumann algebra generated by $A$ has a nontrivial orthogonal projection \cite[14.1.3]{BhatElliottPeter:1999}. We will denote such a projection by $R_Y$.  

\begin{lemma}\label{LmRdependence}
If $g_1,g_2\in G$ are such that $g_1(Y)=g_2 (Y)$, then $\pi(g_1)R_Y\pi(g_1^{-1})=\pi(g_2)R_Y\pi(g_2^{-1})$.
\end{lemma}
\begin{proof} 

Let $h = g_2^{-1} g_1$. We aim to show that $\pi(h) R_Y = R_Y \pi(h)$.  
Since the set $Y$ intersects each $G$-orbit at most once, we conclude that $h|_Y = \text{id}$. Therefore, $h$ acts as the identity on a clopen neighborhood of $Y$. Consequently, $h$ commutes with $h_n$ for all sufficiently large $n$. Thus,  
$$
\pi(h) A = \lim\limits_{n} \pi(h) \pi(h_n) = \lim\limits_{n} \pi(h_n) \pi(h) = A \pi(h),
$$ 
where the limits are taken in the weak operator topology. Since the projection $R_Y$ is (weakly) approximated by linear combinations of powers of $A$, it follows that $\pi(h)$ commutes with $R_Y$.  \end{proof}

Given $Z \in \mathcal{Y}$, define the orthogonal projection  
$$
R_Z = \pi(g) R_Y \pi(g^{-1}),
$$ 
where $g$ is any element of $G$ such that $g(Y) = Z$. By Lemma \ref{LmRdependence}, the operator $R_Z$ is well-defined.

\begin{lemma}\label{LmRgCommute} If the sets $Z_1, Z_2\in \mathcal Y$ are disjoint, then the orthogonal projections $R_{Z_1}$ and $R_{Z_2}$ commute.
\end{lemma}
\begin{proof} Let $Z_1,Z_2\in \mathcal Y$. Fix $g_1,g_2\in G$ such that $g_i( Y)= Z_i$, $i=1,2$. Since each projection $R_{Z_i}$ is approximated by linear combinations of powers of $\pi(g_i)A\pi(g_i^{-1})$, it suffices to show that $\pi(g_1)A\pi(g_1^{-1})$ and $\pi(g_2)A\pi(g_2^{-1})$ commute.

  Using the fact that $Z_1$ and $Z_2$ are disjoint closed sets,  we can find $N>0$ such that  for every  $n,m\geq N$, we have that $$\supp(g_1h_ng_1^{-1})\cap \supp(g_2h_mg_2^{-1})=\emptyset.$$ 
  
  It implies that the group elements $g_1h_ng_1^{-1}$ and $g_2h_mg_2^{-1}$ commute. Hence, $\pi(g_1h_ng_1^{-1})$ and $\pi(g_2h_mg_2^{-1})$ also commute. Thus, for  $n\geq N$, the operators $\pi(g_1h_ng_1^{-1})$ and $\pi(g_2)A\pi(g_2^{-1})=\lim_{m\to\infty}\pi(g_2h_mg_2^{-1})$ commute. Taking the limit as $n\to\infty$, we obtain that
$\pi(g_1)A\pi(g_1^{-1})$ and $\pi(g_2)A\pi(g_2^{-1})$ commute.
\end{proof}
\noindent


We will call a subset $L\subset \mathcal Y$  {\it disjoint} if it consists of pairwise disjoint sets. 
Given two (finite or infinite) subsets $S\subset L\subset \mathcal Y$ with  $L$ being a nonempty {\it disjoint} set, introduce the operator $T_{L,S}$ by
\begin{equation}\label{EqTLS} T_{L,S}= \prod\limits_{W\in S}R_W\cdot \prod\limits_{Z\in L \setminus S}(\id - R_Z) = \prod_{\substack{W\in S \\ Z \in L\setminus S}}\left(R_W-R_W R_Z \right).
\end{equation}

In the case where one or both sets $L$ and $S$ are infinite, the definition of $T_{L,S}$ should be understood as the strong limit of finite products of projections. This limit exists and is well-defined since the operators $R_Z$, for $Z \in L$, are pairwise commuting orthogonal projections, see Lemma \ref{LemmaLimitProjections}. Since the projections in (\ref{EqTLS}) commute, the operator $T_{L,S}$ is also an orthogonal projection.  

Fix $g\in G$.  Since $\pi(g)R_Y \pi(g^{-1}) = R_{g(Y)}$, we obtain that 
 \begin{equation}\label{eqnTgSL}\pi(g)T_{L,S}\pi(g^{-1}) = T_{g(L),g(S)}, \end{equation}
 where $g(L) =\{g(Z) | Z\in L\}$. The set $g(S)$ is defined similarly. 

 We will say that a family $L\subset \mathcal Y$ {\it avoids a (non-empty)  clopen set $U$} if $Z\cap U = \emptyset$ for every $Z\in L$. To avoid trivialities, the set $U$ is assumed to be a proper clopen subset of $X$. In the following series of lemmas,  we will show that  $T_{L,L} =  T_{L,\emptyset} =0$ for every infinite disjoint family $L\subset \mathcal Y$ that avoids some non-empty clopen set.  Note that $${T_{L,L}}=\prod\limits_{Z\in L}R_Z \mbox{ and }{T_{L,\emptyset}}=\prod\limits_{Z\in L}(\id-R_Z).$$ 
 
 To simplify the notation, we will write $T_L$ for $T_{L,L}$, where $L\subset \mathcal Y$ is a disjoint family. 

\begin{lemma}\label{LemmaTL1TL2}  Let $L_1\subset \mathcal Y$ be an infinite disjoint set and $L_2\subset L_1$ be an infinite set. Then $T_{L_1} = T_{L_2}$.
\end{lemma}
\begin{proof}  
For each $i=1,2$, fix an exhausting increasing sequence of subsets $L_{i,n}\subset L_i$ with $\textrm{card} (L_{i,n}) = n$. By  Assumption (b) of  Theorem \ref{ThAutomaticContinuity}, for every $n\geq 1$, there exists an element $g\in G$   such that $g(L_{1,n})=L_{2,n}$. Using (\ref{eqnTgSL}), we obtain $$\tr(T_{L_{2,n}})=\tr(\pi(g)T_{L_{1,n}}\pi(g^{-1}))=\tr(T_{L_{1,n}}).$$ Taking the limit as $n\to\infty$ and applying  Lemma \ref{LemmaLimitProjections},  we conclude that  $\tr(T_{L_1}) = \tr(T_{L_2})$.  Combing it with the fact that $T_{L_1}<T_{L_2}$, we obtain that $T_{L_1} = T_{L_2}$. \end{proof}

Using Assumption (a) of Theorem \ref{ThAutomaticContinuity}, we can inductively construct an {\it everywhere dense disjoint family of sets $Q\subset \mathcal Y$}. 

\begin{lemma}\label{LemmaTLEqual0} Let $L\subset \mathcal Y$ be an infinite disjoint set that avoids a nonempty clopen set. Then $T_{L} = T_{Q}$. 
\end{lemma}
\begin{proof}  Let $U$ be a nonempty clopen set disjoint from the family $L$. Define $Q'$ as the collection of sets $Z\in Q$ that are contained in $U$.  Note that $Q'$ is an infinite set. Thus, by applying Lemma \ref{LemmaTL1TL2}, we obtain 
$$T_{L} = T_{L\cup Q'} = T_{Q'} = T_{Q}.$$
\end{proof}

\begin{lemma}\label{LemmaTLLTLEmpty=0} Let $L\subset \mathcal Y$ be a infinite disjoint set that avoids a nonempty clopen set. Then $T_{L,L} = T_{L,\emptyset} = 0$. 
\end{lemma}
\begin{proof}  In view of Lemma \ref{LemmaTLEqual0}, to show that $T_{L,L} = 0$, it suffices to prove that $T_Q = 0$. Fix an arbitrary group element $g\in G$, $g
\neq e$. Find a clopen set $U$ such that $g(U) \cap U = \emptyset$.   Define $Q'$ as the collection of sets $Z\in Q$ that are contained in $U$.  Note that $Q'$ is an infinite set.  Set $Q'' = g(Q')$. Note that $Q''$ is a disjoint subfamily of $\mathcal Y$ and the families $Q'$ and $Q''$ consist of pairwise disjoint sets. Using Equation (\ref{eqnTgSL}) and Lemma \ref{LemmaTL1TL2}, we obtain that 
$$\pi(g) T_Q \pi(g^{-1}) = \pi(g) T_{Q'} \pi(g^{-1}) = T_{Q''} = T_{Q'\cup Q''}  = T_{Q'} = T_Q.$$ 
It follows that the projection $T_Q$ belongs to the center of the factor $
\mathcal M_\pi$. Therefore,    $T_Q$ is a scalar operator, which means that it is either 0 or $\id$. If $T_Q = \id$, then $R_Y = \id$, which is a contradiction. Thus, $T_Q =0$.

The proof of the fact that  that $T_{L,\emptyset} = T_{Q,\emptyset}$ follows the similar argument as outlined in Lemma \ref{LemmaTL1TL2} and Lemma \ref{LemmaTLEqual0}. Then, using the argument above, one can  show that $T_{Q,\emptyset}$ belongs to the center of $\mathcal M_\pi$. Thus,  $T_{Q,\emptyset}$ is either $0$ or $\id$. If $T_{Q,\emptyset} = \id$, then we obtain that $R_Y = 0$, which is a contradiction. It follows that $T_{Q,\emptyset} = 0$, which completes the proof. 
\end{proof}

%
%


Now we are ready to complete the proof of Theorem \ref{ThAutomaticContinuity}.  Fix a group element $g\in G$, $g\neq e$.  Find a nonempty clopen set $U$ such that ${g(U)\cap U = \emptyset}$.   We can assume that the set $U$ is chosen to be small enough so that the complement of $U \cup g(U)$ is nonempty.

Define $L^{(1)}$ as the collection of sets $Z\in Q$ such that $Z\subset U$, and let $L^{(2)} = g(L^{(1)})$. Note that both families, $L^{(1)}$ and $L^{(2)}$, are infinite, disjoint from each other, and consist of pairwise disjoint sets. Moreover, the collection $L = L^{(1)}\cup L^{(2)}$ avoids the complement of $U \cup g(U)$. By Lemma \ref{LemmaTLLTLEmpty=0}, we have that 

\begin{equation}\label{eqnAuxTLL=0} T_{L,L} = 0 \mbox{ and } T_{L,\emptyset} = 0. 
\end{equation}

Fix an increasing sequence $\{L^{(1)}_n\}$ of sets, each of cardinality $n$, such that $\bigcup_{n\geq 1}L^{(1)}_n = L^{(1)} $.  For every $n\geq 1$, define $L^{(2)}_n= g(L^{(1)}_n)$ and set $L_n = L^{(1)}_n\cup L^{(2)}_n$.  

Observe that in view of Equation (\ref{EqTLS}) for any finite disjoint family $L'\subset \mathcal Y$, we have that 
\begin{equation}\label{EqSumTLS}\id = \prod_{Z\in L'}(R_Z + (\id - R_Z)) = \sum_{S\subset L'}\left(\prod\limits_{Z\in S}R_Z\cdot\prod\limits_{Z\in L' \setminus  S}(\id - R_Z) \right) = \sum\limits_{S\subset L'}T_{L',S}.\end{equation}

Fix $n\in\mathbb N$ and a subset $S \subset L_n$. We claim that if the set $S$ contains exactly one element from the pair $\{Z,g(Z)\}$ for some $Z \in L_n^{(1)}$,  then we have $T_{L_n,S}T_{g(L_n),g(S)} = 0$.

Indeed, assume that $Z \in L_n^{(1)}$ is an element of $S$, but  $g(Z)$ is not. By definition, $T_{L_n,S}$ is a product of pairwise commuting orthogonal projections. Since $g(Z) \in L_n \setminus S$, this product includes the projection $\id - R_{g(Z)}$. On the other hand, since $g(Z) \in g(S)$, the product of projections that define $T_{g(L_n),g(S)}$ includes the projection $R_{g(Z)}$.  

While the projections $T_{L_n,S}$ and $T_{g(L_n),g(S)}$ do not necessarily commute, the term $\id - R_{g(Z)}$ in $T_{L_n,S}$ can be repositioned to appear as the final factor, while the term $R_{g(Z)}$ in $T_{g(L_n),g(S)}$ can be placed as the leading factor in the product. This rearrangement ensures that
$
T_{L_n,S}T_{g(L_n),g(S)} = 0.
$

A similar argument shows that $T_{L_n,S}T_{g(L_n),g(S)} = 0$ whenever $S$ contains $g(Z)$, for some  $Z \in L_n^{(1)}$, but does not contain $Z$. 

Let $\mathcal P_n$ be the set of subsets $S \subset L_n$ such that, for every $Z \in L_n^{(1)}$, the set $S$ either contains both $Z$ and $g(Z)$ or contains neither of them.  Note that if $S\subset L_n$ and $S\notin \mathcal P_n$, then
\begin{equation}\label{eqnAutContAuxTLN=0} T_{L_n,S}T_{g(L_n),g(S)} = 0.
\end{equation}

We note that each set $S\in \mathcal P_n$ has an even number of elements and that there are ${n \choose k}$ subsets of cardinality $2k$ in $\mathcal P_n$. 

From \eqref{eqnTgSL}, \eqref{EqSumTLS}, and \eqref{eqnAutContAuxTLN=0}, the centrality of $\tr$, and the fact that $T_{L_n,S}$ is an orthogonal projection, we obtain that 
\begin{equation}\label{EqChiSumTLS}\begin{split}\chi(g) &=\sum\limits_{S\subset L_n}\tr(\pi(g)T_{L_n,S})  =\sum\limits_{S\subset L_n}\tr(T_{L_n,S}\pi(g)T_{L_n,S}) \\ 
& =\sum\limits_{S\subset L_n}\tr(T_{L_n,S}T_{g(L_n),g(S)}\pi(g)) 
\\ 
& = \sum\limits_{S\in \mathcal P_n }\tr(\pi(g)T_{L_n,S}).
\end{split}\end{equation} 

Note that 
\begin{equation}\label{eqnTraceTLSIneq}|\tr(\pi(g)T_{L_n,S})|\leqslant \tr(T_{L_n,S})\mbox{ for any }S\subset L_n.\end{equation}

 Assumption (b) of Theorem \ref{ThAutomaticContinuity} implies that given any two sets $S_1,S_2\subset L_n$ of the same cardinality, there is a group element $h\in G$ such that $h(L_n) = L_n$ and $h(S_1)=S_2$. It follows from (\ref{eqnTgSL})  that $\pi(h) T_{L_n,S_1}\pi(h^{-1}) = T_{L_n,S_2}$ and, thus, $\tr(T_{L_n,S_1}) = \tr(T_{L_n,S_2})$.   Combing this observation with  \eqref{EqChiSumTLS} and  \eqref{eqnTraceTLSIneq}, we obtain that
\begin{equation}\label{eqnContAuxBoundChi}|\chi(g)| \leq  \sum\limits_{S\in \mathcal P_n }|\tr(\pi(g)T_{L_n,S})| \leq \sum\limits_{S\in \mathcal P_n}\tr(T_{L_n,S})  = \sum_{k=0}^n {n\choose k} \tr(T_{L_n,S_{2k}}),\end{equation}
where $S_{i}$ stands for an arbitrary subset of $L_n$  of cardinality $i$. 

Using again the fact that  $\tr(T_{L_n,S'}) = \tr(T_{L_n,S''})$ for any subsets $S',S''\subset L_n$ of the same cardinality and applying (\ref{EqSumTLS}) to $L=L_n$, we obtain that
\begin{equation}\label{eqnContAuxIneqTr} 1 = \sum_{S\subset L_n}\tr(T_{L_n,S}) = \sum_{j=0}^{2n} {2n \choose j}\tr(T_{L_n,S_j}) .\end{equation} 

Using the factorial identities  $j! = j!! (j-1)!!$ and $(2j)!! = 2^j j!$,  for every $1\leqslant k\leqslant n-1$ , we obtain the inequality 
\begin{equation}\label{eqnContAuxBinomIneq}{2n \choose 2k}={n \choose k}\frac{(2n-1)!!}{(2k-1)!!(2n-2k-1)!!}\geqslant (2n-1){n \choose k}.
\end{equation}

%

Thus, using (\ref{eqnContAuxBoundChi}), (\ref{eqnContAuxIneqTr}),  and (\ref{eqnContAuxBinomIneq}),  we obtain that
\begin{equation*}
\begin{split}
 |\chi(g)| \leq & \tr(T_{L_n,\emptyset}) + \sum_{k=1}^{n-1} {n\choose k} \tr(T_{L_n,S_{2k}}) +  \tr(T_{L_n,L_n})  \\ 
 \leq & \tr(T_{L_n,\emptyset}) + \frac{1}{2n-1}\sum_{k=1}^{n-1} {2n \choose 2k}\tr(T_{L_n,S_{2k}}) +  \tr(T_{L_n,L_n}) \\
 \leq & \tr(T_{L_n,\emptyset}) + \frac{1}{2n-1}\sum_{j=0}^{2n} {2n \choose j}\tr(T_{L_n,S_{j}}) +  \tr(T_{L_n,L_n}) \\
  = & \tr(T_{L_n,\emptyset}) + \frac{1}{2n-1}+  \tr(T_{L_n,L_n}).
 \end{split}
\end{equation*}
Note that in view of \eqref{eqnAuxTLL=0} and Lemma \ref{LemmaLimitProjections}, the right-hand side of this inequality converges to zero when $n\to \infty$. 

Thus, we have shown that $\chi(g)=0$ for every $g\in G\setminus \{e\}$. It follows that $\chi$ is the regular character and $\pi$ is the regular representation. This contradiction completes the proof.
$\hfill\square$


 %

\section{Approximating Properties of Full Groups of Bratteli Diagrams}\label{SectionFullGroupsBratteliDiagrams}

Let $(X,G)$ be a Cantor dynamical system.  Denote by $\mathcal M(X,G)$ the set of $G$-invariant probability measures on $X$.  Denote by $\mathcal E(X,G)$ the set of $G$-invariant probability ergodic measures on $X$.

We will need the following result, see,  for example, in  \cite[Proposition 2.3]{BezuglyiMedynets:2008}.

\begin{proposition}\label{PropositionContinuityMeasures} Let $(X,G)$ be a Cantor minimal system. (1) For every $\varepsilon>0$, there exists $\delta>0$ such that for any clopen set $U$ with $\textrm{diam}(U)<\delta$, we have that $\sup_{\mu\in \mathcal M(X,G)}\mu(U)<\varepsilon$. (2) For every non-empty clopen set $U$, we have that $\inf_{\mu\in \mathcal M(X,G)}\mu(U)>0$.
\end{proposition}

\begin{definition} We say that a dynamical system $(X,G)$ has {\it comparison} if given any clopen sets $A$ and $B$, the condition $\mu(A)<\mu(B)$, for all $\mu\in \mathcal M(X,G)$, implies that there is a clopen partition  $\{A_1,\ldots,A_n\}$ of $A$ and a collection $\{g_1,\ldots,g_n\}$ of elements of $G$ such that the sets $\{g_i(A_i)\}_{i=1}^n$ are pairwise disjoint and contained in $B$.
\end{definition}

The following result follows from  \cite[Lemma 6.7]{Matui:2012} and  \cite[Lemma 5.5.7]{Nekrashevych:Book2022}. For the full groups of simple Bratteli diagrams and the full groups of minimal $\mathbb Z$-systems, the result almost immediately follows from Lemma 2.5 in \cite{GlasnerWeiss-Equivalence-95}. 

 \begin{proposition}\label{PropositionComparisonProperty} If a dynamical system $(X,G)$ is minimal and almost finite, then for any clopen sets $A$ and $B$  with $\mu(A)<\mu(B)$ for every $\mu\in \mathcal M(X,G)$, there is an involution $g\in \mathcal A(G)$ such that $g(A)\subset B$ and $\supp(g)\subset A\cup B$.  In particular, the dynamical system $(X,G)$  has comparison. 
\end{proposition}

\begin{definition} Let $(X,G)$  be a Cantor dynamical system. We say that two clopen sets $A$ and $B$ are {\it $G$-equivalent} if there is an element $g\in G$ such that $g(A) = B$. 
\end{definition}

The following proposition is an adaptation of the  Rokhlin lemma from \cite[Theorem 2]{BezuglyiDooleyMedynets:2005}.  

%
%

\begin{proposition}\label{PropositionRokhlinLemma} Let $G$ be the full group of a simple Bratteli diagram with path-space $X$ and let $A$ be a clopen set. 

(i)  For any $\varepsilon>0$ and any positive integer $m\geq 2$, the exists a clopen set $B\subset A$ and an element $g\in G$ such that the sets $\{B,g(B),\ldots, g^{m-1}(B)\}$ are disjoint subsets of $A$ and 
$$\mu\left(\bigcup_{i=0}^{m-1}g^i(B)\right)>\mu(A)-\varepsilon\mbox{ for every }\mu \in \mathcal M(X,G).$$ 

(ii)  In particular, there exists an increasing sequence of clopen subsets $A_n$  of $A$ such that $\mu(A\setminus A_n)\leq \frac{1}{n}$ for every $n\geq 1$ and $\mu\in \mathcal M(X,G)$ and such that $A_n = A_n'\sqcup A_n''$, where  $A_n'$ and $A_n''$ are clopen $G$-equivalent sets. 
\end{proposition}
\begin{proof} (i) Without loss of generality, we can assume that $A=X$.  Fix an integer $k$ such that $1/k<\varepsilon.$ Choose $n$ large enough so that the towers of the Kakutani-Rokhlin partition $\Xi_n$ have heights at least $mk$, see the definition in Equation \ref{eqnXin}. It follows for every measure $\mu \in \mathcal M(X,G)$ that
 $$\mu\left(\bigcup_{v\in V_n}\bigcup_{i=h_v^{(n)}-m}^{h_v^{(n)}-1}C_{v,i}^{(n)}\right)<\frac{m}{mk}<\varepsilon.$$  
 
 Rewrite $h_v^{(n)} = m d_v+ r_v$, where $d_v\geq 0$ is the integer part and $r_v$ is the remainder of $h_v^{(n)}$ when divided by $m$. Then, the set $B$ is defined as the union of every $m$-th level in each tower taken over all towers from $V_n$. More specifically, 
  $$B = \bigcup_{v\in V_n}\bigcup_{i=0}^{d_v-1} C_{v,i\cdot m}^{(n)} . $$
  
  Let $g\in G_n$ be the permutation that cyclically permutes atoms within each tower of $\Xi_n$, that is,  $g(C_{v,i}^{(n)}) = C_{v,i+1}^{(n)}$ and $g(C_{v,h_v^{(n)}}^{(n)}) = C_{v,0}^{(n)}$  for every vertex $v$.     It follows that the sets $\{B,g(B),\ldots, g^{n-1}(B)\}$ are disjoint and their union cover the entire space up to a clopen set of measure less than $\varepsilon$ for any $G$-invariant measure. 
  
  (ii) Inductively applying Statement (i) to $m=2$, we can find a disjoint sequence  $\{B_n\}$ of clopen subsets of $A$ such that $B_n = B_n'\sqcup B_n''$ for some   $G$-equivalent clopen sets $B_n'$ and $B_n''$ and such that $\mu(A\setminus (B_1\sqcup \ldots \sqcup B_n))<1/n$  for every $\mu \in \mathcal M(X,G)$. Setting $A_n = B_1\sqcup\ldots \sqcup B_n$, we obtain the result.
\end{proof}

Let $G$ be the full group of a simple Bratteli diagram with path-space $X$.  Note that  $G = \bigcup_{n =  1}^\infty G_n$, where $G_n$ is the group of permutations of the initial $n$-segments of the paths from $X$.  For a clopen set $A$, denote by $G_A$ the (local) subgroup consisting of homeomorphisms from $G$ supported by the set $A$.    We also define $G_{A,n} = G_A\cap G_n$. 
 
 One of the defining features of full groups associated with simple Bratteli diagrams is that the conjugacy class of every element ``asymptotically  converges'' to the ``center'' of the group.  The following proposition makes this observation precise for conjugacy classes of involutions.

\begin{proposition}\label{PropositionInvolutionConvergeToCenter} Let $G$ be the  full group of a simple Bratteli diagram with path-space  $X$, $A\subset X$  a clopen subset, and $x\in A$.  For any  involution $h \in G$ with $\supp(h)\subsetneq A\setminus \{x\}$,  there exist two sequences of involutions $\{t_n\}$ and $\{h_n\}$ in $G_A$ such that $\supp(t_n)\to\{x\}$ and for every $n\geq 1$ the element $t_nh$  is conjugate to $h_n$ by an element from $G_A$, the element $h_n$  commutes with $G_{A,n}$, $\supp(h)\cap \supp(t_n) = \emptyset$, and $x\notin \supp(t_n)$.  

Additionally, if $h\in G_A'$, then the elements $\{t_n\}$ and $\{h_n\}$ can be chosen from $G_A'$ and the elements $t_nh$ and $h_n$ are conjugate inside $G_A'$. 
\end{proposition}
\begin{proof} Without loss of generality we may assume  that $A = X$. Since the Bratteli diagram is simple, we may telescope it, if necessary, so that any two vertices from consecutive levels  are connected by at least two edges. Denote by $v_0$ the root vertex of the Bratteli diagram. 

For a clopen set $C$ and a finite collection of vertices $v_1, \ldots, v_m$ from an increasing sequence of levels, denote by $K_{C, v_1, \ldots, v_m}$ the number of finite paths $p$ starting at $v_1$, passing through $v_2, \ldots, v_{m-1}$, and terminating at $v_m$, such that the corresponding clopen set $U(p)$ (the set of infinite paths containing $p$) is contained in $C$. We will omit the set $C$ from the subscript when $C = X$. 

For each $n\geq 1$, denote by  $q_n$ the path of length $n$ consisting of the first $n$ edges of $x = (x_0,x_1,x_2,\ldots)$. In particular, $x \in U(q_n)$ for every $n\geq 1$. Note that for all $n$ large enough, we have that $U(q_{n}) \cap \supp(h) = \emptyset$.

For each $n\geqslant 2$ and $v\in V_n$, choose a finite path $p_{n,v}$ obtained from $q_{n-1}$ by adding an edge, distinct from $x_n$, that connects  $q_{n-1}$ to the  vertex $v$.

Fix $n \in \mathbb{N}$. Without  loss of generality, we can assume that $n$  is chosen large enough so that  $h\in G_n$. In other words, $h$ acts on finite segments between the root vertex $v_0$ and vertices of level $V_n$. Choose $l>n$  so that 
$$K_{v,w}\geqslant 2K_{v_0,v}\mbox{ for every }v\in V_n\mbox{ and }w\in V_l.$$

Elements  $t_n$ and $h_n$ that appear in the statement of the result will act as permutations of finite paths lying between levels $V_n$ and $V_l$.

Note $K_{\supp(h),v_0,v}$ is even for every $v\in V_n$.   
Thus, for every $v\in V_n$ and  $w\in V_l$, we can find an integer  $0\leqslant N_{v,w}<K_{v_0,v}$  such that $K_{\supp(h),v_0,v}K_{v,w}+2N_{v,w}$ is divisible by $2K_{v_0,v}$.  

 For each $v\in V_n$ and  $w\in V_l$, fix any $2N_{v,w}$ distinct finite paths, denoted by   $z_0,\ldots,z_{2N_{v,w}-1}$, joining the vertices $v$ and $w$. Define an involution  $t_n$ as
  $$t_n(p_{n,v}z_i)=\left\{\begin{array}{ll}p_{n,v}z_{i+N_{v,w}},& i<N_{v,w},\; v\in V_n,\\
  p_{n,v}z_{i-N_{v,w}},& N_{v,w}\leqslant i< 2N_{v,w},\; v\in V_n\end{array}\right. $$ and trivially elsewhere. 
 Observe that by construction 
 \begin{equation*}\label{EqSuppTn}\supp(t_n)\subset U(q_{n-1})\setminus U(q_n),\;\;K_{\supp(t_n),v,w}=2N_{v,w}\;\;\text{for all}\;\;v\in V_n,w\in V_l.\end{equation*} In particular, $\supp(t_n)\cap\supp(h)=\emptyset$ and $x\notin\supp(t_n)$ for  large enough $n$.

Now, to construct $h_n$, for every $v\in V_n$ and $w\in V_l$, choose 
\begin{equation}\label{eqnMVW}2M_{v,w}=(K_{\supp(h),v_0,v}K_{v,w}+2N_{v,w})/K_{v_0,v}\end{equation} 
paths $r_0,\ldots,r_{2M_{v,w}-1}$ joining $v$ and $w$. Define $h_n$ as follows: for a path of the form $ar_i$, where $a$ is any path joining $v_0$ and $v\in V_n$, we set $$h_n(ar_i)=\left\{\begin{array}{ll}ar_{i+M_{v,w}},& i<M_{v,w},\\
  ar_{i-M_{v,w}},& M_{v,w}\leqslant i< 2M_{v,w},\end{array}\right.$$ and we define $h_n$ trivially eslewhere.

 Note that, by construction, $h_n$ commutes with $G_n$.  Our next goal is to establish that $t_nh$ and $h_n$ are conjugate.   Observe that two involutions  $g_1,g_2\in G $ are conjugate if and only if one can find  $m\in \mathbb N$ such that $g_1,g_2\in G_m$ and such that $K_{\supp(g_1),v_0,v}=K_{\supp(g_2),v_0,v}$ for every vertex $v\in V_m$.
 
 Since $\supp(h)\cap \supp(t_n) = \emptyset$, we have that $$K_{\supp(t_nh),v_0,w} = K_{\supp(h),v_0,w}+K_{\supp(t_n),v_0,w}\mbox{ for every }w\in V_l.$$ 
  Thus, if we obtain that 
  \begin{equation}\label{eqnAuxHnTnHCong}K_{\supp(h),v_0,w}+K_{\supp(t_n),v_0,w}=K_{\supp(h_n),v_0,w}\mbox{ for every }w\in V_l,\end{equation}  we establish that $h_n$ and $t_n h$ are conjugate.
Indeed, for every $w \in V_l$, we have:
\begin{align*} 
K_{\supp(h), v_0, w} &=  \sum_{v \in V_n} K_{\supp(h), v_0, v} \, K_{v, w}, \\
K_{\supp(t_n), v_0, w} &= \sum_{v \in V_n}  K_{\supp(t_n), v, w}  = \sum_{v \in V_n}  2N_{v,w} \\
K_{\supp(h_n), v_0, w} &= \sum_{v \in V_n} K_{v_0, v} \, K_{\supp(h_n), v, w} = \sum_{v \in V_n} 2 K_{v_0, v} M_{v,w}.
\end{align*}

Using (\ref{eqnMVW}), we obtain that $$K_{v_0, v} \, K_{\supp(h_n), v, w} = K_{\supp(h),v_0,v}K_{v,w}+2N_{v,w}\mbox{ for every }v\in V_n\mbox{ and }w\in V_l,$$ which  implies (\ref{eqnAuxHnTnHCong}).  

If $h\in G_A'$, then in the proof above we can choose $n$ large enough so that $K_{\supp(h), v_0, v}$ is divisible by 4 for every $v\in V_n$. Additionally, we can choose $l\geq n$ so that  $K_{v,w}\geqslant 4K_{v_0,v}$ for every $v\in V_n$ and $w\in V_l$. Then ensuring $N_{v,w}$ in the proof above is divisible by 2, we obtain that the involution $t_n\in G_A'$. It follows that $h_n\in G_A'$. Applying the Splitting Criterion for Conjugacy Classes in the Alternating Group, we obtain that $t_nh$ and $h_n$ are conjugate by an element from $G_A'$. 
\end{proof}

 Let $G$ be the full group of a simple Bratteli diagram with path-space $X$.   Introduce a  $G$-invariant metric $d$ on the Boolean algebra of clopen sets $\textrm{CO}(X)$ by
  \begin{equation}\label{EqDistanceOfSets}d(A,B) = \sup\{\mu(A\triangle B) : \mu\in\mathcal M(X,G)\}.\end{equation}  
 By the Ergodic Decomposition Theorem \cite[Theorem 9.5]{DoughertyJacksonKechris:1994}, every invariant measure can be represented as an integral of ergodic measures. Therefore, 
  \begin{equation}\label{EqErgodicDistanceOfSets}d(A,B) = \sup\{\mu(A\triangle B) : \mu\in\mathcal E(X,G)\}.\end{equation}  

We will need the following two results, which can be viewed as a manifestation of mixing properties of the system $(X,G)$. 

\begin{proposition}\label{Proposition_B_to_Invariant_Bn} Let $G = \bigcup_{n=1}G_n$ be the full group of a simple Bratteli diagram, where $G_n$ is the permutation group corresponding to the level $n$ of the diagram. Then for every clopen set $B$ there exist a sequence $\{B_n\}$ of clopen sets and a sequence $\{g_n\}$ of elements of $G$  such that for every $n\geq 1$ the set $B_n$ is $G_n$-invariant and such that $d(B_n,g_n(B))\to 0$ as $n\to \infty$. 
\end{proposition}
\begin{proof} (1)  Using Proposition \ref{PropositionRokhlinLemma} find a sequence of involutions $\{q_n\}$ supported by $B$ such that $d(B,\supp(q_m))\to 0$ as $m\to\infty$. Applying Proposition \ref{PropositionInvolutionConvergeToCenter}, for each $q_m$ find two sequences of involutions $\{h_{m,n}\}_{n=1}^\infty$  and $\{t_{m,n}\}_{n=1}^\infty$ such that  for every $m\geq 1$ and $n\geq 1$, $\supp(q_m)\cap \supp(t_{m,n})=\emptyset$, $h_{m,n}$ commutes with $G_n$ and $t_{m,n}q_m$ is conjugate to $h_{m,n}$ by an element $c_{m,n}\in G$, and such that for every $m\geq 1$, $\textrm{diam}(\supp(t_{m,n}))\to 0$ as $n\to\infty$. 

Choose a subsequence $\{n_m\}_{m=1}^\infty$ such that $\textrm{diam}(\supp(t_{m,n_m}))\to 0$. It follows that $d(B,\supp(t_{m,n_m}q_m))\to 0$ as $m\to \infty $. Set $B_m = \supp(h_{m,n_m})$. Then $d(c_{m,n_m}(B),B_m)\to 0$ as $m\to\infty$.

 Since $h_{m,n_m}$ commutes with $G_{n_m}$, the set $B_m$ is $G_{n_m}$-invariant.  Reindexing the sequence $\{n_m\}$ if necessary, we obtain the result. 
\end{proof}

\begin{proposition}\label{PropositionMovingSetClose}  Let $G$ be the full group of a simple Bratteli diagram with path-space $X$. 

(1) Let $A\subset X$ be a clopen set, $0 \leq \lambda \leq  1$ and $\varepsilon>0$. Then there exists a clopen set $A'\subset A$ such that 
$$\lambda \mu(A) - \varepsilon < \mu(A')\leq \lambda \mu(A)\mbox{ for every }\mu \in \mathcal M(X,G). $$

(2) Let  $\varepsilon>0$ and let $A,B$ be nonempty clopen sets such that  $$|\mu(A)-\mu(B)|<\varepsilon\mbox{ for every }\mu\in \mathcal M (X,G).$$  Then there exists $g\in G$ such that such that $d(g(A),B)<4\varepsilon$. 
\end{proposition}
\begin{proof} 
(1)  For each $v\in V_n$, denote by $B_v$ the base of the corresponding Kakutani-Rokhlin tower.  For each vertex $v\in V_n$, denote by $K_{A,v}$  the number of paths $p$ connecting the root vertex $v_0$ to the vertex $v$ such that $U(p)\subset A$. Choose $n$ large enough so that $A$ is a union of cylinder sets of level $n$ and so that $$\sum\limits_{v\in V_n}\mu(B_v)<\varepsilon\mbox{ for every }\mu \in \mathcal M(X,G).$$ 

For each $v \in V_n$ choose a collection $C_v$ of $\lfloor\lambda K_{A,v}\rfloor$ (the integer part of $\lambda K_{A,v}$)  paths $p$ connecting $v_0$ to $v$ with $U(p)\subset A$.  Set
$$A' =\bigcup\limits_{v\in V_n}\bigcup\limits_{p\in C_v}U(p).$$ Then for each $\mu \in \mathcal M(X,G)$, we have that 
\begin{align*} \lambda \mu(A) & =\sum\limits_{v\in V_n}\lambda K_{A,v}\mu(B_v) \geq \sum\limits_{v\in V_n} |C_v| \mu(B_v) = \mu(A') \\
& \geq \sum\limits_{v\in V_n}(\lambda K_{A,v}-1)\mu(B_v) \geq \lambda \mu(A) - \varepsilon, 
\end{align*}
which proves the result.

(2) Let $A,B$ be clopen sets as in the statement of the proposition. Applying Statement (1) with $\lambda =  1-\frac{\varepsilon}{\mu(A)}$, we can find a clopen set $A'\subset A$ such that $$\mu(B)-3\varepsilon < \mu(A)\left[ 1-\frac{\varepsilon}{\mu(A)} \right] - \varepsilon <  \mu(A')\leq \mu(A)\left[ 1-\frac{\varepsilon}{\mu(A)} \right]<\mu(B),$$ for every $\mu \in \mathcal M(X,G)$. By Proposition \ref{PropositionComparisonProperty}, we can find an element $g\in G$ such that $g(A')\subset B$. It follows that $d(g(A),B)<4\varepsilon$.

\end{proof}

%


 %
 %

%
%

\section{Projections onto Subspaces Invariant under Local Subgroups}\label{Section:Projections}

Let $G$ be the full group associated with a simple Bratteli diagram with path-space $X$. Recall that the commutator subgroup $G’$, as well as all local subgroups $G_A’$ supported on clopen sets, are simple, Proposition \ref{PropositionSimplicityCommutatorSubgroup}.

Consider a nontrivial factor representation $(\mathcal{H}, \pi, \xi)$ of the commutator subgroup $G’$. Denote by $\mathcal{M}_\pi$ the von Neumann algebra generated by $\pi(G’)$. For every clopen set $A$, let $P^A$ denote the orthogonal projection onto the subspace of $G_A’$-invariant vectors. Note that $P^\emptyset = \id$ and, since $\mathcal{M}_\pi$ is a nontrivial factor, we have $P^X = 0$.

It follows from the following classical result that for every clopen set $A$, the projection $P^A$ belongs to $\mathcal{M}_\pi$. The proof of the lemma follows the same argument as in \cite[Lemma 20]{DudkoGrigorchuk-DiagonalActions-18}.

\begin{lemma}\label{lemmaProjBelongsToAlgebra} Let $\mathcal M$ be a von Neumann algebra realized on some Hilbert space $\mathcal H$. Let $\mathcal S\subset \mathcal M$ and $\mathcal H_1 = \{v\in \mathcal H : Qv=v\mbox{ for }Q\in \mathcal S\}$. Then the orthogonal projection onto $\mathcal H_1$ belongs to $\mathcal M$. 
\end{lemma}

The following theorem is the main result of this section. It establishes that the value of an indecomposable character on a group element $g$ depends only on the support of $g$. One of the key intermediate steps in the proof involves demonstrating that the orthogonal projections ${P^A}$ form a commutative semigroup that is normalized by the unitary operators in $\pi(G)$.

%
%

\begin{theorem}\label{TheoremCharFromTrace} Let $G$ be the full group of a simple Bratteli diagram with path-space $X$. Let $\chi$ be an indecomposable character of the commutator subgroup $G’$, and let $(\mathcal{H}, \pi, \xi)$ be the associated unitary representation. Let $\mathcal{M}_\pi$ the $\text{II}_1$-factor generated by $\pi(G’)$, and  $\tr$ be the corresponding trace on $\mathcal{M}_\pi$. Then, for every $g \in G’$, we have
\[
\chi(g) = \tr(P^{\textrm{supp}(g)}).
\]

\end{theorem}
{\it Proof.} The proof will consist of a series of lemmas. Fix an increasing sequence of finite subgroups $G_n$  such that $G = \bigcup_n G_n$. For a clopen set $A$, set $G_{A,n}'=G_n\cap G_A'$.  Using the standard arguments, one can show that $$P_n^A = \frac{1}{|G_{A,n}'|}\sum_{h\in G_{A,n}'} \pi(h)$$ is the orthogonal projection onto the set of $G_{A,n}'$-invariant vectors. Furthermore, $\{P^A_n\}$ is a monotone decreasing sequence of projections.  By Lemma \ref{LemmaLimitProjections}, the strong (weak)  limit of $\{P^A_n\}$ exists and coincides with $P^A$.

If  $\chi$ is the regular character, then $\mathcal M_\pi$ is quasi-equivalent to the regular representation. Thus, without loss of generality, we may assume that   $
\mathcal H = l_2(G)$ and $\pi(g)f(h) = f(g^{-1}h)$ for $f\in l_2(G)$.  Therefore, for every  clopen set $A$, we have that $$(P^A\delta_g,\delta_g) = \lim_{n\to\infty} (P_n^A \delta_g,\delta_g) =  \lim_{n\to\infty}\frac{1}{|G_{A,n}'|}\sum_{h\in G_{A,n}'} (\pi(h)\delta_g,\delta_g)  =  \lim_{n\to\infty}\frac{1}{|G_{A,n}'|},$$ which converges to zero whenever the set $A$ is non-empty.  It follows that  $P^A\delta_g = 0$ for every $g\in G'$ and every non-empty clopen set $A$, which implies that $P^A = 0$ for every non-empty clopen set $A$. Therefore, the formula $$\chi(g) = \tr(P^{\supp(g)})$$ holds trivially for regular characters.

From now on, we will assume that the character $\chi$ and the representation $\pi$ are non-regular and nontrivial.  Our first task is to find an alternative description of projections $\{P^A\}$.  For a clopen set $A$, denote by $\mathcal S(A)$ the set of sequences $\{h_n\}\subset G_A'$ such that $h_n^2=1$,  $h_n$ commutes with $G_{A,n}'$ for every $n\geq 1$, and the sequence $\{\pi(h_n)\}$ converges in the weak operator topology.  Denote by $\mathcal L(A)$ the set of all weak limits of sequences from  $\mathcal S(A)$, that is, 
\begin{equation}\label{EqLA}\begin{split} \mathcal L(A)=\{Q \in B(\mathcal H): Q = \lim_n\pi(h_n) \mbox{ for some } \{h_n\}\in \mathcal S(A)\}.
\end{split}\end{equation}

\begin{lemma}\label{lemmaHnInvarSubsets} Let $\{h_n\} \in \mathcal S(A)$. Then for any clopen set $B$, we have that  $h_n(B) = B$ for all $n$ large enough.
\end{lemma}
\begin{proof} Note that elements $\{h_n\}$ pointwise stabilize the set $X\setminus A$. Thus,  in order to prove the result it is enough to show that $h_n(B\cap A) = B\cap A$ for all $n$ large enough. Set $C = B\cap A$. 

Assuming $C\neq \emptyset$ and using the minimality of $(X,G')$, we can find $n_0\geq 1$ such that for every $n\geq n_0$ and  for every $x\in C$ there exists $g\in G_{C,n}'$ with $g(x)\neq x$ and $\supp(g)\subset C$. Assume towards contradiction that $h_{n}(C)\neq C$ for some $n\geq n_0$. 
Since every $G'$-invariant measure is fully supported,  no element of $G'$ can map a clopen set into a proper subset. Thus, there exists $x\in C$ such that $h_{n}(x) \in X\setminus C$. Fix $g\in G_{C,n}'$  with $g(x) \neq x$. It follows that $h_ng(x) \neq h_n(x) = gh_n(x)$, which  contradicts the assumption that $h_n$ commutes with $G_{A,n}'$. 
\end{proof}

\begin{lemma}\label{LmLAinv} The set of operators $\mathcal L(A)$ is commutative and invariant under multiplication, \ie for any $Q_1,Q_2\in \mathcal L(A)$ one has $Q_1Q_2 = Q_2Q_1\in\mathcal L(A)$. Furthermore, every operator in $\mathcal L(A)$ commutes with $\pi(G_A')$.
\end{lemma}
\begin{proof} Notice that if $\{h_n\}_{n\geqslant 1}\in\mathcal S(A)$, then, by definition,  $h_n$ commutes with $G_{A,n}'$ for every $n$. Since every operator  $Q\in\mathcal L(A)$ is the limit of a sequence from $\mathcal S(A)$, we obtain that  $Q$ commutes with $\pi(G_A')$. 

Fix $Q_1,Q_2\in \mathcal L(A)$. Find two sequences $\{h_n^{(i)}\}_{n\geqslant 1}\in\mathcal S(A)$, $i=1,2$, such that $Q_i = \lim \pi(h_n^{(i)})$.  Since multiplication is separately continuous in the weak operator topology, there exist  two increasing sequences $\{k_n\}_{n\geqslant 1}$ and $\{m_n\}_{n\geqslant 1}$ of positive integers such that $Q_1Q_2$  is the limit of the sequence $\{\pi(h_{k_n}^{(1)}h_{m_n}^{(2)})\}$. Additionally, we can take $m_n>k_n$ large enough to ensure that for every $n\geq 1$ the elements $h_{k_n}^{(1)}$ and $h_{m_n}^{(2)}$ commute. It follows that $Q_1Q_2 = Q_2Q_1$ and $(h_{k_n}^{(1)}h_{m_n}^{(2)})^2= (h_{k_n}^{(1)})^2 (h_{m_n}^{(2)})^2= 1$. Thus, $\{h_{k_n}^{(1)}h_{m_n}^{(2)}\} $ belongs to $\mathcal S(A)$, which implies that $Q_1Q_2\in \mathcal L(A)$.
\end{proof}

 Set  
$$\mathcal H_A' = \{ v \in \mathcal H : Qv = v \mbox{ for every }Q\in \mathcal L(A)\}.$$

Denote by $Q^A$ the orthogonal projection onto $\mathcal H_A'$. In view of Lemma \ref{lemmaProjBelongsToAlgebra}, we obtain that $Q^A\in \mathcal M_\pi$.

\begin{lemma}\label{LemmaAuxImageSubspace}  For $g\in G'$, we have that $\pi(g)Q^A  \pi(g^{-1}) = Q^{g(A)}$. In particular, $\pi(g)Q^A = Q^A\pi(g)$ for every $g\in G_A'$. 
\end{lemma} 
\begin{proof}  Note that $\pi(g)Q^A  \pi(g^{-1})$ is the projection onto $\pi(g)\mathcal H_A'$, which consists of the vectors invariant under the operators $\pi(g)\mathcal L(A) \pi(g)^{-1}$.  Note that $\pi(g)\mathcal L(A) \pi(g)^{-1} = \mathcal L(g(A))$. It follows that $\pi(g)\mathcal H_A' = \mathcal H_{g(A)}'$, which completes the proof. 
\end{proof}


\begin{lemma}\label{LemmaAugProjectionAbsorption}  Let $A$ be a clopen set.   Then  $\pi(g) Q^A = Q^A$ for every $g\in G_A'$.  
\end{lemma}
\begin{proof} (I)  If $Q^A = 0$, then the result holds trivially.  Consider the case when $Q^A \neq 0$.  Fix an arbitrary element $h\in G_A'$ such that $h^2 = 1$, $h\neq 1$, and $\supp(h)\subsetneq A$.  

 By Proposition \ref{PropositionInvolutionConvergeToCenter},  for every $n\geq 1$ there exists an element  $h_n$ that commutes with $G_{A,n}'$  and that is ``almost'' conjugate to $h$. More precisely, for every $x\in A\setminus\supp(h)$, there exists a sequence $\{t_n\}\subset G_A'$ such that  $\supp(t_n)\to\{x\}$  and $x\notin\supp(t_n)$ for every $n$, and there exists a sequence $\{h_n\}\subset G_A'$ such that $h_n$ commutes with $G_{A,n}'$ and $h_n$ is conjugate to $ht_n$ for every $n$.
Let $s_n\in G_A'$ be such that $ht_n=s_nh_ns_n^{-1}$.

By taking a subsequence if necessary, we may assume that both $\{\pi(t_n)\}$ and $\{\pi(h_n)\}$ are weakly convergent. By  the automatic continuity property of group representations (Theorem  \ref{ThAutomaticContinuity}), we obtain that $\pi(t_n)$ converges to some scalar operator $\lambda\cdot \id$.  In particular, $(\pi(t_n)\xi,\xi)\to \lambda (\xi,\xi).$  Applying the Cauchy-Schwartz inequality, we get that $|\lambda|\leq 1$. 


 By Lemma \ref{LemmaAuxImageSubspace}, we get that $\pi(s_n^{-1})Q^A = Q^A\pi(s_n^{-1})$.   Recall that $\tr(T) = (T\xi,\xi)$ for every $T\in \mathcal M_\pi$.   Therefore, for every $n$ we have 
  \begin{equation}\label{eqnProjectorsAux1}
  \begin{split}
  \tr(\pi(h t_n) Q^A) = \tr(\pi(s_n h_n s_n^{-1}) Q^A)  = \tr(\pi(s_n) \pi( h_n) Q^A \pi( s_n^{-1}) )  \\ 
  = \tr(\pi(h_n)Q^A) = (\pi(h_n)Q^A\xi,\xi). \end{split}\end{equation}
     
Set $L=\lim\pi(h_n)$. Then $L \in \mathcal L(A)$ and  $LQ^A=Q^A$.  Taking the limit  in (\ref{eqnProjectorsAux1}) as $n\to\infty$, we obtain:
\begin{equation*}(\lambda \pi(h)Q^A\xi,\xi)=
(LQ^A\xi,\xi) = (Q^A\xi,\xi) = (Q^A\xi,Q^A\xi) = ||Q^A\xi||^2.\end{equation*}

Lemma \ref{LemmaAuxImageSubspace} ensures that $\pi(h)$ commutes with $Q^A$. Thus, applying the Cauchy-Schwartz inequality we obtain that 
\begin{equation}\label{EqPihQA}||Q^A\xi||^2 =  (\lambda \pi(h)Q^A\xi,\xi)=
(\lambda \pi(h)Q^A\xi,Q^A\xi) \leq |\lambda|  ||Q^A\xi||^2.\end{equation}

There are two possibilities: (i) $|\lambda|<1$ and (ii) $|\lambda| = 1$.

(i) If $|\lambda|<1$, then it follows from \eqref{EqPihQA} that $Q^A\xi = 0$. Since $\xi$ is a separating vector, we obtain that $Q^A=0$, which is a contradiction.

(ii) If $|\lambda| =1$, then it follows from the Cauchy-Schwartz inequality and \eqref{EqPihQA} that $\lambda \pi(h)Q^A\xi = c Q^A\xi$ for some $c\in \mathbb C$, which implies that $$\pi(h)Q^A = c_h Q^A\mbox{ for some }c_h\in \mathbb C\setminus\{0\}. $$ 

Since the operators $\pi(G_A')$ commute with $Q^A$ (Lemma \ref{LemmaAuxImageSubspace}), we obtain that $\pi(qhq^{-1})Q^A = c_h Q^A$ for every $q\in G_A'$.  By simplicity of the group $G_A'$, see Proposition \ref{PropositionSimplicityCommutatorSubgroup} and the references preceding it, the conjugacy class of $h$ generates $G_A'$. Therefore, for every $g\in G_A'$ there exists $c_g\in \mathbb C\setminus \{0\}$ such that $\pi(g)Q^A = c_g Q^A$.  It follows that the set of operators $\pi(G_A')$ restricted onto $Q^A\mathcal H \neq \{0\}$ are scalar, whereby giving rise to a homomorphism $g\mapsto c_g$ into the multiplicative group $\mathbb C\setminus \{0\}$. Simplicity of the group $G_A'$ implies that this homomorphism is trivial, that is, $c_g = 1$ for every $g\in G_A'$, which implies that $\pi(g) Q^A = Q^A$  for every $g\in G_A'$. 
\end{proof}

\begin{lemma}\label{LmPAQA}   For any clopen set $A$ we have that $P^A=Q^A$. In particular,  for every $g\in G'$, we have that $\pi(g)P^A  \pi(g^{-1}) = P^{g(A)}$.
\end{lemma}
\begin{proof}  Set $$\mathcal H_A'' = \{v\in \mathcal H : \pi(g)v = v\mbox{ for every }g\in G_A'\},$$ the range of the projection $P^A$. Clearly, $\mathcal H_A''\subset \mathcal H_A'$. 

Lemma \ref{LemmaAugProjectionAbsorption} implies that $\pi(g) Q^A = Q^A$ for every $g\in G_A'$. Thus, for any $v\in \mathcal H_A'$ and $g\in G_A'$, we have that $Q^Av = \pi(g)v = \pi(g)Q^Av = Q^Av = v$, establishing $\mathcal H_A'\subset \mathcal H_A''$. We conclude that $P^A=Q^A$.
\end{proof}

In a series of lemmas that follow we will show that the projections $\{P^A\}$ form a commutative semigroup.

\begin{lemma}\label{LmPAPBintersect} Let $A$ and $B$ be two clopen sets such that $A\cap B\neq \emptyset$. Then $P^AP^B=P^{A\cup B}$.
\end{lemma}
\begin{proof}  Consider an arbitrary operator $Q\in \mathcal L(A)$. Choose $\{h_n\}\in \mathcal S(A)$ such that  $Q = \lim \pi(h_n)$.  
By Lemmas \ref{lemmaHnInvarSubsets} and \ref{LmPAQA}, we obtain that $\pi(h_n)P^B\pi(h_n^{-1})= P^{h_n(B)}  = P^B$ for all $n$ large enough. Hence,  $\pi(h_n) P^B = P^B \pi(h_n)$ for all $n$ large enough. It follows that $QP^B=P^BQ$. Since $Q$ is an arbitrary element of $\mathcal L(A)$, we obtain that $P^B\mathcal H_A'\subset \mathcal H_A'$. Similarly,  $P^A\mathcal H_B'\subset \mathcal H_B'$. 

 Recall that by Lemma \ref{LmPAQA}, $\mathcal H_A'$ is the range of $P^A$ and  $\mathcal H_B'$ is the range of $P^B$.  Thus for any $v\in \mathcal H$, we obtain that $P^BP^Av\in \mathcal H_A'$. Hence, $P^AP^BP^A v = P^BP^Av$. It follows that $P^AP^BP^A = P^BP^A$. Taking the adjoint of these operators, we obtain that $$P^A P^B = (P^B P^A)^* = (P^AP^BP^A)^* = P^A P^BP^A = P^BP^A.$$

It follows that $P^AP^B$ is the orthogonal projection onto the subspace of $<G_A',G_B'>$-invariant vectors in $\mathcal H$ \cite[Theorem 9.5-3]{Kreyszig:Book}. Since $A\cap B\neq\emptyset$, by Proposition \ref{PropositionSimplicityCommutatorSubgroup}, we have that $<G_A',G_B'>=G_{A\cup B}'$, which completes the proof. 
\end{proof}


Recall that throughout the paper $\mathcal M(X,G)$ denotes the set of $G$-invariant probability measures on $X$. Introduce the quantity:
\begin{equation*}\label{EqRA} r(A)=\inf\{\mu(A) : \mu\in \mathcal M(X,G)\}. 
\end{equation*} 

 By minimality of $(X,G)$, $r(A)>0$ for every non-empty clopen set  $A$, see Lemma 2.8 in \cite{GlasnerWeiss-Equivalence-95}.  Let $A$ and $B$ be clopen sets. Recall that in (\ref{EqDistanceOfSets}) we defined the distance between $A$ and $B$ by 
  \begin{equation*}d(A,B) = \sup\{\mu(A\triangle B) : \mu\in\mathcal M(X,G)\}.\end{equation*}  
 The following lemma shows that the function $$\textrm{CO}(X)\ni A\mapsto \tr(P^A)\in [0,1]$$ is continuous with respect to the metric $d$.   We would like to emphasize that all sets in the lemma are assumed to be non-empty.

\begin{lemma}\label{LmContinuityTraceOnPA} Let $A\subset X$ be a nonempty clopen set. Then for any $0<\varepsilon<1$ there exists $\delta>0$ such that for any nonempty clopen set $B$ with $d(A,B) < \delta$, we have  
 $|\tr(P^A) - \tr(P^B)|<\varepsilon$.
 \end{lemma}
\begin{proof}  (I) Set  $\delta'=\epsilon r(A)/16$.  Let $C\subset A$ be a nonempty clopen set such that $d(A, C)<\delta'$.  We will show that  $|\tr(P^A) - \tr(P^C)|<\varepsilon/2$.  

Pick an integer $m$ such that $2/\varepsilon < m < (2/\varepsilon) +2 $. Then  $\varepsilon/2<2/m<\varepsilon$.  
For every  $\mu\in\mathcal M(X,G)$, we obtain that  
\begin{equation}\label{ineqAuxCGreateA}\mu(A\setminus C)  < \delta' \leq \frac{\varepsilon r(A)}{16}\leq \frac{\varepsilon\mu(A)}{16} < \frac{\mu(A)}{4m}\end{equation}

In particular, $\mu(A\setminus C) < \mu(A)/2$, which implies that $\mu(C) > \mu(A)/2$. Substituting this inequality back into (\ref{ineqAuxCGreateA}), we obtain that

 $$\mu(A\setminus C) \leq \frac{\mu(A)}{4m} < \frac{\mu(C)}{2m} \mbox{ for every }\mu \in \mathcal M(X,G).$$

Using Proposition \ref{PropositionComparisonProperty}, we can find involutions
 $s_j\in G_A'$, $j=1,\ldots,m$, with $s_i(A\setminus C)\subset C$ and 
$$s_j(A\setminus C)\cap s_l(A\setminus C) = \emptyset \mbox{ whenever }j\neq l.$$
In particular, since $s_k(A)=A$,  we have that  $s_j(C)\cup s_l(C) = A$ for  $j\neq l$. Thus, by Lemma \ref{LmPAPBintersect}, we obtain that ($j\neq l$)
$$\begin{array}{ll}\displaystyle \left(P^{s_j(C)} - P^A\right)\left(P^{s_l(C)} - P^A\right)  \\
= P^{s_j(C)}P^{s_l(C)} - P^{s_j(C)}P^A   - P^AP^{s_l(C)} +
 P^AP^A   \\
= P^A - P^A - P^A + P^A =0.
\end{array}$$

Therefore, $\left\{P^{s_j(C)} - P^A\right\}_{j=1}^m$ is a family of mutually orthogonal projections. It follows that
$P = \sum\limits_{j=1}^m\left(P^{s_j(C)} - P^A\right)$ is a projection. Hence, using Lemma \ref{LmPAQA}, we obtain that 
\begin{align*}1\geq  \tr(P) &= \sum\limits_{j=1}^m \tr\left(P^{s_j(C)} - P^A\right) \\ 
& = \sum\limits_{j=1}^m \tr\left(\pi(s_j)(P^{C}-P^{A})\pi(s_j^{-1})\right) \\ 
& = m\cdot \tr\left(P^{C}-P^{A}\right).\end{align*} It follows that
$0\leq \tr\left(P^C-P^{A}\right) = \tr(P^C)-\tr(P^{A})\leq 1/m<\varepsilon/2$.  

(II) Set $\delta = \delta'/2$ and let $B$ be a non-empty clopen set with $d(A,B) < \delta$. Let $C = A \cap B$. Then $d(A,C) < \delta'$, and by the argument in Part (I), we obtain
\begin{equation}\label{eqnTracePCIneq}|\tr(P^C)-\tr(P^{A})| < \varepsilon/2.\end{equation}

Fix $\mu \in \mathcal M(X,G)$. Since
$$
\mu(B) > \mu(A) - \delta'/2 \ge \mu(A) - \frac{\varepsilon}{32}\mu(A),
$$
it follows that $\mu(A) \le 2\mu(B)$. Consequently,
$$
\mu(B \setminus C)
< \frac{\delta'}{2}
\le \frac{\varepsilon}{32}\mu(A)
< \frac{\varepsilon}{16}\mu(B)
\le \frac{\mu(B)}{4m}.
$$

Applying the same argument as in Part (I) to the sets $B$ and $C$, we obtain
$$
|\tr(P^C) - \tr(P^{B})| < \varepsilon/2,
$$
which, together with $(\ref{eqnTracePCIneq})$, completes the proof.
\end{proof}



\begin{lemma}\label{LemmaUnionProjections} For any clopen sets $A$ and $B$ we have that  $P^AP^B=P^{A\cup B}$.
\end{lemma}
\begin{proof} 
The case when $A$ or $B$ is empty is trivial. Assume that both $A$ and $B$ are non-empty. The case $A\cap B\neq \emptyset$ has been established in Lemma \ref{LmPAPBintersect}. 

Assume that $A\cap B=\emptyset$.  Choose a decreasing sequence of non-empty clopen sets $\{C_n\}$ such that $C_n\subset B$ and $\textrm{diam}(C_n)\to 0$ as $n\to\infty$.   
 
 Set $A_n=A\cup C_n$.   Denote by $P$ the limit of the  sequence of increasing projections $\{P^{A_n}\}$, which exists by Lemma \ref{LemmaLimitProjections}.  
  In view of  Proposition \ref{PropositionContinuityMeasures}, we have that $$d(A_n,A)\leq \sup\{\mu(C_n) : \mu\in \mathcal M(G)\}\to 0\mbox{ as }n\to\infty.$$  
  
Thus, Lemma \ref{LmContinuityTraceOnPA} implies that $\tr(P) = \tr(P^A)$.  Since $P^{A_n}\leq P^A$, we obtain that $P\leq P^A$. Therefore, $P^A - P$ is a projection with zero trace, which implies that $P = P^A$. 
 By Lemma \ref{LmPAPBintersect}, we have that
$P^{A_n}P^B= P^{A\cup B}$.  Hence, $P^{A\cup B} = \lim_n P^{A_n}P^B = P^AP^B$. 
\end{proof}

 We will need the following well-known fact  \cite[Section 2.5]{Olshanski-GKpairs-89}. We note that the result in \cite[Section 2.5]{Olshanski-GKpairs-89} is stated for unitary representations of groups, but its proof also covers the case of semigroups.

\begin{proposition}\label{PropOlshanksi} Let $T$ be a unitary representation of a semigroup $U$ on a Hilbert space $\mathcal H_T$.  Let  $P$ be  the orthogonal projection onto the subspace of vectors invariant with respect to $T(U)$. Then for any $v\in \mathcal H_T$ and any $\varepsilon>0$ there exist elements $u_1,\ldots,u_p\in U$ and  positive numbers $\alpha_1,\ldots,\alpha_p$ such that  $\alpha_1+\cdots + \alpha_p=1 $ and $$||\sum_i^p \alpha_i T(u_i)v - Pv||<\varepsilon.$$ 
\end{proposition}

\begin{lemma}\label{lemmaOlshanskyConseq} Let $A$ be a clopen set and $R\in \mathcal M_\pi$. If $\tr(RQ) = \tr(R)$ for every $Q\in \mathcal L(A)$, then
 $\tr(RQ^A) = \tr(R)$. 
\end{lemma}
\begin{proof}   Fix an arbitrary $\varepsilon>0$. According to Lemma \ref{LmLAinv},  $\mathcal L(A)$ is a semi-group. Thus, applying  Proposition \ref{PropOlshanksi}, we can find operators $Q_i\in \mathcal L(A)$ and real numbers $\alpha_i>0$ such that $\alpha_1+\cdots + \alpha_p = 1$ and $|| Q^A\xi - \sum_{i=1}^p \alpha_i Q_i \xi||<\varepsilon $.  It follows that 

\begin{equation*}\begin{split}  |\tr(R Q^A) - \tr(R)| & = |\tr(RQ^A)  - \sum_{i=1}^p\alpha_i \tr( RQ_i)  |  \\ 
 \leq || RQ^A\xi - \sum_{i=1}^p \alpha_i RQ_i \xi|| & \leq ||R|||| Q^A\xi - \sum_{i=1}^p \alpha_i Q_i \xi|| \leq ||R|| \varepsilon. \end{split} \end{equation*}

Since $\varepsilon>0$ is arbitrary, we obtain that   $\tr(RQ^A) = \tr(R)$.
\end{proof}

Now we are ready to complete the proof of Theorem \ref{TheoremCharFromTrace}. 
\begin{proposition}   For any group element $g\in G'$, we have that $$\chi(g)=\tr(P^{\supp(g)}).$$
\end{proposition}
\begin{proof}   Fix $g\in G'$, $g\neq 1$. Using the cycle decomposition of $g$, we can always represent  the support of $g$ as the union of two disjoint clopen $g$-towers of height 2 and of height 3, $$\supp(g) = \left(B\sqcup g(B)\right) \sqcup \left(C\sqcup g(C) \sqcup g^2(C)\right).$$ 

We note that the set $B$ or $C$ might be empty.   For example, for a cycle $(1,2,\ldots,2n)$, we can choose $B = \{1,3,\ldots,2n-1\}$ and  $C = \emptyset$. For a cycle $(1,2,\ldots,2n+1)$, $n\geq 2$, we use $B = \{1,3,\ldots,2n-3\}$ and $C = \{2n-1\}$.  We notice that representing the dynamics of homeomorphisms by its actions on towers is the essence of the Rokhlin lemma which holds for a large class of homeomorphisms of the Cantor set, see, for example, \cite[Proposition 3]{BezuglyiDooleyMedynets:2005} and \cite[Theorem 3.1]{DownarowiczKarpel:2019}. 

Using Proposition \ref{PropositionRokhlinLemma}, we can find increasing sequences of clopen sets $\{B_n\}$, $B_n\subset B$ and $\{C_n\}$, $C_n\subset C$, such that $\mu(B\setminus B_n) \leq 1/n$ and $\mu(C\setminus C_n)\leq 1/n$ for every $\mu\in \mathcal M(X,G)$ and $n\geq 1$ and such that $B_n = B_n'\sqcup B_n''$ and $C_n = C_n'\sqcup C_n''$, where the sets $B_n'$ and $B_n''$ as well as  $C_n'$ and $C_n''$  are $G$-equivalent. 

Fix $n\geq 1$ and arbitrary elements  $$Q_1 \in \mathcal L(B_n\cup g(B_n)),\; {Q_2 \in \mathcal L(C_n\cup g(C_n))},\mbox{ and }{Q_3 \in \mathcal L(g(C_n)\cup g^2(C_n))}.$$ Choose sequences $\{s_k^{(1)}\}\in \mathcal S(B_n\cup g(B_n))$, ${\{s_k^{(2)}\} \in \mathcal S(C_n\cup g(C_n))}$, and ${\{s_k^{(3)}\} \in \mathcal S(g(C_n)\cup g^2(C_n))}$  such that $Q_i = \lim \pi(s_k^{(i)}).$

Note that by Lemma \ref{lemmaHnInvarSubsets} the sets $B_n'$, $B_n''$, $g(B_n')$, and $g(B_n'')$ are $s_k^{(1)}$-invariant   for all $k$ large enough.  Thus, for all $k$ large enough, the homeomorphisms  $h_k'$ and $h_k''$ defined as the restrictions of $s_k^{(1)}$ onto $B_n'$ and $B_n''$ are well-defined and belong to $G$. Note that, in general, restrictions of elements from $G'$ to clopen sets do not have to belong to $G'$. 

 By construction, $q(B_n') = B_n''$ for some $q\in G_{B_n}$. Multiplying $q$ by an odd permutation if necessary, we can assume that $q\in G_{B_n}'$.  Since $s_k^{(1)}$ commutes with  $q$ for all large $k$, we obtain that $h_k'=s_k^{(1)}|B_n'$ and $h_k''=s_k^{(1)}|B_n''$ are conjugate by $q$. It follows that the element $$h_k^{(1)} = h_k'h_k'' = h_k' q h_k' q^{-1} = h_k' q (h_k')^{-1} q^{-1} $$ is a commutator and, thus, belongs to $G_{B_n}'$ for all large $k$.  In other words, we have shown that the restriction of $s_k^{(1)}$ to the set $B_n$ belongs to $G_{B_n}'$.
 
 Since $s_k^{(1)}$ commutes with $g$, we obtain that the restriction of $s_k^{(1)}$ to $g(B_n)$ coincides with $g h_k^{(1)} g^{-1}$. Therefore,  $$s_k^{(1)} = h_k^{(1)}g h_k^{(1)} g^{-1}\mbox{ for all }k\mbox{ large enough}.$$ 
 
 
 Similarly, for all $k$ large enough we can find involutions $h_k^{(2)}\in G_{C_n}'$ and $h_k^{(3)}\in G_{g(C_n)}'$ such that $s_k^{(i)}  = h_k^{(i)}g h_k^{(i)} g^{-1}$ for $i=2,3$. Note that by construction, the support of $h_k^{(1)}$ is disjoint from the supports of $s_l^{(2)}$ and $s_m^{(3)}$ and the support of $h_l^{(2)}$ is disjoint from the support of $s_m^{(3)}$.  Thus, using the fact that $h_k^{(i)} = (h_k^{(i)})^{-1}$, for all $k,l,m$ large enough we obtain that

 \begin{equation*}\begin{split}
 \tr(\pi(s_m^{(3)}s_l^{(2)}s_k^{(1)} g)) & =   \tr(\pi(s_m^{(3)}s_l^{(2)} h_k^{(1)}g h_k^{(1)})) =  \tr(\pi( h_k^{(1)} s_m^{(3)}s_l^{(2)} g h_k^{(1)}))  \\
 & =  \tr(\pi( s_m^{(3)}s_l^{(2)} g )) = \tr(\pi(s_m^{(3)} h_l^{(2)}g h_l^{(2)}))  = \tr(\pi(h_l^{(2)} s_m^{(3)} g h_l^{(2)})) \\
 & = \tr(\pi(s_m^{(3)} g )) = \tr(\pi(h_m^{(3)}g h_m^{(3)})) = \tr(\pi(g)). 
\end{split} \end{equation*}
 
 Taking the limit as $m,l,k\to\infty$, we obtain that $$\tr(Q_3Q_2Q_1\pi(g)) = \tr(\pi(g)).$$   Since $Q_1$, $Q_2$, and $Q_3$ are arbitrary elements in $\mathcal L(B_n\cup g(B_n))$, $\mathcal L(C_n\cup g(C_n))$, and $\mathcal L(g(C_n)\cup g^2(C_n))$ respectively,  by Lemmas \ref{lemmaOlshanskyConseq} and \ref{LemmaUnionProjections}, we obtain that
\begin{equation}\label{eqnTraceProjectionLimit}\begin{split}  \tr(\pi(g)) = \tr\left(Q^{g(C_n)\cup g^2(C_n)} Q^{C_n\cup g(C_n)}Q^{B_n\cup g(B_n)} \pi(g)\right) =  \tr\left(P^{W_n}\pi(g)\right),\end{split}\end{equation}  where  $W_n = \left( B_n\sqcup g(B_n)\right) \sqcup \left(C_n\sqcup g(C_n) \sqcup g^2(C_n)\right)$. 

Notice that $\{W_n\}$ is an increasing sequence of clopen sets, and hence $\{P^{W_n}\}$ is a decreasing sequence of projections. Since ${d(W_n,\supp(g))\to 0}$, Lemma  \ref{LmContinuityTraceOnPA} implies that $\lim_n\tr(P^{W_n}) = \tr(P^{\supp(g)})$. Using the same argument as in Lemma \ref{LemmaUnionProjections}, we obtain that $\lim_n P^{W_n} = P^{\supp(g)}$. Taking the limit in (\ref{eqnTraceProjectionLimit}), we conclude that $$\tr(\pi(g)) = \tr(P^{\supp(g)}\pi(g)) = \tr(P^{\supp(g)}),$$ which establishes the result. 
\end{proof}


%

\section{The Space of Functions on Invariant Measures}\label{Section:SpaceOfMeasures} 

In this section we discuss properties of the space of invariant measures and continuous functions on invariant measures for full groups associated with simple Bratteli diagrams. 

Let $G$ be the full group of a simple Bratteli diagram with path-space $X$.  Throughout this section,  $G_n$ will stand for the permutations of the first $n$-segments of the infinite paths in $X$. Then $G = \bigcup_{n=1}^\infty G_n$. 

To  simplify the notation we will write $\mathcal M$ for  $\mathcal M(X,G)$  and $\mathcal E$ for the set of ergodic measures $\mathcal E(X,G)$.  Note that  $\mathcal M$ is a metrizable compact space  with respect to the weak*-topology and  $\mathcal E$ is $G_\delta$ in $\mathcal M$. Denote by $\mathcal K$ the closure of  the set of ergodic measures $\mathcal E$ in $\mathcal M$. 

Given a clopen set $A\subset X$, define the evaluation function $m_A:\mathcal K\to [0,1]$ by 
\begin{equation*}\label{EqMA} m_A(\mu)=\mu(A).
\end{equation*}  

Denote by  $E(\mathcal K)$ the set of evaluation functions $m_A$, where $A$ runs over all proper clopen sets. Note that $E(\mathcal K) \subset  C(\mathcal K) $.  Denote by $F(\mathcal K)$ the closure of  $E(\mathcal K)$ in $C(\mathcal K)$ with respect to the topology of uniform convergence.    In what follows we establish properties of the set   $F(\mathcal K)$. In particular, we will show that 
 
 $$F(\mathcal K) = \{f\in C(\mathcal K) : 0\leq f \leq 1\}.$$ 

Recall that in  (\ref{EqDistanceOfSets}) and (\ref{EqErgodicDistanceOfSets}) we introduced   the distance function between clopen sets $A$ and $B$ by 
\begin{equation*}
\begin{split}
d(A,B) & = \sup\{\mu(A\triangle B) : \mu\in\mathcal M\} \\
 & =  \sup\{\mu(A\triangle B) : \mu\in\mathcal E\} \\ 
& =  \sup\{\mu(A\triangle B) : \mu\in\mathcal K\}.
\end{split}
\end{equation*}

The following result is a slight modification of Lemma 5.3 from \cite{DudkoMedynets-CharactersSymmetric-13} and can be proven in the same way, just by replacing $\mu(B)$ with  $b$.

\begin{lemma}\label{LemmaMuMultiplicativity} Let $A$ be a clopen set and $\{B_n\}$ be a sequence of clopen sets such that for every $n\geq 1$, the set $B_n$ is $G_n$-invariant. Let $\mu\in \mathcal E$. Assume that the sequence $\{\mu(B_n)\}$ is convergent to some number $b\in (0,1]$. Then $$\mu(A\cap B_n)\to b\mu(A)\mbox{ as }n\to \infty.$$ 
\end{lemma}

%
%

\begin{lemma}\label{lemmaPropertiesFK1} (1) If $f\in F(\mathcal K)$, then there exists a sequence of clopen sets $\{B_n\}$ such that for every $n\geq 1$ the set $B_n$ is $G_n$-invariant and such that the evaluation functions $\{m_{B_n}\}$ uniformly converge to $f$.

(2) If $f\in F(\mathcal K)$ and $\alpha\in [0,1]$, then $\alpha f\in F (\mathcal K)$.  

(3) If $f_1, f_2\in F(\mathcal K)$, then $f_1f_2\in F(\mathcal K)$. 

(4) If  $f_1,f_2\in F(\mathcal K)$ are such that $f_1+f_2<1$, then $f_1+f_2\in F(\mathcal K)$.   

(5) If  $f_1,f_2\in F(\mathcal K)$ are such that $0<f_1- f_2$, then $f_1- f_2\in F(\mathcal K)$.   

\end{lemma}
\begin{proof} (1) Fix a sequence $\{m_{A_n}\}$ of evaluation functions, where $\{A_n\}$ are clopen sets, that is uniformly convergent to $f$.  By Proposition \ref{Proposition_B_to_Invariant_Bn}, we can find a sequence $\{B_n\}$ of clopen sets and a sequence of elements $\{g_n\}$ of $G$ such that for every $n\geq 1$, the set $B_n$ is $G_n$-invariant and $d(g_n(A_n),B_n)<1/n$. Then for every $\mu\in \mathcal K$, we have that   
$$ |f(\mu) - \mu(B_n)|\leq |f(\mu) - \mu(A_n)| + |\mu(g_n(A_n))- \mu(B_n)| \leq |f(\mu) - \mu(A_n)| +\frac{1}{n},
$$ which shows that the sequence $\{m_{B_n}\}$ uniformly converges to $f$. 


(2) It follows from Proposition \ref{PropositionMovingSetClose} that for every clopen $A$ and every real $\alpha\in [0,1]$, one has that $\alpha m_A\in F(\mathcal K)$. By definition, every function $f\in F(\mathcal K)$ is a uniform limit of evaluation functions $\{m_{A_n}\}$ for some clopen sets $\{A_n\}$. Hence, $\alpha f\in F(\mathcal K)$.

(3)  Let $f \in F(\mathcal{K})$, and let $m_A$ be the evaluation function for some clopen set $A$. By (1), there exists a sequence of clopen sets $\{B_n\}$ such that each $B_n$ is $G_n$-invariant and the sequence $\{m_{B_n}\}$ converges uniformly to $f$. Given $\varepsilon>0$ we can find $N>0$ such that for every $m,n\geq N$ and for every $\mu\in \mathcal K$, we have that $$|\mu(B_n)-\mu(B_m)|<\varepsilon.$$
Using Proposition \ref{PropositionMovingSetClose}, we can find elements $g_{n,m}\in G$ such that for every $n,m\geq N$,  $$d(g_{n,m}(B_m),B_n)< 4 \varepsilon .$$ Furthermore, since the sets $B_n$ and $B_m$ are $G_N$-invariant, the proof of Proposition \ref{PropositionMovingSetClose} can be slightly modified to ensure that $g_{n,m}$ acts only on the paths below level $N$. In particular, for every $n,m\geq N$, $g_{n,m}(A) = A$ and, thus, 

 $$ d(g_{n,m}(A\cap B_n),A\cap B_m)  = d(A\cap g_{n,m}( B_n),A\cap B_m) < 4\varepsilon.
 $$
Hence, for every $n,m\geq N$ and every $\mu \in \mathcal K$,

$$|\mu(A\cap B_n)-\mu(A\cap B_m)| =  |\mu(g_{n,m}(A\cap B_n))-\mu(A\cap B_m)| <4\epsilon,$$ which shows that the sequence 
$\{\mu(A\cap B_n)\}$, or, equivalently $\{m_{A\cap B_n}\}$, is uniformly Cauchy on $\mathcal K$.

Applying Lemma \ref{LemmaMuMultiplicativity}, we obtain that the functions 
$m_{A \cap B_n}$ converge pointwise on $\mathcal{E}$ to $m_A f$. Since the sequence 
$\{m_{A \cap B_n}\}$ is uniformly Cauchy on $\mathcal{K}$, the closure of $\mathcal{E}$, 
it follows that $m_{A \cap B_n}$ converges uniformly to $m_A f$ on $\mathcal{K}$. 
Hence, $m_A f \in F(\mathcal{K})$.

Since, by definition, every function in $F(\mathcal{K})$ is a uniform limit of evaluation functions, the result follows.

(4) Since evaluation functions are uniformly dense in $F(\mathcal K)$, it is sufficient to establish the result for $f_1 = m_A$ and $f_2= m_B$, where $A$ and $B$ are clopen sets. Since $\mu(A) +\mu(B)<1$ for every $\mu \in \mathcal K$, we obtain that $\mu(A)< \mu(X\setminus B)$ for every $\mu\in \mathcal K$. By compactness of $\mathcal K$, there is $\delta>0$ such that $\mu(A)<\mu(X\setminus B)-\delta$ for every $\mu\in \mathcal K$. 

Applying  The Ergodic Decomposition Theorem we obtain that  $\mu(A)\leq \mu(X\setminus B)-\delta$ for every $\mu \in \mathcal M$. Applying  Proposition \ref{PropositionComparisonProperty}, we can find  $g\in G$ such that $g(A)\subset X\setminus B$. Since $g(A)$ and $B$ are disjoint, we obtain 
$$m_A+m_B = m_{g(A)}+m_B = m_{g(A)\cup B},$$ which completes the proof. 

(5) Similarly to the proof of (4) above, it is sufficient to establish the result for $f_1 = m_A$ and $f_2= m_B$, where $A$ and $B$ are clopen sets.  Since $m_A-m_B>0$ on $\mathcal K$, by compactness of $\mathcal K$, we can find $\delta>0$ such that $\delta< \mu(A)-\mu(B)$. Using the Ergodic Decomposition Theorem, we obtain that $\mu(B)\leq \mu(A)-\delta$ for every $\mu \in \mathcal M$. 

By Proposition \ref{PropositionComparisonProperty}, we can find  $g\in G$ such that $g(B)\subset A$. It follows that 
$$m_A- m_B = m_{A}-m_{g(B)} = m_{A\setminus g(B)},$$ which completes the proof. 
\end{proof}

\begin{lemma}\label{lemmaPropertiesFK2}  Let $f_1,f_2\in F(\mathcal K)$ and $\alpha_1,\alpha_2\geq 0$. 

(1) If $\alpha_1 f_1 + \alpha_2 f_2<1$, then $\alpha_1f_1+ \alpha_2 f_2\in F(\mathcal K)$.

(2) If $0<\alpha_1 f_1 - \alpha_2 f_2<1$, then $\alpha_1f_1- \alpha_2 f_2\in F(\mathcal K)$. 

\end{lemma}

\begin{proof}

Fix an integer $N>\max\{\alpha_1,\alpha_2\}$.  

(1)  Inductively applying Lemma \ref{lemmaPropertiesFK1}(2, 4), we obtain that $$\alpha_1f_1+ \alpha_2 f_2 = N\left(\frac{\alpha_1}{N}f_1+ \frac{\alpha_2}{N} f_2\right)\in F(\mathcal K).$$

(2) Writing  $$\alpha_1f_1- \alpha_2 f_2 = N\left(\frac{\alpha_1}{N}f_1- \frac{\alpha_2}{N} f_2\right)$$ and applying Lemma \ref{lemmaPropertiesFK1}(2, 4, 5), we obtain the result.
\end{proof}

The following result describes the structure of continuous functions from $C(\mathcal K)$. 

\begin{theorem}\label{ThmStructureContFun} (1) $C(\mathcal K) = \{\alpha_1 f_1 - \alpha_2 f_2 : f_1, f_2 \in F(\mathcal K),\; \alpha_1,\alpha_2 \geq 0\}$.

(2) $F(\mathcal K) = \{f\in C(\mathcal K) : 0\leq f \leq 1\}$.
\end{theorem}
\begin{proof}   Set $\Omega = \{\alpha_1 f_1 - \alpha_2 f_2 : f_1, f_2 \in F(\mathcal{K}),\; \alpha_1, \alpha_2 \geq 0\}$. By Statements (2) and (3) of Lemma \ref{lemmaPropertiesFK1} and Lemma \ref{lemmaPropertiesFK2}, $\Omega$ is a closed subalgebra of $C(\mathcal{K})$. Since evaluation functions separate points of $\mathcal{K}$, the Stone–Weierstrass Theorem implies that $C(\mathcal{K}) = \Omega$.

Note that if $f\in F(\mathcal K)$, then $0 \leq f \leq 1$. Conversely, if $f\in C(\mathcal K)$ is such that $0<f<1$, then $f = \alpha_1 f_1 - \alpha_2 f_2 $ for some $\alpha_1,\alpha_2\geq 0$ and $f_1,f_2\in F(\mathcal K)$. By Lemma \ref{lemmaPropertiesFK2}(2), we obtain that $f\in F(\mathcal K)$. To complete the proof, we notice that  functions $f\in C(\mathcal K)$ with $0<f<1$ are dense in the space of functions $g\in C(\mathcal K)$ satisfying $0\leq g\leq 1$ in the topology of uniform convergence. 

\end{proof}


\section{Indecomposable Characters of the Commutator Subgroups}\label{SectionCharactersCommutarorSub}

In this section we show that each indecomposable character on the full group $G$ of a simple Bratteli diagram gives rise a linear functional on $C(\mathcal K)$ where $\mathcal K$ is the closure of the set of $G$-invariant ergodic measures on the path-space of the associated Bratteli diagram.

In what follows we use the notation of Section \ref{Section:Projections} where  $G'$ is the commutator subgroup of $G$, $X$ is the path-space of the associated  Bratteli diagram,  $(\mathcal{H}, \pi, \xi)$ is a non-trivial factor representation of $G’$,  $\mathcal{M}_\pi$ is the von Neumann algebra generated by $\pi(G’)$, and $\tr$ is the corresponding trace on $\mathcal{M}_\pi$.  For each clopen set $B$, $P^B$ denotes the orthogonal projection onto the space $\pi(G'_B)$-invariant vectors.

For each evaluation function $m_A$, where $A$ is clopen, set \begin{equation} 
\phi(m_A) = \tr(P^{X\setminus A}).
\end{equation} 

We note that the definition of the function $\phi$ is tied to the indecomposable character $\chi$.  

\begin{lemma}\label{LemmaContituityPhi}  If the character $\chi$ is non-regular, then the function $\phi : E(\mathcal K) \rightarrow \mathbb R$ is continuous with respect to the topology of uniform convergence on $E(\mathcal K)$. 
\end{lemma}
\begin{proof}
By the tracial property of $\tr$, Proposition~\ref{PropositionMovingSetClose}(2), and Lemma~\ref{LmContinuityTraceOnPA}, for every $\varepsilon>0$ and every nonempty clopen set $A$ there exists $\delta>0$ such that, for any nonempty clopen set $B$ satisfying
$$
|\mu(A)-\mu(B)|<\delta \quad \text{for all } \mu\in\mathcal M,
$$
we have
$$
|\tr(P^A)-\tr(P^B)|<\varepsilon.
$$

By the Ergodic Decomposition Theorem \cite[Theorem~9.5]{DoughertyJacksonKechris:1994}, the condition
$
|\mu(A)-\mu(B)|<\delta \quad \text{for all } \mu\in\mathcal K
$
implies the same inequality for all $\mu\in\mathcal M$. Consequently, the function $\phi$, when restricted to the set of evaluation functions $E(\mathcal K)$, is continuous.
\end{proof}

By the continuity of $\phi$ we can extend its domain to $F(\mathcal K)$, which, by Theorem \ref{ThmStructureContFun}(2), coincides with the set of functions  $f\in C(\mathcal K)$ with $0\leq f \leq 1$.

\begin{lemma}\label{LemmaPropLittlePhi}  Suppose that the character $\chi$ is non-regular. Then 

(1) For any $f_1, f_2\in F(\mathcal K)$, we have that $$\phi(f_1 f_2) = \phi(f_1)\phi(f_2).$$ 

(2) For any $f_1, f_2\in F(\mathcal K)$ with $\sup_{\mu \in \mathcal K} f_1(\mu) < \inf_{\mu \in \mathcal K}f_2(\mu)$, we have that $\phi(f_1)\leq \phi(f_2)$.

(3) For any $f \in F(\mathcal K)$ such that $f>0$, we have that $\phi(f)>0$. 
\end{lemma}
\begin{proof}  (1) Let $f_1 = m_A$, where $A$ is a proper clopen set, and $f_2\in F(\mathcal K)$.  By  Lemma \ref{lemmaPropertiesFK1}, $m_A f_2\in F(\mathcal K)$ and we can represent $f_2$ as the uniform limit of evaluation functions $\{m_{B_n}\}$ such that for each $n\geq 1$, the set $B_n$ is $G_n$-invariant.   It follows from the proof of Lemma \ref{lemmaPropertiesFK1}(3) that the functions $\{m_{A\cap B_n}\}$ converge to $m_A f_2$ uniformly on $\mathcal K$.  

Taking a subsequence if needed, without loss of generality, we can assume that the sequence of projections $\{P^{X\setminus B_n}\}$ is convergent. Since each set $B_n$ is invariant under the action of $G_n'$, by Lemma \ref{LmPAQA}, we obtain that the orthogonal projections $\{P^{X\setminus B_n}\}$ asymptotically commute with $\pi(G')$ and, thus, converge to a scalar operator $\lambda \cdot \id$. Note that $\phi(f_2) = \lim_n \tr(P^{X\setminus B_n}) = \lambda$.  Using Lemma \ref{LemmaUnionProjections}, we conclude that 
\begin{multline*}
\phi(m_Af_1) = \lim_{n} \phi (m_{A\cap B_n}) = \lim_n \tr(P^{X\setminus (A\cap B_n)}) = \lim_n \tr(P^{X\setminus A} P^{X\setminus B_n}) \\ = \tr(P^{X\setminus A}) \phi(f_2) = \phi(m_A)\phi(f_2). 
\end{multline*}

Since evaluation functions are uniformly dense in $F(\mathcal K)$, we obtain that $\phi(f_1f_2) = \phi(f_1) \phi(f_2)$ for every $f_1,f_2\in F(\mathcal K)$. 

(2)  By definition of $F(\mathcal K)$, we can find two sequences of nonempty clopen sets $\{A_n\}$ and $\{B_n\}$ such that the sequence $\{m_{A_n}\}$ uniformly converges to $f_1$ and $\{m_{B_n}\}$ uniformly converges to $f_2$. By compactness of $\mathcal K$, we can find $N>0$ such that $m_{A_n}<m_{B_n}$  on $\mathcal K$ for all $n\geq N$. Thus, for every $n\geq N$, we have that $\mu(A_n)<\mu(B_n)$ for every $\mu \in \mathcal K$, and, by the Ergodic Decomposition Theorem, for every $\mu \in \mathcal M$. 

By the Comparison Property, Proposition \ref{PropositionComparisonProperty}, and the fact that every measure $\mu$ is $G$-invariant, we can assume that $A_n\subset B_n$ for every $n\geq N$. Thus, for the orthogonal projections $P^{X\setminus A_n}$ and $P^{X\setminus B_n}$, see the definition in Section \ref{Section:Projections}, we obtain that $P^{X\setminus A_n}\leq P^{X\setminus B_n}$. Hence, 
$$\phi(m_{A_n}) = \tr(P^{X\setminus A_n})\leq \tr(P^{X\setminus B_n}) = \phi(m_{B_n})\mbox{ for every }n\geq N.$$

By the continuity of $\phi$, Lemma \ref{LemmaContituityPhi}, we obtain that $\phi(f_1)\leq \phi(f_2)$. 

(3) Assume that $f\in F(\mathcal K)$ is such that $f>0$ and $\phi(f) = 0$. By Proposition \ref{PropositionContinuityMeasures}, we can find a nonempty clopen set $A$ such that $\sup_{\mu\in \mathcal K} m_A(\mu)<\inf_{\mu\in \mathcal K} f(\mu)$. Then by Statement (2) above, we obtain that $\phi(m_A) = 0$.  

Fix an arbitrary clopen set $B\neq X$. Note that $\mu(B)<1$  for every $G$-invariant measure $\mu$. Hence, by compactness of $\mathcal K$,  we can find an integer $n>0$ such that $\sup_{\mu\in \mathcal K} \mu(B)^n <\inf_{\mu\in \mathcal K} \mu(A)$. Applying Statements (1) and (2), we obtain $$\phi(m_B)^n = \phi(m_B^n)\leq \phi(m_A) = 0,$$  which implies that $$\phi(m_B) = \tr(P^{X\setminus B}) = 0. $$

Hence, by Theorem \ref{TheoremCharFromTrace}, for any $g\in G'$ such that $\fix(g)\neq X$, we obtain that $$\chi(g) = \tr(P^{X\setminus \fix(g)}) = 0,$$ which contradicts our assumption that the character $ \chi$ is non-regular. 
\end{proof}

For $f\in C(\mathcal K)$ such that $f> 0$, we can rewrite it as $f = f_1/f_2$ for some $f_i\in C(\mathcal K)$, $0<f_i\leq 1$,  $i=1,2$. For example, we can use $f_1 = f/L$ and $f_2 = 1/L$, where $L = \sup_{\mu \in \mathcal K}f(\mu)+1$.  Note that by Lemma \ref{LemmaPropLittlePhi}, $\phi(f_2)>0$. Thus, we can define 
\begin{equation}\label{eqnDefPhi}\Phi(f) = \frac{\phi(f_1)}{\phi(f_2)}.\end{equation}

\begin{lemma}\label{lemmaPropBigPhi} Suppose that the character $\chi$ is non-regular.  Then 

(1) The function $\Phi$ in (\ref{eqnDefPhi}) is well-defined and continuous on the set of strictly positive functions in $C(\mathcal K)$. 

(2) The function $\Phi$ coincides with $\phi$ on the set of functions $f\in C(\mathcal K)$, $0 < f\leq 1$.

(3)  For every $f\in C(\mathcal K)$ such that $f>0$, we have that $\Phi(f)>0$. 

(4) For every $f_1,f_2\in C(\mathcal K)$ such that $f_i>0$, $i=1,2$, we have that $\Phi(f_1f_2) = \Phi(f_1)\Phi(f_2)$. 
\end{lemma}
\begin{proof}  Assume that $f= f_1/f_2 = f_1'/f_2'$ for some strictly positive $f_1, f_1',f_2, f_2'\in F(\mathcal K)$. Then $f_1 f_2' = f_2f_1'$. By Lemma \ref{LemmaPropLittlePhi}, we conclude that $$\phi(f_1)\phi(f_2') = \phi(f_2)\phi(f_1').$$ Note that $\phi(f_2)>0$ and $\phi(f_2')>0$. Therefore,  $$\Phi(f) = \frac{\phi(f_1)}{\phi(f_2)} = \frac{\phi(f_1')}{\phi(f_2')},$$ which establishes that the function $\Phi$ is well-defined.  The continuity of $\Phi$ follows from the continuity of $\phi$. 

If $f\in C(\mathcal K)$ such that $0<f\leq 1$, then by (\ref{eqnDefPhi}) and multiplicativity of $\phi$, Lemma \ref{LemmaPropLittlePhi}, we obtain that $$\Phi(f) = \phi(f^2)/\phi(f) =\phi(f).$$ 

Statements (3) and (4) follow from the definition of $\Phi$ and Lemma \ref{LemmaPropLittlePhi}.
\end{proof}

Define the function $l: C(\mathcal K) \rightarrow \mathbb R$ by 
\begin{equation} l(f) = \log \Phi(\exp(f)) \mbox{ for every }f\in C(\mathcal K).  \end{equation}

\begin{proposition} $l$ is a continuous positive linear functional on $C(\mathcal{K})$.
\end{proposition}
\begin{proof}
By Lemma \ref{lemmaPropBigPhi}, $l$ is a continuous additive function on $C(\mathcal{K})$. Let $f \in C(\mathcal{K})$. The additivity of $l$ implies that for any $p, q \in \mathbb{N}$, we have
$$l\left(\frac{p}{2^q}f\right) = \frac{p}{2^q}l(f).$$
By continuity, it follows that $l(\lambda f) = \lambda l(f)$ for any $\lambda > 0$. Finally, 
$$l(0) = \log \Phi(1) = \log\left(\frac{\phi(1)}{\phi(1)}\right) = 0.$$ Here by $1$ we mean the function that is equal to 1 on $\mathcal K$.
It follows that $l$ is a continuous linear functional on $C(\mathcal K)$.  

Consider an arbitrary $f\in C(\mathcal K)$ with $f>0$.  Let $\delta = \inf_{\mu \in \mathcal K}f(\mu)>0$. Consider the functions $$f_1(\mu) = \frac{\exp(f(\mu))}{1+\exp(f(\mu))}\mbox{ and }f_2(\mu)=\frac{1}{1+\exp(f(\mu))}.$$ 
Note that $$0<  f_2(\mu) \leq \frac{1}{1+\exp(\delta)} <  \frac{\exp(\delta)}{1+\exp(\delta)} \leq  f_1(\mu) < 1  \mbox{ for every }\mu\in \mathcal K. $$

Applying Lemma \ref{LemmaPropLittlePhi}, we get that $\phi(f_1)\geq \phi(f_2)$. It follows that $\Phi(f) = \phi(f_1)/\phi(f_2)\geq 1$.  Hence, $l(f)\geq 0$. 

Since every non-negative continuous function can be approximated by strictly positive functions in the topology of the uniform convergence, we conclude that $l(f)\geq 0$ for every $f\in C(\mathcal K)$ with $f\geq 0$.
\end{proof}

By the Reisz-Markov-Kakutani Representation Theorem,  there exists a unique Radon measure $\nu$ on $\mathcal K$ such that $$l(f) =  \log\left (\Phi \left(\exp(f)\right)\right) = \int_{\mathcal K} f(\mu) d\nu(\mu) \mbox{ for every }f\in C(\mathcal K).$$ 

Thus, for every $f\in C(\mathcal K)$, $f>0$, we obtain that  
\begin{equation} 
 \Phi(f) = \exp\left(\int_{\mathcal K} \log(f(\mu)) d\nu(\mu) \right)
\end{equation}

Applying Theorem \ref{TheoremCharFromTrace}, for every $g\in G'$, we obtain that 

\begin{equation}\label{eqnChig=Integral} \chi(g) = \tr(P^{X\setminus \fix(g)}) = \Phi(m_{\fix(g)}) = \exp\left(\int_{\mathcal K} \log(\mu(\fix(g))) d\nu(\mu) \right) 
\end{equation}

\begin{remark}\label{remarkCharactersFiniteErgodic} If the system $(X,G)$ admits only a finite number of ergodic measures $\mathcal E = \{\mu_1,\ldots, \mu_n\}$, then $\mathcal K = \mathcal E$. Setting $\alpha_i = \nu(\{\mu_i\}) \geq 0$ and applying (\ref{eqnChig=Integral}), we obtain that $$\chi(g) = \prod_{i=1}^n \mu_i(\fix(g))^{\alpha_i} \mbox{ for every }g\in G'.$$  It follows from the proof of Proposition \ref{PropAlphaAnInteger} that the numbers $\{\alpha_i\}$ are  integers. 
\end{remark}

The following theorem, together with Proposition~\ref{propProductMeasuresIsCharacter}, completely classifies the indecomposable characters of the commutator subgroups of full groups of simple Bratteli diagrams.

%
%
\begin{theorem}\label{thmClassificationCharactersCommutatorSubgroup}  Let $G$ be the full group of a simple Bratteli diagram with path-space $X$. Let $\chi$ be a non-regular  indecomposable character of the commutator subgroup $G’$. Then there exists a finite number of $G$-invariant ergodic measures $\{\mu_1,\ldots,\mu_n\}$ and non-negative integers $\{\alpha_1,\ldots,\alpha_n\}$ such that  
$$
\chi(g) = \prod_{i=1}^n \mu_i(\fix(g))^{\alpha_i}\mbox{ for every }g\in G'.
$$
\end{theorem}
\begin{proof} Fix a non-empty clopen set $B$, $B\neq X$, an integer $m\geq 1$, and $\delta >0$.  Using Proposition \ref{PropositionRokhlinLemma}(2), we can find pairwise disjoint clopen sets $\{B_0,\ldots, B_{m-1} \}$ contained in $B$ and a subgroup $H_m<G_B'$ isomorphic to $\sym(m)$  that permutes the sets $\{B_i\}$ and acts trivially elsewhere and such that $$\mu\left(B\setminus B_\delta'\right ) < \delta \mbox{ for every }\mu \in \mathcal M,\mbox{ where }B_\delta'=\bigcup_{i=0}^{m-1}B_i.$$  

Denote by $\rho_{m,B_\delta'} : \sym(m) \rightarrow H_m$ the isomorphism between the groups.  Note that the function $\chi_{m,B,\delta}(s) = \chi(\rho_{m,B_\delta'}(s))$ defines a character on $
 \sym(m)$.

For $s\in \sym(m)$, denote by $$\chi_m(s) = \frac{|\{1\leq i \leq m: s(i) = i \}|}{m}, $$ the normalized permutation character.   It follows from (\ref{eqnChig=Integral}) that for $s\in \sym(m)$: 
\begin{align*}\chi_{m,B,\delta}(s) &= \exp\left (\int_\mathcal K \log(\mu(\fix(s))) d\nu(\mu) \right)  \\
&= \exp\left (\int_\mathcal K \log\left( \mu(X\setminus B_\delta')  +  \chi_m(s) \mu(B_\delta') \right) d\nu(\mu) \right) \\
&= \exp\left (\int_\mathcal K \log\left( \mu(X\setminus B)  +  \chi_m(s) \mu(B) +(1-\chi_m(s))\mu(B\setminus B_\delta') \right) d\nu(\mu) \right).
\end{align*} 
Recall that  the space of characters is closed under point-wise limits. Applying Lebesgue's Dominated Convergence Theorem as $\delta \to 0$, we obtain that the function 

\begin{equation}\label{eqnClassChiMB}\begin{split}\chi_{m,B}(s) & =  \exp\left (\int_\mathcal K \log\left[ \mu(X\setminus B)  +  \mu(B)\chi_m(s) \right] d\nu(\mu) \right)  \\
& =  \exp\left (\int_\mathcal K \log\left[ 1-m_{B}(\mu)  +  \chi_m(s) m_B(\mu) \right] d\nu(\mu) \right)\end{split}\end{equation}
defines a character on $\sym(m)$. 
 
 Fix $f\in C(\mathcal K)$, $0\leq f< 1$. Using Theorem \ref{ThmStructureContFun}, find a sequence of evaluation functions $\{m_{B_n}\}$ that uniformly converges to $f$. Substituting $m_{B_n}$ into  (\ref{eqnClassChiMB}) and applying Lebesgue’s Dominated Convergence Theorem  as $n\to \infty$, we can conclude that 
 
 \begin{equation}\label{eqnClassChiFun} \chi_{m,f}(s) =  \exp\left (\int_\mathcal K \log\left[ 1-f(\mu)  +  \chi_m(s) f(\mu) \right] d\nu(\mu) \right)
 \end{equation}
  defines a character on $\sym(s)$.

Consider an arbitrary non-empty  closed (hence compact) set $K\subset \mathcal K$. Using that $\mathcal K$ is a Normal Hausdorff space,  we can inductively construct  a decreasing sequence of open sets $\{U_n\}$ in $\mathcal K$ such that $$U_n\supset \overline{U_{n+1}}\text{ for each }n\in\mathbb N\text{ and }\cap_{n\in\mathbb N}U_n=K.$$

 Moreover, by Urysohn's Lemma for each $n\in\mathbb N$ there is a function $f_n\in C(\mathcal K)$, $0\leq f_n\leq 1$, such that $f_n(\mu)=0$ for all $\mu\in \mathcal K\setminus U_n$ and $f_n(\mu)=1$ for each $\mu\in U_{n+1}$. Then the sequence of functions $\{f_n\}$ converges monotonically to the characteristic function $\mathbbm{1}_K$. Fix $\tfrac{1}{2}>\epsilon>0$.  Substituting $g_{n,\epsilon}=\epsilon+(1-2\epsilon)f_n$ into (\ref{eqnClassChiFun}) with $s\neq e$ and applying Lebesgue’s Dominated Convergence Theorem as $n\to \infty$, we conclude that the function 
\begin{equation*}
 \begin{split}\phi_{K,\epsilon}(s) =  \lim\limits_{n\to\infty}\exp\left (\int_\mathcal K \log\left[ 1-g_{n,\epsilon}(\mu)  +  \chi_m(s) g_{n,\epsilon}(\mu) \right] d\nu(\mu) \right)  \\
 = \exp\left (\int_\mathcal K \log\left[ 1- \epsilon-(1-2\epsilon)\mathbbm{1}_K(\mu)  +  \chi_m(s) (\epsilon+(1-2\epsilon)\mathbbm{1}_K(\mu)) \right] d\nu(\mu) \right) 
  \\
  = \left((1-\epsilon + \epsilon \chi_m(s))^{\nu(\mathcal K \setminus K)}\right) \cdot \left( (\epsilon + (1-\epsilon)\chi_m(s))^{\nu(K)}\right)
  \end{split}
 \end{equation*}
 defines a character on $\sym(m)$.  Taking the limit as $\epsilon\to 0$, we obtain that for every closed set $K\subset \mathcal K$, the function \begin{equation}\label{eqnPhiKChar}\phi_K(s) = \chi_m(s)^{\nu(K)}\end{equation} defines a character on $\sym(m)$. 
 
Consider the full group $S_{2^\infty}$ corresponding to the Bratteli diagram of the 2-odometer defined in Example \ref{exampleS2Infty}. The group $S_{2^\infty}$ is the inductive limit of groups $\sym(2^n)$ and the characters in (\ref{eqnPhiKChar}) converge to the character $\mu(\fix(g))^{\nu(K)}$, where $\mu$ is the unique $S_{2^\infty}$-invariant measure on its Bratteli diagram. Thus, by \cite[Proposition 12]{Dudko-CharsErgodic-11}, see also Proposition \ref{PropAlphaAnInteger}, we conclude that $\nu(K)$ must be an integer. 

 Thus, we have established that  $\nu(K)$ is a non-negative integer for  every closed set $K\subset \mathcal K$. Therefore, $\nu$ must be a  purely atomic integer-valued measure with finitely many atoms.  Denoting the atoms of $\nu$ by $\{\mu_1,\ldots,\mu_n\}\subset \mathcal K$ and  applying (\ref{eqnChig=Integral}), we obtain that
 \begin{equation}\label{eqnClassCharaFinalForm}\chi(g) = \prod_{i=1}^n \mu_i(\fix(g))^{\nu(\{\mu_i\})} \mbox{ for every }g\in G'.\end{equation}  

Applying Proposition~\ref{propProductMeasuresIsCharacter} we conclude that each measure $\mu_i$ is ergodic, which completes the proof. 
\end{proof}

 In view of the simplicity of the commutator subgroup $G' $ of a simple Bratteli diagram and the fact that $G'$ contains arbitrarily large symmetric groups, all finite-type factor representations are of type $\mathrm{II}_1$.  The following result is an immediate corollary of Theorem \ref{thmClassificationCharactersCommutatorSubgroup}, Proposition \ref{propProductMeasuresIsCharacter}, and Proposition \ref{PropGNSQuasiEquivalence}. 
 
 \begin{corollary}\label{CorollaryQuasiEquivalenceCommutatorSubgroup} Let $\pi$ be a non-trivial (non-identity, non-regular) factor representation of the commutator subgroup $G'$ of the full group of a simple Bratteli diagrtam with path-space $X$. Then there exist $G$-invariant ergodic measures $\mu_1,\ldots, \mu_n$ (repetitions permitted) such that $\pi$ is quasi-equivalent to the Feldman-Moore groupoid representation  (Example \ref{exampleProductAction}) associated with the product action of $G'$ on $(X^n,\mu_1\times \cdots \times \mu_n)$. 
  \end{corollary}
  
  In the next result, we establish the  automatic continuity of finite-type unitary representations of the group $G'$.

  \begin{corollary}\label{corollaryCommutatorAutomaticCont}
Let $G'$ be the commutator subgroup of the full group associated with a simple Bratteli diagram, and let $\pi$ be a finite-type unitary representation of $G'$. Suppose that the \emph{central decomposition} (Theorem~\ref{TheoremCentralDecomposition}) of the representation $\pi|_{G'}$ contains no regular subrepresentation of $G'$.

If $\{g_n\}$ is a Cauchy sequence in $G'$ with respect to the metric
\begin{equation}\label{EqUniformMetric}
D(g',g'') = \sup_{\mu \in \mathcal{M}} \mu(\{x \in X : g'x \neq g''x\}),
\end{equation}
then the sequence $\{\pi(g_n)\}$ converges in the strong operator topology.

Moreover, if  the sequence $g_n$ converges to $h\in G'$ with respect to the metric $D$, then the sequence $\{\pi(g_n)\}$ converges to $\pi(h)$ in the strong-operator topology. 
\end{corollary}
    \begin{proof}   Let $\mathcal M_\pi = \{\pi(G')\}''$.  By Theorem~\ref{TheoremCentralDecomposition}, there exists a standard probability measure space $(\Omega, \nu)$ and a measurable field $\omega \mapsto (\mathcal{H}_\omega, \pi_\omega)$ of Hilbert spaces and representations such that the Hilbert space $\mathcal{H}$ is isomorphic to the direct integral 
$
\mathcal{H}' = \int_{\Omega}^{\oplus} \mathcal{H}_\omega \, d\nu(\omega).
$
Under this identification, the representation $\pi$ decomposes as
\begin{equation}\label{eqnAutContDirectIntegral}
\pi(g) = \int_{\Omega}^{\oplus} \pi_\omega(g) \, d\nu(\omega), \quad \text{for all } g \in G'.
\end{equation}
    
 Additionally, we have that  for every $\omega \in \Omega$, the von Neumann algebra $\mathcal M_\omega$ generated by $\pi_\omega(G')$ is a factor. 

By our assumptions, none of the representations $\pi_\omega$ is quasi-equivalent to the regular representation of $G'$. Hence, by Corollary~\ref{CorollaryQuasiEquivalenceCommutatorSubgroup}, each $\pi_\omega$ is quasi-equivalent to a groupoid representation $\widetilde{\pi}_\omega$ of $G'$, arising from the product action of $G'$ on the space $(X^n,  \widetilde \mu)$, where $\widetilde \mu = \mu_1 \times \cdots \times \mu_n$ and the measures $\mu_i$ are $G'$-invariant and ergodic (repetitions permitted); see Example~\ref{exampleProductAction} for details.  

It follows from our assumptions that $\{g_n\}$ is a Cauchy sequence with respect to the metric $d(g',g'') = \widetilde \mu \{x\in X^n : g'(x) \neq g''(x)\}$.  Thus, by Proposition \ref{PropGroupRepresentationContinuity}, we obtain that the  sequence of operators $\{\widetilde{\pi}_\omega(g_n)\}$ converges in the  strong operator topology.  

By  \cite[Proposition III.2.2.14]{Blackadar:BookOperatorAlgebras}, the quasi-equivalence preserves convergence of unitary operators in the strong operator topology. Thus, the  sequence  $\pi_\omega(g_n)$ converges in the  strong operator topology.  

Applying Lebesgue's dominated convergence theorem to the direct integral (\ref{eqnAutContDirectIntegral}), we obtain that the sequence $\{\pi(g_n)\}$ converges in the strong operator topology, which completes the proof.     
    \end{proof}

%

\section{Indecomposable Characters of Full Groups}\label{Section:CharactersFullGroup}

In this section, we classify characters of various groups that contain, as subgroups, the commutator subgroups of full groups associated with simple Bratteli diagrams.

Recall that for a minimal homeomorphism $T$ of the Cantor set $X$, the full group is denoted by $\mathcal{F}(\langle T \rangle)$, and its commutator subgroup is denoted by $\mathcal{A}(\langle T \rangle)$. Throughout this section, $G$ and $G'$ will denote, respectively, the full group and its commutator subgroup associated with the Bratteli diagram corresponding to $(X, T)$. Accordingly, we have the inclusions
$$
G' \subset G\subset \mathcal{F}(\langle T \rangle) \mbox{ and }
G' \subset \mathcal{A}(\langle T \rangle) \subset \mathcal{F}(\langle T \rangle).
$$

  %
  
  \begin{definition}\label{DefinitionCompatibleSubgroups}
We say that a pair of groups $(\Gamma, H)$ is \emph{compatible} if $\Gamma \subset H$ and, for every $h \in H\setminus \{e\}$, there exists a sequence of distinct elements $\{\gamma_n\} \subset \Gamma$ such that, setting $h_n = \gamma_n h \gamma_n^{-1}$, we have
\begin{equation}\label{eqnCompClassConjugacy}
h_n \neq h_m \quad \text{and} \quad h_n h_m^{-1} \in \Gamma 
\quad \text{for all } n \neq m.
\end{equation}
\end{definition}

Recall that the structure of the central decomposition of group representations is described in Theorem~\ref{TheoremCentralDecomposition}.

\begin{proposition}\label{PropRegularReprCompatibleGroups}
Let $(\Gamma, H)$ be a pair of compatible groups, and let $\pi$ be a factor representation of $H$ on a Hilbert space $\mathcal{H}$. Suppose that $\Gamma$ is an ICC group. If the central decomposition of $\pi|_{\Gamma}$ on $\mathcal{H}$ contains a subrepresentation quasi-equivalent to the regular representation of $\Gamma$, then the representation $\pi$ of $H$ is quasi-equivalent to the regular representation of $H$.
\end{proposition}
\begin{proof}  
We note that if $\pi_1$ and $\pi_2$ are quasi-equivalent finite-type unitary representations of $\Gamma$, and if the central decomposition of $\pi_1$ contains a subrepresentation quasi-equivalent to the regular representation, then the central decomposition of $\pi_2$ also contains a subrepresentation quasi-equivalent to the regular representation. Consequently, in our proof, we may work with a representation of $\pi|_H$ that is quasi-equivalent to the original one.

By Proposition~\ref{PropGNSQuasiEquivalence}, we may assume without loss of generality that $\pi$ arises from a GNS construction $(\mathcal{H}, \pi, \xi)$, where $\operatorname{tr}(m) = (m\xi, \xi)$ defines a trace on $\mathcal{M}_{\pi(H)} = \pi(H)''$. Assume that the central decomposition (Theorem~\ref{TheoremCentralDecomposition}) of the restriction $\pi|_{\Gamma}$ contains a component $\mathcal{H}_0$ such that $\pi(\Gamma)|_{\mathcal{H}_0}$ is quasi-equivalent to the regular representation of $\Gamma$. Denote by $P$ the orthogonal projection onto $\mathcal{H}_0$.

Then $P$ belongs to the von Neumann algebra $\mathcal{M}_{\pi(\Gamma)} \cap \mathcal{M}_{\pi(\Gamma)}'$, where $\mathcal{M}_{\pi(\Gamma)} = \pi(\Gamma)''$. In particular, $P \in \mathcal{M}_{\pi(H)}$. Define $\xi_0 = P\xi$. By Part~II of Theorem~\ref{TheoremCentralDecomposition}, the functional $(m\xi_0, \xi_0)$ is a trace on $P\mathcal{M}_{\pi(\Gamma)}P$. 

Since $P\mathcal{M}_{\pi(\Gamma)}P$ is a factor quasi-equivalent to the regular representation of $\Gamma$, the uniqueness of the trace implies that $(\pi(\gamma)\xi_0, \xi_0) = 0$ for every $\gamma \in \Gamma \setminus \{e\}$.

Fix $h \in H \setminus \{e\}$, and choose a sequence of distinct elements $\{\gamma_n\} \subset \Gamma$ satisfying~\eqref{eqnCompClassConjugacy}. Set $h_n = \gamma_n h \gamma_n^{-1}$. Then
$$
(\pi(h_n)\xi_0, \pi(h_m)\xi_0) = (\pi(h_m^{-1} h_n)\xi_0, \xi_0) = 0
\quad \text{whenever } n \neq m.
$$
Hence $\{\pi(h_n)\xi_0\}$ is an orthonormal sequence and therefore converges weakly to $0$. Moreover,
\begin{equation*}
\begin{split}
(\pi(h_n)\xi_0, \xi_0) 
&= (\pi(h_n)P\xi, P\xi)
= (P\pi(h_n)P\xi, \xi) \\
&= (\pi(\gamma_n) P\pi(h)P \pi(\gamma_n^{-1})\xi, \xi)
= (P\pi(h)P\xi, \xi)
= (\pi(h)\xi_0, \xi_0) \to 0.
\end{split}
\end{equation*}
Thus $(\pi(h)\xi_0, \xi_0) = 0$ for every $h \in H \setminus \{e\}$.

Let $\mathcal H_1 = \overline{\textrm{Lin}(\pi(H)\xi_0)}$. Then $\pi|_{\mathcal{H}_1}$ is unitarily equivalent to the regular representation. Since $\pi|_{\mathcal{H}}$ is a factor representation, Proposition~\ref{PropPropertiesQuasiEquivalence} implies that $\pi|_{\mathcal{H}_1}$ and $\pi|_{\mathcal{H}}$ are quasi-equivalent, which completes the proof.
\end{proof}

%

\begin{proposition}\label{PropCompatibleGroups}
Let $T$ be a minimal homeomorphism of the Cantor set $X$, $\mathcal{F}(\langle T \rangle)$ its full group, $\mathcal{A}(\langle T \rangle)$ the commutator subgroup of its full group, $G$ the full group of a Bratteli diagram associated with $T$, and $G'$ the commutator subgroup of $G$. Then the following pairs of groups are compatible:  
\begin{equation}\label{eqnListCompatibleGroups}
(G', G), \quad (G', \mathcal{A}(\langle T \rangle)), \quad (\mathcal{A}(\langle T \rangle), \mathcal{F}(\langle T \rangle)).
\end{equation}
\end{proposition}

\begin{proof}
First, consider the pair $(G', G)$. Fix $g\in G\setminus \{e\}$. Let $\{\gamma_n\}$ be an arbitrary sequence in $G'$ such that $\gamma_n  g \gamma_n \neq \gamma_m  g \gamma_m$ whenever $n\neq m$. Notice that 
\[
(\gamma_n g \gamma_n^{-1})(\gamma_m g^{-1} \gamma_m^{-1}) = \gamma_n [g, \gamma_n^{-1} \gamma_m] \gamma_n^{-1} \in G'.
\]
It follows that the pair $(G', G)$ is compatible. A similar argument applies to $(\mathcal{A}(\langle T \rangle), \mathcal{F}(\langle T \rangle))$.

Now consider the pair $(G', \mathcal{A}(\langle T \rangle))$. Fix $g \in \mathcal{A}(\langle T \rangle)$ with $g \neq e$. By Proposition~\ref{PropStructureDerivedFullGroup}, there exist elements $p \in G'$ and $r \in \mathcal{A}(\langle T \rangle)$ such that $g = pr$ and $\supp(p) \setminus \supp(r) \neq \emptyset$. Choose a sequence $\{\gamma_n\}$ of elements in $G'$, each supported on $\supp(p) \setminus \supp(r)$, such that $g_n \neq g_m$ whenever $n \neq m$, where $g_n = \gamma_n g \gamma_n^{-1}$. Then
\[
(\gamma_n g \gamma_n^{-1})(\gamma_m g^{-1} \gamma_m^{-1}) = \gamma_n p \gamma_n^{-1} \, \gamma_m p^{-1} \gamma_m^{-1} \in G',
\]
which completes the proof.
\end{proof}

In the following result we classify indecomposable characters of $\mathcal A(\langle T\rangle )$.

\begin{theorem}\label{thmClassificationCharactersLargeCommutatorSubgroup}  Let $(X,T)$ be a Cantor minimal homeomorphism. Let $\chi$ be a non-regular  indecomposable character of $\mathcal A(\langle T\rangle )$. Then there exists a finite number of $T$-invariant ergodic measures $\{\mu_1,\ldots,\mu_n\}$ and non-negative integers $\{\alpha_1,\ldots,\alpha_n\}$ such that  
\[
\chi(g) = \prod_{i=1}^n \mu_i(\fix(g))^{\alpha_i}\mbox{ for every }g\in \mathcal A(\langle T \rangle).
\]
\end{theorem}
\begin{proof}  Fix an indecomposable character $\chi$ and the associated GNS representation $(\mathcal H, \pi, \xi)$.  Fix points $x,y\in X$ from distinct $T$-orbits.  To simplify the notation, we will  denote  the commutator subgroup $\mathcal A(\langle T\rangle)$  by $\Gamma$ and the commutator subgroups of associated Bratteli diagrams with base points $\{x\}$ and $\{y\}$ by $\Gamma_x$ and $\Gamma_y$, respectively.  Recall that the groups $\Gamma_x$ and $\Gamma_y$ are isomorphic and satisfy the assumptions of Theorem \ref{thmClassificationCharactersCommutatorSubgroup} and Corollary \ref{corollaryCommutatorAutomaticCont}. 

By Proposition \ref{PropCompatibleGroups}, the pairs $(\Gamma_x,\Gamma)$ and $(\Gamma_y,\Gamma)$ are compatible. Therefore, by Proposition \ref{PropRegularReprCompatibleGroups}, the central decompositions of  $\pi|_{\Gamma_x}$ and $\pi|_{\Gamma_y}$ do not contain subrepresentations that are quasi-equivalent to the regular representation.  Therefore, by Corollary \ref{corollaryCommutatorAutomaticCont}, the maps $\pi : \Gamma_x\rightarrow U(\mathcal H)$ and $\pi : \Gamma_y\rightarrow U(\mathcal H)$ are continuous with respect to the metric $D$ defined in Equation (\ref{EqUniformMetric}) and the strong operator topology on $U(\mathcal H)$.   

Consider the restriction of $\pi$ and $\chi$ onto $\Gamma_x$. Applying Theorem \ref{TheoremCentralDecomposition}, we can find a standard probability measure space $(\Omega,\nu)$ and  a measurable field $\omega \mapsto \chi_\omega$ of indecomposable characters on $\Gamma_x$ such that 
\begin{equation}\label{eqnClassCharLargeCommAux1}\chi(\gamma)  = \int_\Omega \chi_\omega(\gamma)\, d\nu(\omega)\mbox{ for every }\gamma\in \Gamma_x.\end{equation} 
  Note that by the classification of characters established in  Theorem \ref{thmClassificationCharactersCommutatorSubgroup}, each character $\chi_\omega$ naturally extends to a character on $\Gamma$. 

Let $g \in \Gamma$. By Proposition~\ref{PropStructureDerivedFullGroup}, there exist 
sequences $\{p_n\} \subset \Gamma_x$ and $\{r_n\} \subset \Gamma_y$ such that
$    g = p_n r_n$, for all $n \ge 1$,
where $\{r_n\}$ converges to the group identity and 
$\{p_n\}$ is a Cauchy sequence with respect to the metric $D$.

Note that the sequence of operators $\{\pi(r_n)\}$ converges to the identity operator in the strong operator topology. Therefore, 
 $$\lim_{n\to\infty} \chi(p_n)  = \lim_{n\to\infty} (\pi(p_n)\xi,\xi) =  \lim\limits_{n\to\infty} (\pi(gr_n^{-1})\xi,\xi)  = (\pi(g)\xi,\xi) = \chi(g). $$ 
 
 Since  the sequence $\{\chi_\omega(p_n)\}$ converges to $\chi_\omega(g)$ for every $\omega \in \Omega$, 
applying Lebesgue’s dominated convergence theorem to (\ref{eqnClassCharLargeCommAux1}) yields  
  $$\chi(g) = \lim_{n\to\infty} \chi(p_n) = \lim_{n\to\infty}  \int_\Omega \chi_\omega(p_n)d\nu(\omega) = \int_\Omega \chi_\omega(g) d\nu(\omega).$$
  
  Hence, we have shown that the character $\chi$, which is an extreme point in the (Choquet) simplex of characters of the group $\Gamma$, can be represented by two integrals with representing measures $\delta_\chi$ and $\nu$. By the uniqueness of representing measures in Choquet theory, it follows that $\chi = \chi_\omega$ for some $\omega$. Applying the classification of characters $\{\chi_\omega\}$ established in Theorem \ref{thmClassificationCharactersCommutatorSubgroup}, we complete the proof.
   \end{proof}

The proof of the following result is analogous to that of Corollary \ref{corollaryCommutatorAutomaticCont}. We leave the details to the reader. 

\begin{corollary}\label{corollaryLargeCommutatorAutomaticCont}
Let $(X,T)$ be a Cantor minimal system, and let $\pi$ be a finite-type unitary representation of $\mathcal A(\langle T \rangle)$. Suppose that the \emph{central decomposition} of $\pi$ (Theorem~\ref{TheoremCentralDecomposition}) contains no regular subrepresentation of $\mathcal A(\langle T \rangle)$.

Then, for any Cauchy sequence $\{g_n\}$ in $\mathcal A(\langle T \rangle)$ with respect to the metric $D$ defined in \eqref{EqUniformMetric}, the sequence $\{\pi(g_n)\}$ converges in the strong operator topology.
\end{corollary}

Finally, we are ready to establish the main result of the paper.

%

\begin{theorem}\label{theoremClassificationCharacters} Let $(X,T)$ be a Cantor minimal system and $\Gamma$ be either the full group  $\mathcal F(\langle T \rangle)$ or the full group of an associated Bratteli diagram. If $\chi$ is a non-identity indecomposable character of $\Gamma$, then either $\chi$ is the regular character or $\chi$ is of the form 
\begin{equation}\label{eqnMainClassification}\chi(\gamma)=\rho(\gamma) \mu_1(\fix(\gamma))
\cdots \mu_k(Fix(\gamma)) \mbox{ for }\gamma\in \Gamma, \end{equation} where $\mu_1,\ldots,\mu_k$ are $T$-invariant ergodic measures, repetitions permitted, and  $\rho: \Gamma \rightarrow S^1=\{z\in\mathbb C:|z|=1\}$ is a group homomorphism.

Additionally, the characters of the form (\ref{eqnMainClassification}) are indecomposable. 
\end{theorem}
\begin{proof}  We first note that the indecomposability of characters defined in (\ref{eqnMainClassification}) follows from  Proposition \ref{propProductMeasuresIsCharacter} and Remark \ref{remarkCharacterTimesHomomorsphism}. 

Let  $\chi$ be a non-regular, non-identity indecomposable character. Fix the associated GNS representation $(\mathcal H, \pi,\xi)$ of $\Gamma$. Note that $\pi$ is not quasi-equivalent to the regular representation of $\Gamma$.  

By Propositions \ref{PropRegularReprCompatibleGroups} and \ref{PropCompatibleGroups}, the central decomposition of the representation $\pi|_{\Gamma'}$ contains no regular sub-representations.  It follows from Corollaries \ref{corollaryCommutatorAutomaticCont} and 
\ref{corollaryLargeCommutatorAutomaticCont} that the restriction of the representation $\pi$ onto the commutator subgroup $\Gamma'$ is automatically continuous with respect to the strong operator topology and the topology of uniform convergence defined by the metric $D$ in (\ref{EqUniformMetric}).

By Propositions \ref{PropCosetFullGroupSmallSupport} and \ref{PropositionContinuityMeasures}, every element of $\Gamma$ can be approximated by elements from $\Gamma'$ with respect to the metric $D$.  It follows from Corollaries \ref{corollaryCommutatorAutomaticCont} and  \ref{corollaryLargeCommutatorAutomaticCont} that the representation $\pi|_{\Gamma'}$ can be extended by continuity to a unique representation $\pi_0$ of $\Gamma$.  Note that, in general, $\pi_0\neq \pi$.  In what follows we show that $\gamma \mapsto \pi(\gamma)\pi_0(\gamma^{-1})$ defines a homomorphism into $S^1$.   

Fix $x_0 \in X$, and a sequence of clopen sets $\{C_n\}$ converging to $x_0$, but not containing $x_0$. Given $\gamma\in \Gamma$, using Proposition \ref{PropCosetFullGroupSmallSupport}, we can find a sequence of elements $\{\gamma_n'\}$ from $\Gamma'$ and  a sequence of elements $\{\gamma_n''\}$ from $\Gamma$ such that $\{\gamma_n\}$ converges to $\gamma$ with respect to the metric $D$, and for every $n \neq 1$, $\gamma = \gamma_n'\gamma_n'' $ and $\supp(\gamma_n)\subset C_n$. 

 Note that the sequence $\{\gamma_n''\}$ satisfies the assumptions of Theorem \ref{ThAutomaticContinuity} for $Y = \{x_0\}$. Therefore, every convergent subsequence of the sequence $\{\pi(\gamma_n'')\}$ converges to a scalar operator in the weak operator topology. Without loss of generality, we can assume that $\{\pi(\gamma_n'')\}$ converges to a scalar operator.

Since the strong and weak operator topologies coincide on the set of unitaries and  multiplication is jointly continuous on bounded sets \cite[Proposition III.2.2.14]{Blackadar:BookOperatorAlgebras}, we obtain that 
\begin{equation}\label{eqnMainTheoremEqualityPi} \pi(\gamma) = \lim_{n\to\infty}\pi(\gamma_n')\pi(\gamma_n'') = \lambda(\gamma) \pi_0(\gamma)\mbox{ and }\chi(\gamma) = \lambda(\gamma) (\pi_0(\gamma)\xi,\xi),\end{equation}
where $\lambda(\gamma)\in S^1$ depends only on $\gamma$. Since $\pi$ and $\pi_0$ are representations of $\Gamma$, the map $\lambda : \Gamma\rightarrow S^1$ must be a group homomorphism.  

It follows from (\ref{eqnMainTheoremEqualityPi}) that the von Neumann algebras  generated by $\{\pi(\Gamma)\}$ and by $\{\pi_0(\Gamma)\}$ coincide.  Since the von Neumann algebras generated by $\{\pi(\Gamma')\}$ and $\{\pi_0(\Gamma')\}$ also coincide, we conclude that $\pi_{\Gamma'}$ is a factor representation. Therefore, the character $\chi|_{\Gamma'}$ is indecomposable. 

Applying the classification of characters  established in Theorems \ref{thmClassificationCharactersCommutatorSubgroup} and \ref{thmClassificationCharactersLargeCommutatorSubgroup}, we complete the proof.
\end{proof}

%
%
%

\bibliographystyle{amsalpha}

{\it \small Disclaimer: For the second-named author, the views expressed in this paper  are those of the author and do not necessarily represent  the views of the U.S. Naval Academy, the Department of the Navy, the Department of War, or the
U.S. Government.}

\end{document}